\documentclass[10pt,a4paper,reqno]{amsart} 

\usepackage{geometry}
\usepackage[T1]{fontenc}
\usepackage{lmodern}
\usepackage{fix-cm}
\usepackage[right]{eurosym} 
\usepackage[utf8]{inputenc} 
\usepackage[ngerman, english]{babel}

\usepackage[mathscr]{euscript}
\usepackage{amsmath}
\usepackage{wasysym}
\usepackage{amssymb}
\usepackage{bbm}
\usepackage{amscd}
\usepackage{esint}

\usepackage{graphicx}
\usepackage{float} 

\usepackage{color}
\usepackage{enumerate} 
\usepackage{hyperref} 
\usepackage{verbatim} 

\numberwithin{equation}{section}  
\usepackage[capitalize,nameinlink,noabbrev]{cleveref} 
\usepackage{tikz-cd}
\usetikzlibrary{quotes,babel,angles}

\theoremstyle{plain}
\newtheorem{theorem}{Theorem}[section]

\newtheorem{Def}[theorem]{Definition}

\newtheorem{lemma}[theorem]{Lemma}

\newtheorem{lem}[theorem]{Lemma}

\newtheorem{proposition}[theorem]{Proposition}

\newtheorem{kor}[theorem]{Corollary}

\newtheorem{cor}[theorem]{Corollary}

\theoremstyle{definition}

\newtheorem{rem}[theorem]{Remark}

\newcommand{\CB}{\mathcal{B}}
\newcommand{\CC}{\mathcal{C}}

\newcommand{\CE}{\mathcal{E}}

\newcommand{\CL}{\mathcal{L}}

\newcommand{\CN}{\mathcal{N}}

\newcommand{\CS}{\mathcal{S}}

\newcommand{\CX}{\mathcal{X}}

\newcommand{\BN}{\mathbb{N}}

\newcommand{\BR}{\mathbb{R}}

\newcommand{\SF}{\mathscr{F}}

\newcommand{\SH}{\mathscr{H}}

\newcommand{\SO}{\mathscr{O}}
\newcommand{\SP}{\mathscr{P}}

\newcommand{\ST}{\mathscr{T}}

\newcommand{\1}{\mathbbm{1}} 
 
\DeclareMathOperator*{\supp}{supp}

\DeclareMathOperator*{\ext}{ext}
\DeclareMathOperator*{\extn}{Ext}
\DeclareMathOperator*{\tr}{\tilde{\gamma}}
\DeclareMathOperator*{\trc}{\gamma}
\DeclareMathOperator*{\trn}{Tr}
\DeclareMathOperator*{\dive}{div}

\newcommand{\norm}[1]{\left\lVert#1\right\rVert}
\newcommand{\normtwo}[1]{{\left\vert\kern-0.25ex\left\vert\kern-0.25ex\left\vert #1 
		\right\vert\kern-0.25ex\right\vert\kern-0.25ex\right\vert}}
\newcommand{\id}{\mathrm{id}}

\newcommand{\di}{\,\text{d}}
\newcommand{\dishort}{\text{d}}

\newcommand{\abs}[1]{\ensuremath{\left\vert#1\right\vert}}

\DeclareMathOperator{\sgn}{sgn}
\newcommand{\lap}{\Delta}

\newcommand{\ssubset}{\subset \joinrel \subset} 
 
\DeclareMathOperator{\distw}{dist}
\newcommand{\dist}[2]{\distw(#1, #2)} 
 
\newcommand*{\diam}[1]{\text{diam}(#1)} 
\newcommand{\ie}{i.e.\,}
\newcommand{\aev}{a.e.\,}
\newcommand{\eg}{e.g.\,}
\newcommand{\etc}{etc.\,}

\usepackage{csquotes}

\usepackage{varwidth}

\usepackage{changepage}                 
\usepackage{lipsum}

\begin{document}
	\title{Robust nonlocal trace spaces and Neumann problems}
	
	\author{Florian Grube, Thorben Hensiek}
	\address{Fakultät für Mathematik, Universität Bielefeld, Postfach 10 01 31, 33501 Bielefeld, Germany}
	\email{fgrube@math.uni-bielefeld.de, thensiek@math.uni-bielefeld.de}
	
	\makeatletter
	\@namedef{subjclassname@2020}{%
		\textup{2020} Mathematics Subject Classification}
	\makeatother
	
	
	
	\begin{abstract} We prove trace and extension results for fractional Sobolev spaces of order $s\in(0,1)$. These spaces are used in the study of nonlocal Dirichlet and Neumann problems on bounded domains. The results are robust in the sense that the continuity of the trace and extension operators is uniform as $s$ approaches $1$ and our trace spaces converge to $H^{1/2}(\partial \Omega)$. We apply these results in order to study the convergence of solutions of nonlocal Neumann problems as the integro-differential operators localize to a symmetric, second order operator in divergence form.
	\end{abstract}
	
	\maketitle
\section{Introduction}
The study of trace and extension operators is motivated by the classical Dirichlet problem for the Laplacian, \ie
\begin{align}\label{eq:dirichlet}
	-\Delta u &=0 \text{ in }\Omega,\nonumber\\
	u &= g \text{ on }\partial\Omega
\end{align}	
for a function $g:\partial \Omega\to \BR$ and a sufficiently smooth domain $\Omega\subset \BR^d$. A classical question is under which assumptions on $g$ there exists a unique solution to the Dirichlet problem. The Poincaré inequality and Lax-Milgram lemma yield the existence and uniqueness of a weak solution $u\in H^1(\Omega)$ to the problem \eqref{eq:dirichlet} for $g\in H^1(\Omega)$, \ie 
\begin{align*}
	\int_{\Omega} \nabla u(x) \cdot \nabla v(x) \di x = 0
\end{align*}
for all $v \in H^1_0(\Omega):= \overline{C_c^\infty(\Omega)}^{H^1(\Omega)}$ and $u -g \in H^1_0(\Omega)$. This notion of weak solution is motivated by the Green-Gauß formula. These standard tools require the function $g$ to be prescribed on the whole domain. 

In the classical works of Aronszajn \cite{trace_aronzajin}, Prodi \cite{trace_prodi} and Slobodeckij \cite{trace_Slobodeckij} the existence of a continuous trace operator
\begin{align*}
	\trc  : H^1(\Omega) \to L^2(\partial\Omega)
\end{align*}
satisfying $\trc (u)=u|_{\partial\Omega}$ for all $u \in C(\overline{\Omega})\cap H^1(\Omega)$ has been established. Additionally, the image of the trace operator has been characterized as the Sobolev-Slobodeckij space $H^{1/2}(\partial\Omega)$ and the existence of a continuous right inverse $\ext:H^{1/2}(\partial \Omega)\to H^1(\Omega)$, the classical extension operator, has been proven. Gagliardo extended this result in \cite{Trace_Galirado} to $W^{1,p}(\Omega) \to W^{1-1/p,p}(\partial\Omega)$, $p > 1$. The existence of an extension operator allows to prescribe the boundary datum $g \in H^{1/2}(\partial \Omega)$ in \eqref{eq:dirichlet}. The existence of the trace operator guaranties that this definition of a weak solution is a consistent generalization of classical solutions because the condition $u-v \in H^1_0(\Omega)$ is equivalent to $\trc (u-v) =0$.
\medskip

In recent years there has been an intense study of nonlocal operators. The most prominent example is the fractional Laplacian
\begin{align*}
	(-\Delta)^s u(x):= \kappa_{d,s}\text{p.v.}\int_{\BR^d}\frac{u(x)-u(y)}{\abs{x-y}^{d+2s}}\di x,
\end{align*}        
for $s \in (0,1)$ with
\begin{equation*}
	\kappa_{d,s}:= \frac{2^{2s}\, s\, \Gamma(\frac{d+2s}{2})}{\pi^{d/2}\, \Gamma(1-s)}.
\end{equation*}
Here $\Gamma$ denotes Euler's gamma function. The normalization constant $\kappa_{d,s}$ guaranties the Fourier representation
\begin{align*}
	\SF(	(-\Delta)^s u)(\xi) = \abs{\xi}^{2s} \SF(u)(\xi).
\end{align*} 
This representation implies the convergence $(-\Delta)^s \to -\Delta$ as $s \to 1-$. In this work the asymptotic behavior of $\kappa_{d,s} \asymp s(1-s)$, see \cref{prop:asymptotics_constant}, will be important. Dirichlet problems for nonlocal integro-differential operators have been studied extensively in the literature, see the survey \cite{Ros-Oton_Survey} by Ros-Oton. For Hilbert space approaches we refer for example to \cite{Felsinger2013} by Felsinger, Kassmann and Voigt, \cite{bucur_valdinoci} by Bucur and Valdinoci as well as \cite{Rutkowski_Dirichlet} by Rutkowski. The nonlocality of $(-\Delta)^s$ requires functions to be defined on $\BR^d$. Thus, Dirichlet problems for these kind of operators are typically formulated as complement value problems. As in the case for $-\Delta$ the notion of a weak solution to the Dirichlet problem 
\begin{align}\label{eq:dirichlet_nonlocal}
	(-\Delta)^s u&= 0 \text{ in }\Omega,\nonumber\\
	u&= g \text{ on }\Omega^c
\end{align}
is motivated by a nonlocal Green Gauß formula. For $u \in C^2_b(\BR^d)$ and $v \in C_b^1(\Omega)$ it holds
\begin{align}\label{eq:green_gauss_frac_lap}
		\int\limits_{\Omega} [(-\Delta)^s u(x)] v(x) \di x &= \frac{\kappa_{d,s}\,}{2} \iint\limits_{ \BR^d\times \BR^d\setminus \Omega^c\times \Omega^c} \frac{(u(x)-u(y))(v(x)-v(y))}{\abs{x-y}^{d+2s}} \di x \di y\nonumber\\
		&\qquad\qquad\qquad- \int\limits_{\Omega^c} \CN_{s}u(y) v(y) \di y.
\end{align}
Here $\CN_s$ is the nonlocal normal derivative with respect to $(-\Delta)^s$ and $\Omega$ defined via
\begin{equation*}
	\CN_{s}u(y)= \kappa_{d,s}\,\text{p.v.}\int\limits_{\Omega}\frac{u(y)-u(x)}{\abs{x-y}^{d+2s}}\di y,\, y\in \Omega^c.
\end{equation*}
This concept of a nonlocal normal derivative first appeared in the work of Dipierro, Ros-Oton and Valdinoci in \cite{Valdinoci_nonlocal_derivative}. A similar operator has been introduced in the earlier work of Du, Gunzburger, Lehoucq and Zhou in \cite{Du_nonlocal_normal} in the context of peridynamics. This formula yields the following fractional analogue of $H^1(\Omega)$. We define the bilinear form
\begin{equation*}
	[u,v]_{V^s(\Omega \, |\, \BR^d)}:= \frac{\kappa_{d,s}\,}{2}\iint\limits_{\BR^d\times \BR^d\setminus \Omega^c\times \Omega^c} \frac{(u(x)-u(y))(v(x)-v(y)) }{\abs{x-y}^{d+2s}}\di x \di y
\end{equation*}
for measurable functions $u,v : \BR^d \to \BR $ and the set
\begin{equation*}
	V^s(\Omega\,|\,\BR^d) := \{ u : \BR^d \to \BR \, | \, 	[u,u]_{V^s(\Omega \, |\, \BR^d)} <\infty \}.
\end{equation*}
Since $u \in 	V^s(\Omega\,|\,\BR^d) $ implies $ u \in L^2(\Omega)$ for bounded domains $\Omega$, we equip this space with the norm
\begin{align*}
	\norm{u}_{	V^s(\Omega\,|\,\BR^d) } := \big([u,u]_{V^s(\Omega \, |\, \BR^d)} + \norm{u}^2_{L^2(\Omega)}\big)^{1/2}.
\end{align*} 
Such forms for more general L\'{e}vy measures are also considered by Servadei and Valdinoci in \cite{energy_valdinoci_1} and \cite{energy_valdinoci_2}. For nonsymmetric kernels and related forms the Dirichlet problem has been studied in \cite{Felsinger2013}. Furthermore, such forms appeared in \cite{Valdinoci_nonlocal_derivative}. For $g \in V^s(\Omega\,|\, \BR^d)$ we call $u \in V^s(\Omega\,|\, \BR^d)$ a weak solution to the Dirichlet problem \eqref{eq:dirichlet_nonlocal}, if 
\begin{align*}
	[u,v]_{ V^s(\Omega\,|\, \BR^d)} =0
\end{align*}
for all $v \in  V_0^s(\Omega\,|\, \BR^d) :=  \{w \in  V^s(\Omega\,|\, \BR^d) \, | \, w =0 \text{ on }\Omega^c \}$ and $u-g \in V_0^s(\Omega\,|\, \BR^d)$. This definition requires the complement data to be defined on the whole space. It is a natural question to ask for which complement values $g:\Omega^c\to \BR$ the Dirichlet problem has a unique weak solution. As for \eqref{eq:dirichlet} this motivates the study of trace and extension operators. In contrast to the local case, where a specific construction for the trace operator is needed, here the trace operator is the restriction to $\Omega^c$. 

The asymptotics of fractional phenomena as they localize have gained considerable attention in recent years. In particular, the space of weak solutions to the fractional Dirichlet problem $V^s(\Omega\,|\,\BR^d)$ converges to $H^1(\Omega)$ as $s \to 1-$. Indeed, the following holds. Let $\Omega$ be a bounded Lipschitz domain and $u \in H^1(\BR^d)$. Then 
\begin{align*}
		\lim\limits_{s \to 1-}[u,u]_{V^s(\Omega \, |\, \BR^d)} = \int_{\Omega} \abs{\nabla u(x)}^2 \di x.
\end{align*}
For a proof we refer the reader to \cite[Corollary 2]{bourgain_compactness} by Bourgain, Brezis and Mironescu, \cite{ponce} by Ponce and \cite[Theorem 3.4, (3.5)]{kassmann_mosco} by Foghem, Kassmann and Voigt. Furthermore, the family of Dirichlet forms $[\cdot , \cdot ]_{V^s(\Omega\,|\,\BR^d)}$ converge to $[\cdot , \cdot ]_{H^1(\Omega)}$ in the Mosco sense. We refer the reader to \cite[Theorem 1.6]{kassmann_mosco} and Foghem \cite[Theorem 5.73]{FoghemGounoue2020}. In \cite[Theorem 7.1]{Marvin_Mosco} Kassmann and Weidner proved Mosco convergence of forms related to nonsymmetric kernels. Mosco convergence implies the convergence of the corresponding stochastic processes and semigroups, see \cite{mosco_original} by Mosco. \medskip

In light of the previous discussion it is natural to ask the following question.

\textbf{Question:} Do there exist Hilbert spaces $X_s$ of functions $g: \Omega^c \to \BR$ for $s\in(0,1)$ such that
\begin{enumerate}[(i)]
	\item{ there exist trace operators $\trn :V^s(\Omega \, |\, \BR^d) \to X_s$ which are continuous uniformly in the limit $s\to 1-$, }
	\item{ there exist extension operators $\extn : X_s \to V^s(\Omega \, |\, \BR^d)$ which are continuous uniformly as $s\to 1-$,}
	\item{ $X_s$ converges to the classical trace space $H^{1/2}(\partial \Omega)$ as $s\to 1$?}
\end{enumerate}
The main goal of this article is an answer to this question.
\subsection{Main results}
For the remainder of this paper we fix the dimension $d\in \BN$, $d\ge 2$ and only consider domains $\Omega\subset \BR^d$. 

Now we define a space of functions on $\Omega^c$ which answers the aforementioned question. For sufficiently regular $f,g:\Omega^c \to \BR$ we set 
\begin{align}\label{eq:L2-part}
	(f,g)_{L^2(\Omega^c, \tau_s)} :=\int\limits_{\Omega^c } f(x)g(x)\tau_s(x)  \di x \ \text{ with } \ \tau_s(x):=  \frac{1-s}{d_x^s\,(1+d_x)^{d+s}}
\end{align}
and the weighted $L^2$ norm $\norm{f}_{L^2(\Omega^c, \tau_s)}:= (f,f)_{L^2(\Omega^c, \tau_s)}^{1/2}$. Throughout this paper we use the notation $d_x := \dist{x}{\partial\Omega} := \inf\{\abs{x-z} | z \in \partial\Omega\}$. The weight $\tau_s$ captures the decay of the kernel of the fractional Laplacian $\kappa_{d,s}\abs{\cdot}^{-d-2s}$ at infinity. The term $(1-s)d_x^{-s}$ concentrates at the boundary as $s$ increases. In fact, it is responsible for the reduction of dimension of $\Omega^c$ to the boundary $\partial \Omega$. In \cref{lem:weak_convergence_measures} we prove that the measure $\tau_s(x)\di x$ converges weakly to the surface measure on $\partial \Omega$. Additionally, we introduce the bilinear form
\begin{align}\label{eq:seminorm-part}
	[f,g]_{\ST^s(\Omega^c|\Omega^c)}:= \int\limits_{\Omega^c}\int\limits_{\Omega^c} (f(x)-f(y))(g(x) -g(y)) k_s(x,y)\di x \di y,
\end{align}
and the seminorm $[f]_{\ST^s(\Omega^c|\Omega^c)}:= [f,f]_{\ST^s(\Omega^c|\Omega^c)}^{1/2}$ 
where 
\begin{equation}\label{eq:our_trace_kernel}
	k_s(x,y):= \frac{(1-s)^2}{d_x^s \,(1+d_x)^s  d_y^s\,(1+d_y)^s \big( \abs{x-y}+ d_x\,d_y+d_x+d_y \big)^d  } 
\end{equation}
for any $x,y\in \overline{\Omega}^{\,c}$. We call $k_s$ the interaction kernel on $\Omega^c$. It has the same decay properties in each variable as $\kappa_{d,s}\abs{\cdot}^{-d-2s}$ at infinity. Notice that functions in $V^s(\Omega\,|\, \BR^d)$ are merely integrable on $\Omega^c$ away from the boundary with this decay. The term 
\begin{equation*}
	\kappa_{d,s}\int\limits_{\Omega} \int\limits_{\Omega^c} \frac{(u(x)-u(y))^2}{\abs{x-y}^{d+2s}}\di x \di y
\end{equation*} 
in the $V^s$-norm requires functions to have some regularity close to the boundary $\partial \Omega$. This behavior is captured by the term $d_x^{-s}d_y^{-s}(\abs{x-y}+d_x+d_y+d_xd_y)^{-d}$ in the kernel $k_s$. Again the terms $(1-s)d_x^{-s}$ and $(1-s)d_y^{-s}$ are responsible for the dimension reduction $\Omega^c\to \partial \Omega$ as $s\to 1-$. 
\begin{Def}\label{def:our_trace_space}
	We define the Hilbert space 
	\begin{align*}
	\ST^{s}(\Omega^c) := \{g:\Omega^c\to \BR \text{ measurable } \mid \norm{g}_{\ST^s(\Omega^c)}<\infty  \}
	\end{align*}
	endowed with the norm 
	\begin{align*}
	\norm{g}_{\ST^s(\Omega^c)} := \big(\norm{g}_{L^2(\Omega^c, \tau_s)}^2 + [g]_{\ST^s(\Omega^c|\Omega^c)}^2  \big)^{1/2}.
	\end{align*}
\end{Def}
Now we state our main results.
	\begin{theorem}\label{th:trace_and_extension}
		Let $\Omega$ be a bounded $C^{1,1}$-domain and $s_\star\in(0,1)$.
		\begin{enumerate}
			\item{ There exists a continuous trace operator $\trn : V^s(\Omega \, | \, \BR^d) \to \ST^{s}(\Omega^c)$ and constant $C=C(d,\Omega, s_\star)>0$ such that for all $s \in (s_\star,1)$ and $u \in   V^s(\Omega \, | \, \BR^d) $
			\begin{align*}
				\norm{\trn u}_{\ST^{s}(\Omega^c)} \le C \norm{u}_{V^s(\Omega \, | \, \BR^d)}.
			\end{align*}}
			\item There exists a continuous extension operator $\extn : \ST^{s}(\Omega^c) \to V^s(\Omega \, | \, \BR^d)$ and a constant $C=C(d,\Omega)>0$ such that for all $s \in (0,1)$ and $g \in  \ST^{s}(\Omega^c) $
			\begin{align*}
				\norm{\extn g}_{V^s(\Omega \, | \, \BR^d)} \le C \norm{g}_{\ST^{s}(\Omega^c)}.
			\end{align*}
			\item $\trn$ is the left inverse of $\extn$, i.e. $ \trn \circ \extn  = \id$.
		\end{enumerate}
	\end{theorem}
\begin{rem}
	The domain assumptions in \cref{th:trace_and_extension} are due to robust Poisson kernel estimates, see \cref{th:chen99_poissonestimate}. In the proof of \cref{th:trace_and_extension} we also show the existence of a continuous trace $\trn: V^s(\Omega\,|\, \BR^d)\to L^2(\Omega^c, \tau_{s})$ under milder assumptions on the boundary of the domain $\Omega\subset \BR^d$. Here the main tool is \cref{prop:robust_trace_hardy_type}. More precisely, it holds for any bounded Lipschitz domain $\Omega\subset \BR^d$ and $s_\star\in (0,1)$ there exists a continuous trace operator $\trn: V^s(\Omega\,|\, \BR^d)\to L^2(\Omega^c, \tau_s)$ and a constant $C=C(d,\Omega,s_\star)>0$ such that for all $s\in (s_\star, 1)$ and $u\in V^s(\Omega\,|\, \BR^d)$ 
	\begin{equation*}
		\norm{\trn u}_{L^2(\Omega^c,\tau_s)}\le C\, \norm{u}_{V^s(\Omega\,|\, \BR^d)}. 
	\end{equation*}
\end{rem}
We define the operator $\extn$ via the Poisson extension operator $\SP_{s,\Omega}$ as in the work by Bogdan, Grzywny, Pietruska-Pałuba and Rutkowski, see\cite{Bogdan_trace}. Thus, for $ g\in \ST^s(\Omega^c)$ the function $\extn (g) \in V^s(\Omega\,|\,\BR^d)$ is the unique solution to the Dirichlet problem \eqref{eq:dirichlet_nonlocal}, see \cite[Theorem 5.5]{Bogdan_trace}. The robust continuity of the trace and extension yield robust estimates of solutions to Dirichlet and Neumann problems in terms of the Dirichlet data in $\ST^s(\Omega^c)$ respectively Neumann data in $\ST^s(\Omega^c)'$, see \cref{th:solution_neumann_nonlocal}. The robustness of these estimates is essential for the convergence of solutions as $s\to 1-$ in \cref{sec:neumann}.

The next theorem answers the question on the asymptotics for $s\to 1- $.	
	\begin{theorem}\label{th:convergence_pointwise}
		Let $\Omega\subset\BR^d$ be a bounded $C^{1,1}$-domain. If $g\in H^1(\Omega^c)$, then 
		\begin{align*}
			&\norm{g}_{L^2(\Omega^c, \tau_s)} \to \norm{\tr g}_{L^2(\partial \Omega)},	\\
			&[g,g]_{\ST^s(\Omega^c|\Omega^c)}\to [\tr g,\tr g]_{H^{1/2}(\partial \Omega)} 
		\end{align*}
		as $ s \to 1-$. In particular, $\norm{g}_{\ST^{s}(\Omega^c)} \to \norm{\tr g}_{H^{1/2}(\partial \Omega)}	$ as $s \to 1-$. Here $\tr :H^1(\Omega^c)\to H^{1/2}(\partial \Omega)$ is the classical trace operator, see \cref{prop:definition_E}. 
	\end{theorem}
The benefit of the space $\ST^s(\Omega^c)$ is that it is intrinsically defined and one can decide whether a function is in $\ST^s(\Omega^c)$ by simply calculating the integrals. This is particularly important for the study localization phenomena.

Our study of trace spaces allows for a detailed discussion of nonlocal Neumann problems and their asymptotics as the operator localizes. Recall the definition of a weak solution to the Neumann problem, motivated by the Green-Gauß formula. If $ g \in L^2(\partial\Omega)$ and $f \in L^2(\Omega)$ are given, then $u \in H^1(\Omega)$ is called a weak solution to   
\begin{align*}
-\Delta u &=f \text{ in } \Omega,\\
\partial_{n} u &= g \text{ on } \partial\Omega,
\end{align*} 
if 
\begin{align*}
\int_{\Omega} \nabla u(x) \cdot \nabla v(x) \di x = \int_{\Omega} f(x) v(x) \di x + \int_{\partial\Omega} g(x) \trc  v(x) \di \sigma (x)
\end{align*}
for all $ v \in H^1(\Omega)$. These problems are typically solved in $H^1_\perp(\Omega)= H^1(\Omega)\cap\{ u\in L^2(\Omega)\,|\, \int_\Omega u =0 \}$ using the Lax-Milgram lemma. Thus, this problem may be generalized easily to Neumann data from the dual of the trace space $H^{1/2}(\partial \Omega)$ and inhomogeneities from the dual of $H_\perp^1(\Omega)$. For $G\in H^{1/2}(\partial \Omega)'$ and $F \in H_\perp^1(\Omega)'$, we call $u\in H^1_\perp(\Omega)$ a weak solution to the Neumann problem with Neumann data $G$ and inhomogeneity $F$, if $\int_{\Omega}\nabla u \cdot \nabla v = F(v)+ G(\trc  \phi)$ for all $v\in H_\perp^1(\Omega)$. In sight of the nonlocal Green-Gauß formula \eqref{eq:green_gauss_frac_lap}, for $G_s\in \ST^s(\Omega^c)'$ and $F_s \in V^s_\perp(\Omega \,|\, \BR^d)'$ we say that $u\in V_\perp^s(\Omega\,|\, \BR^d)= V^s(\Omega\,|\, \BR^d)\cap \{ \int_\Omega u =0 \}$ is a weak solution to the nonlocal Neumann problem 
\begin{align*}
(-\Delta)^s u &=F_s \text{ in } \Omega,\\
\CN_{s} u &= G_s \text{ on } \Omega^c, 
\end{align*}
if
\begin{equation*}
[u,v]_{V^s(\Omega\,|\, \BR^d)}=  F_s(v)+ G_s(\trn v)
\end{equation*}
for all $v\in V_\perp^s(\Omega\,|\, \BR^d)$. The trace spaces $H^{1/2}(\partial \Omega)$ and $\ST^s(\Omega^c)$ appear naturally in the formulation of Neumann problems. An application of \cref{th:trace_and_extension} and \cref{th:convergence_pointwise} is the convergence of solutions of nonlocal Neumann problems for operators $\CL_{s}$, which are comparable to the fractional Laplacian, to a solution of a local Neumann problem, see \cref{th:convergence_neumann_mean_zero} and \cref{th:convergence_neumann_mean_zero_example}. Furthermore, given a solution to a local Neumann problem for a symmetric elliptic second order operator in divergence form, we prove that there exists a sequence of nonlocal Neumann problems, such that the solutions converge to each other, see \cref{th:convergence_neumann_mean_zero_contrary} and \cref{th:convergence_neumann_mean_zero_contrary_example}. We generalize the results from \cite{FoghemGounoue2020} and \cite{Kassmann_Foghem2022} by Foghem and Kassmann.   

\subsection{Related literature} Nonlocal trace spaces have first been studied by Dyda and Kassmann in their work \cite{kassmann_dyda_ext}. They introduced\footnote{Note that in some results of \cite{kassmann_dyda_ext} the double-integral in \eqref{eq:DyKa-error} erroneously is taken over $\Omega^c \times \Omega^c$. The corrected version was communicated to us by the authors.} for $ 1\le p < \infty$ the space $X^{s,p}(\Omega^c)$ of functions $f:\Omega^c \to \BR$ satisfying
\begin{align}\label{eq:DyKa-error}
	\int\limits_{ \Omega^c} \int\limits_{\Omega^{\text{ext}}_{\text{inr}(\Omega)}} \frac{\abs{f(x)-f(y)}^p}{(\abs{x-y} + d_x + d_y)^{d+sp}} \di x \di y < \infty \,,
\end{align}
where $\Omega^{\text{ext}}_{\text{inr}(\Omega)} = \{ y \in \Omega^c | \, d_y < \text{inr}(\Omega)\}$ and $\text{inr}(\Omega) = \sup \{ r > 0 | B_r(x) \subset \Omega \text{ for some } x \in \Omega\}$.
The space $X^{s,p}(\Omega^c)$, equipped with the norm
\begin{align*}
	\bigg( \int_{\Omega^c} \frac{\abs{f(x)}^p}{\abs{x-y}^{d+sp}} \di x + \int\limits_{ \Omega^c} \int\limits_{\Omega^{\text{ext}}_{\text{inr}(\Omega)}} \frac{\abs{f(x)-f(y)}^p}{(\abs{x-y} + d_x + d_y)^{d+sp}} \di x \di y  \bigg)^{1/p},
\end{align*}
is a Banach space. \cite[Theorem 3, Theorem 5, Theorem 8]{kassmann_dyda_ext} yield the existence of a continuous trace and extension operator $X^s(\Omega^c) \leftrightarrow V^s(\Omega\,|\,\BR^d) $. The extension is constructed via a Whitney decomposition of $\Omega$. Under their domain assumptions, \cite[Defintion 14, Definition 15]{kassmann_dyda_ext} for every cube inside $\Omega$ there exists a reflected cube outside. This allows to copy the values of a function defined on $\Omega^c$ inside $\Omega$, see \cite[p. 16]{kassmann_dyda_ext}. In contrast to our space $\ST^s(\Omega^c)$ the space $X^s(\Omega^c)$ does not converge to $H^{1/2}(\partial\Omega)$. But their extension result yields for $s=1$ a new extension theorem for classical Sobolev spaces.

Bogdan et al. proved in \cite{Bogdan_trace} trace and extension results for a variety of nonlocal Sobolev spaces based on unimodal Lévy measures $\nu$, \ie radial and almost decreasing, with some additional assumptions, see \cite[A1, A2]{Bogdan_trace}. In the case that $\nu$ is the L\'{e}vy measure $\nu = \kappa_{d,s} \abs{\cdot }^{-d-2s}$ of the fractional Laplacian their result reads as follows. They defined the Hilbert space $\CX^s(\Omega^c)$ as the space of all measurable functions $f:\Omega^c\to \BR$ such that $[f,f]_{\CX^s(\Omega^c)}<\infty$, where
\begin{equation*}
	[f,g]_{\CX^s(\Omega^c)}:= \iint\limits_{\Omega^c \times \Omega^c} (f(x)-f(y))(g(x)-g(y))\,k_s^{\star}(x,y)\di x \di y, 
\end{equation*}
see  \cite[(2.8)]{Bogdan_trace}. Here the interaction kernel $k_s^{\star}$ is defined via
\begin{align}\label{eq:bogdan_trace_kernel}
	k_s^{\star}(x,y)&:= \kappa_{d,s}\,\int\limits_{\Omega} \abs{y-z}^{-d-2s} P_{s,\Omega}(z,x) \di z,
\end{align}
where $P_{s,\Omega}$ is the Poisson kernel to $(-\lap)^s$ on $\Omega$. Additionally, they fixed some $x_0\in \Omega$ and defined for $f:\Omega^c\to \BR$ measurable the weighted $L^2$ norm
\begin{equation*}
	|f|_{\CX^s(\Omega^c)}:= \big(\int\limits_{\Omega^c} \abs{f(y)}^2 \, P_{s,\Omega}(x_0,y) \di y\big)^{1/2}.
\end{equation*}
By \cite[Lemma 4.6]{Bogdan_trace}, if $f\in \CX^s(\Omega^c)$, then $|f|_{\CX^s(\Omega^c)}<\infty$. The authors endowed the space $\CX^s(\Omega^c)$ with the norm 
\begin{equation*}
	\norm{f}^2_{\CX^s(\Omega^c)}:= |f|_{\CX^s(\Omega^c)}^2 + 	[f,f]_{\CX^s(\Omega^c)}
\end{equation*}
and the canonical inner product, see \cite[(4.3)]{Bogdan_trace}. Naturally, the Poisson extension operator for any $g\in \CX^s(\Omega^c)$ is defined by 
\begin{equation}\label{eq:poisson_extension}
	\SP_{s,\Omega}\,g(z)= \begin{cases}
		\int\limits_{\Omega^c} P_{s,\Omega}(z,x)\, g(z) \di x& z\in \Omega,\\
		g(z) & z\in \Omega^c.
	\end{cases}
\end{equation}
In \cite[Theorem 2.3]{Bogdan_trace} they proved for an open set $D\subset \BR^d$ such that $D^c$ satisfies volume density condition, see \cite[(VDC)]{Bogdan_trace}, and $\abs{\partial \Omega}=0$ the following.  If $g\in \CX^s(\Omega^c)$, then $\SP_{s,\Omega}\,g\in V^s(\Omega\,|\, \BR^d)$ and
\begin{equation}\label{eq:douglas_intro}
	[\SP_{s,\Omega}\,g,\SP_{s,\Omega}\,g]_{V^s(\Omega \, | \, \BR^d)}= [g,g]_{\CX^s(\Omega^c)}.
\end{equation}
Additionally, if $u\in V^s(\Omega\,|\, \BR^d)$, then $g= u|_{\Omega^c}\in \CX^s(\Omega^c)$ and 
\begin{equation*}
 	[u,u]_{V^s(\Omega \, | \, \BR^d)} \ge [g,g]_{\CX^s(\Omega^c)}.
\end{equation*}
The equality \eqref{eq:douglas_intro} can be understood nonlocal version of the classical Douglas identity in \cite{Douglas_identity_1931}. Notice that \cite[Theorem 2.3]{Bogdan_trace} does not include the continuity of the trace operator ($V^s(\Omega\,|\, \BR^d)\to \CX^s(\Omega^c)$) and extension operator ($\CX^s(\Omega^c)\to V^s(\Omega\, |\, \BR^d)$) as a map between normed spaces. Furthermore,  they proved estimates on the interaction kernel $k_s^{\star}$ in \cite[Theorem 2.6]{Bogdan_trace}. For $s\in(0,1)$, these estimates yield after a short calculation constants $c_s,C_s> 0$ such that $s(1-s)\,c_sk_s^{\star}(w,z)\le k_s(w,z)\le s(1-s)\, C_s k_s^{\star}(w,z)$ for any $w,z \in \Omega^c$ with $d_w\le \diam{\Omega}$ or $d_z\le \diam{\Omega}$. These estimates are not robust in the limit $s\to 1-$. Therefore, we prove new estimates on the interaction kernel $k_s^{\star}$ to retrieve robust bounds.
\begin{proposition}\label{th:equivalence}
	Let $\Omega\subset \BR^d$ be a bounded $C^{1,1}$-domain. The norms of the spaces $\CX^s(\Omega^c)$ and $\ST^s(\Omega^c)$ are equivalent, \ie there exists a constant $C=C(d,\Omega)>0$ such that for all $s \in (0,1)$  
	\begin{equation*}
		\frac{s}{C} \, \norm{f}_{\ST^s(\Omega^c)}\le \norm{f}_{\CX^s(\Omega^c)} \le \sqrt{s}C\,  \norm{f}_{\ST^s(\Omega^c)}.
	\end{equation*}
\end{proposition}
\cref{th:trace_and_extension} and \cref{th:equivalence} prove the trace and extension between $\CX^s(\Omega^c)$ and $V^s(\Omega\,|\, \BR^d)$ to be robustly continuous in the limit $s\to 1-$.

The same authors considered in \cite{bogdan_nonlinear} the nonlinear case for a variety of Lévy measures. In contrast to \cite{kassmann_dyda_ext}, where increments of the form $\abs{u(x)-u(y)}^p$ are studied, Bogdan et al. considered forms based on increments of type $(u(x)^{\langle p-1\rangle}-u(y)^{\langle p-1\rangle})(u(x)-u(y))$, where $x^{\langle \alpha \rangle} = \sgn(x)\abs{x}^\alpha $ is the french power. Similar to \cite{Bogdan_trace}, they proved an extension and trace result as well as a Douglas type identity
\begin{align*}
	&\iint\limits_{(\Omega^c\times \Omega^c)^c} (u(x)^{\langle p-1\rangle}-u(y)^{\langle p-1\rangle})(u(x)-u(y)) \nu(x-y) \di x \di y\\
	&\qquad\qquad\qquad=\iint\limits_{\Omega^c\times \Omega^c} (g(z)^{\langle p-1\rangle}-g(w)^{\langle p-1\rangle})(g(z)-g(w)) k^\star(z,w)\di z \di w,
\end{align*}
where $k^\star$ is the interaction defined analogous to \eqref{eq:bogdan_trace_kernel}. They compared their results to the $p$-increment case in \cite[Section 6]{bogdan_nonlinear}.

In \cite{zoran_neumann}, Vondra\v{c}ek constructed reflected jump Markov processes related to nonlocal Neumann problems. The author considered the Sobolev-Slobodeckij-type space $V_\nu(\Omega\,|\, \BR^d)\cap L^2(\BR^d, m(x) \dishort x)$, where $m(x)=\1_{\Omega}(x)+ \1_{\Omega^c}(x)\mu(x)$ and $\mu(x)= \int_{\Omega} \nu(x-y)\di y$. Here $V_{\nu}(\Omega\,|\, \BR^d)$ is defined similar to $V^s(\Omega\,|\, \BR^d)$ based on a unimodal Lévy measure $\nu$ with obvious modifications. In the case of the fractional Laplacian, this space is smaller than the typical energy space $V^s(\Omega\,|\, \BR^d)$. In \cite[Lemma 2.2]{zoran_neumann}, the author proved that $L^2(\Omega^c, \mu(x)\dishort x)$ is the trace space and that zero extensions define functions in $V_\nu(\Omega\,|\, \BR^d)\cap L^2(\BR^d, m(x) \dishort x)$. For further discussions we also refer the reader to \cite[Remark 2.37]{Kassmann_Foghem2022}.

In \cite{Kassmann_Foghem2022}, Foghem and Kassmann introduced several possible choices of weighted $L^2-$spaces $L^2(\Omega^c, \tilde{\nu})$ with weights $\tilde{\nu}$ based on a symmetric Lévy measure $\nu$ and proved the continuity of the trace map $\trn: V_{\nu}(\Omega\,|\, \BR^d)\to L^2(\Omega^c, \tilde{\nu})$, see \cite[Proposition 2.34]{Kassmann_Foghem2022}. They defined the general trace space 
\begin{equation}
	T_{\nu}(\Omega^c):= \{ g:\Omega^c\to \BR \text{ measurable }\,|\, \exists u\in V_{\nu}(\Omega\,|\, \BR^d): v=u|_{\Omega^c}  \}
\end{equation}
equipped with its natural norm, see \cite[Definition 2.29]{Kassmann_Foghem2022}. Again $V_{\nu}(\Omega\,|\, \BR^d)$ is defined similar to $V^s(\Omega\,|\, \BR^d)$ based on $\nu$. The authors introduced several equivalent norms on $T_{\nu}$ which combine their weighted $L^2$-norms on $L^2(\Omega^c, \tilde{\nu})$ and the seminorms from \cite{Bogdan_trace} or \cite{kassmann_dyda_ext}, see \cite[Proposition 2.31]{Kassmann_Foghem2022}. We also refer to the discussion in \cref{sec:abstract_trace_space}. 

Recently, Frerick, Vollmann and Vu considered several approaches for nonlocal trace spaces in \cite{vu_trier}. In \cite[Theorem 5.2]{vu_trier} they generalized the result \cite[Lemma 2.2]{zoran_neumann} to more general kernels. Based on a kernel $j:\BR^d\times \BR^d \to [0,\infty]$ they define the weight 
\begin{equation*}
	w(y):= \int\limits_{\Omega} \frac{j(y,x)}{\int_{\Gamma} j(z,x)\di z +c }, \quad y\in \Gamma,
\end{equation*}
where $c\ge 0$ and $\Gamma:= \{ y\in \Omega^c \,|\, \int_{\Omega} j(y,x)\di x>0 \}$. Under additional assumptions on $j$ and $c$ they prove a continuous trace embedding $V_k(\Omega\,|\, \BR^d)\to L^2(\Gamma\,|\, w(y)\dishort y), f\mapsto f|_{\Gamma}$. In the case of the fractional Laplacian $j(y,x)= (1-s) \abs{x-y}^{-d-2s}$ and $c=0$ this weight behaves like
\begin{equation*}
	w_s(y)\asymp \int_{\Omega} \frac{d_x^{2s} }{\abs{x-y}^{d+2s}}\di x \asymp \begin{cases}
		-\ln(d_y) &\,\text{if } d_y<1/2,\\
		(1+d_y)^{-d-2s}&\,\text{if } d_y\ge 1/2.
		\end{cases}
\end{equation*}
Thereby, $L^2(\Omega^c, \, w_s(y)\dishort y)$ does not collapse to $L^2(\partial \Omega)$ as $s\to 1-$. In \cite[Theorem 5.4]{vu_trier} they introduced the additional seminorm 
\begin{equation*}
	\int\limits_{ \Gamma} \int\limits_{ \Gamma} (u(y)-u(z))^2 \int_{\Omega} \frac{j(y,x)j(z,x)}{\int_{\Gamma} j(s,x) \di s +c } \di x \di y \di z
\end{equation*}
on the space $L^2(\Gamma, w(y)\dishort y)$ for some fixed $c\ge 0$. The kernel is similar to the kernel introduced in \cite[Equation (6)]{Abatangelo_Neumann} for the fractional Laplacian. In this case it behaves like
\begin{equation*}
	\begin{cases}
	\frac{1+\abs{\ln\big(\frac{d_x\wedge d_y}{\abs{x-y}}\big)}}{\abs{x-y}^{d+2s}} &\,\text{if } d_x \wedge d_y \le \abs{x-y},\\
		\big(d_x \wedge d_y\big)^{-d-2s}&\,\text{if } d_x, d_y \ge \abs{x-y}
	\end{cases}
\end{equation*}
for $x,y\in \Omega^c$ which has been proven in \cite[Proposition 2.1]{Ros_Oton_Neumann_regularity}. Under further assumptions, see \cite[Corollary 5.5, 5.6]{vu_trier}, they proved a continuous trace and extension result. It is unclear whether \cite[Corollary 5.5]{vu_trier} or \cite[Corollary 5.6]{vu_trier} is applicable to the case of the fractional Laplacian. 

Du, Tian, Wright and Yu studied in \cite{du_trace_extension} trace and extension result for nonlocal Dirichlet problems with finite range of interaction. Let $\Omega$ be a bounded, simply connected Lipschitz domain, set $\Omega_\delta:= \{ x\notin \Omega\,|\, \dist{x}{\Omega}<\delta \}$, $\hat{\Omega}:= \Omega\cup \Omega_\delta$. They considered kernels like $k_\delta^\beta(\abs{h}):= C_{d,p,\beta}\delta^{-d-p+\beta} \abs{h}^{-\beta}\1_{\abs{h}<\delta}$ for $\delta>0$ and $\beta \in [0,d+p)$. The constant $C_{d,p,\beta}$ normalizes the $p$-th moment of the kernel. The corresponding Sobolev-type function space $S_\delta^\beta(\hat{\Omega})$ consists of all functions $u\in L^p(\hat{\Omega})$ such that
\begin{equation*}
	\abs{u}_{S_\delta^\beta(\hat{\Omega})}:= \Big(\iint\limits_{\hat{\Omega}\times\hat{\Omega}} k_\delta^\beta(\abs{x-y}) \abs{u(x)-u(y)}^p \di y \di x\Big)^{1/p}
\end{equation*}
is finite. This space is a Banach space equipped with the norm $\norm{u}_{{S_\delta^\beta(\hat{\Omega})}}:= \Big( \norm{u}_{L^p(\hat{\Omega})}^p +  \abs{u}_{S_\delta^\beta(\hat{\Omega})}^p \Big)^{1/p}$. For $\beta\in (d,d+p)$ this space is equivalent to the classical Sobolev space $W^{(\beta-d)/p,p}(\hat{\Omega})$. The space $S_\delta^\beta(\hat{\Omega})$ converges to $W^{1,p}(\Omega)$ as $\delta\to 0+$. Since they consider kernels with range $\delta>0$, $\Omega_\delta$ is their \enquote*{nonlocal boundary} of the domain $\Omega$. They introduced the trace space $T_\delta^\beta(\Omega_\delta)$ as the space of all functions $L^p(\Omega_\delta)$ such that 
\begin{equation*}
	\abs{u}_{T_\delta^\beta(\Omega_\delta)}:= \Big( \delta^{\beta-2} \iint\limits_{\Omega_\delta\times\Omega_\delta} \frac{\abs{u(x)-u(y)}^p}{(\abs{x-y}\vee \delta)^{d+p-2}(\abs{x-y}\wedge \delta)^{\beta}} \di x \di y \Big)^{1/p}
\end{equation*}
is finite. The space $T_\delta^\beta(\Omega_\delta)$ is a Banach space equipped with the norm $\norm{g}_{T_\delta^\beta(\Omega_\delta)}:= \big( \norm{u}_{L^p(\Omega_\delta)}^p + \abs{u}_{T_\delta^\beta(\Omega_\delta)}^p \big)^{1/p}$. They proved the existence of a continuous trace operator $S_\delta^\beta(\hat{\Omega})\to T_\delta^\beta(\Omega_\delta)$ as well as a continuous extension operator $T_\delta^\beta(\Omega_\delta)\to S_\delta^\beta(\hat{\Omega})$ which are robust in the limit $\delta>0$. In \cite[Proposition 2.1]{du_trace_extension} they proved the convergence of the trace space $T_\delta^\beta(\Omega_\delta)\to W^{1-1/p,p}(\partial \Omega)$ in case that $\Omega$ is the half space. In contrast to our work, Du, Tian, Wright and Yu localized by reducing the horizon of the kernel $k_\delta^\beta$ while scaling it up. On the other hand, we are interested in kernels with infinite range of interaction where the localization is due to increasing the singularity of the kernel $\kappa_{d,s}\abs{\cdot}^{-d-2s}$, $s\to 1-$.

A detailed discussion of related literature on nonlocal Neumann problems can be found in \cref{sec:related_literature_neumann}.  

	\subsection{Outline}\label{sec:outline} In \cref{sec:prelims} we introduce notation used throughout this work and discuss function spaces and their basic properties. We prove the trace and extension result, \cref{th:trace_and_extension}, in \cref{sec:trace_ext} as well as the equivalence to the space $\CX^s(\Omega^c)$ introduced in \cite{Bogdan_trace}, \ie \cref{th:equivalence}. In \cref{sec:abstract_trace_space} we shortly discuss the abstract trace space. The asymptotics of the trace space $\ST^s(\Omega^c)$ as $s\to 1-$ are studied in \cref{sec:convergence} which includes the proof of \cref{th:convergence_pointwise}. In \cref{sec:hilbert_space_convergence} we prove that the spaces $\ST^{s}(\Omega^c)$ converge to $H^{1/2}(\partial \Omega)$ as $s\to 1-$ in the sense of Kuwae and Shioya introduced in their work \cite{kuwae_mosco}, see also \cref{sec:appendix_convergence}. This will be important in \cref{sec:neumann} where we study Neumann problems for operators comparable to the fractional Laplacian and the convergence of solutions as the operators localize, see \cref{th:convergence_neumann_mean_zero}, \cref{th:convergence_neumann_mean_zero_example}, \cref{th:convergence_neumann_mean_zero_contrary} and \cref{th:convergence_neumann_mean_zero_contrary_example}.	
	\subsection*{Acknowledgments} 
	Financial support by the German Research Foundation (GRK 2235
	- 282638148) is gratefully acknowledged. We like to thank Moritz Kassmann and Guy Fabrice Foghem Gounoue for very helpful discussions and Solveig Hepp and Soobin Cho for careful proofreading of the manuscript. 
	\section{Preliminaries}\label{sec:prelims}
	We introduce notation used throughout this paper. We write $a\wedge b = \min\{ a,b \}$ and $a \vee b = \max \{ a,b \}$ for real numbers $a,b\in \BR$. $\Omega \subseteq \mathbb{R}^d$ always denotes a domain, \ie $\Omega$ is open and connected. We will add conditions on $\Omega$ if we need them. For $x \in \mathbb{R}^d$ we define the distance to the boundary of $\Omega$ via $d_x:=\dist{x}{\partial\Omega} := \inf\{\abs{x-y} \mid y \in \partial\Omega\}$. Additionally, we define an $\varepsilon$-annulus around $\Omega$ by $\Omega_{\varepsilon}:= \{ x\in \Omega^c\,|\, \dist{x}{\Omega}<\varepsilon \}$ and denote the remainder of the complement by $\Omega^{\varepsilon}:= \Omega^c \setminus \Omega_{\varepsilon}$ for $\varepsilon>0$. Notice that $\Omega_{\varepsilon}$ is neither open nor closed and $\Omega^\varepsilon$ is closed. We call $\Omega$ a bounded Lipschitz (resp. $C^{1,1}$) domain, if $\Omega$ is bounded and for every point $z\in \partial \Omega$ there exists a ball $B_\varepsilon(z)$, a translation and rotation $T_z:\BR^d\to \BR^d$ as well as a Lipschitz continuous (resp. $C^{1,1}$) function $\phi_z:B_1^{(d-1)}(0)\to \BR$ such that $T_z\big( \Omega \cap B_\varepsilon(z) \big)= \{ (x',x_d)\in B_1(0)\,|\, \phi_z(x')> x_d \}$. Here $B_{1}^{(d-1)}(0)\subset \BR^{d-1}$ is the $(d-1)$-dimensional unit ball centered at the origin. We say a domain $\Omega$ satisfies uniform interior (resp. exterior) cone condition if there exists a height $h$ and an angle $\alpha$ such that for every $z \in \partial\Omega$ there exists a cone $\CC \subset \Omega$ (resp. $\CC \subset \overline{\Omega}^c$) with height $h$ and opening angle $\alpha$ satisfying $\overline{\CC} \cap \partial\Omega = \{z\}$. For a Lipschitz domain $\Omega$ we denote the outer normal vector at the boundary point $z\in \partial \Omega$ by $n_z$ whenever it exists. Recall that a bounded domain $\Omega$ is $C^{1,1}$ if and only if it satisfies uniform interior and exterior ball condition, \ie there exists a radius $\rho>0$ such that for every boundary point $z\in \partial \Omega$ there exist an interior ball $B_\rho\subset \Omega$ and an exterior ball $B^\star_{\rho}\subset \overline{\Omega}^c$ with radius $\rho$ satisfying $\overline{B_\rho}\cap \Omega^c = \{z\}=\overline{B^\star_\rho}\cap \overline{\Omega}$. Furthermore, $\SH$ is the $(d-1)$-dimensional Hausdorff measure on $\BR^d$. The Hausdorff measure is monotone and equals the standard surface measure on $(d-1)$-dimensional, compact, Lipschitz submanifolds, see \eg \cite[Chapter 3]{coarea}. The classical trace operator for $H^1(\Omega)$ functions will be denoted by \begin{equation*}
		\trc  : H^1(\Omega) \to H^{1/2}(\partial\Omega)
	\end{equation*} if $\Omega$ is sufficiently regular. We denote by $\sigma$ the surface measure on $\partial \Omega$. We write $\omega_{d-1}:= 2\pi^{d/2}/\Gamma(d/2)$ for the measure of the $(d-1)$-dimensional unit sphere. For two normed spaces $(X, \norm{\cdot}_X),(Y, \norm{\cdot}_{Y})$ and a continuous linear map $l:X\to Y$ we define the operator norm by $\norm{l}_{X\to Y}:= \sup_{\norm{x}_X\le1} \norm{l(x)}_Y$. We will use small case $c_1,c_2$ \etc as running constants and we will reset them in every proof. \medskip
	
	We introduce function spaces which we use throughout this paper and recall basic properties. We assume all functions in this work to be Borel measurable. If $\Omega$ is bounded, we define the closed subspace of $L^2(\Omega)$ of functions with mean zero by \begin{equation*}
		L_\perp^2(\Omega)= \{ u\in L^2(\Omega)\,|\, \int_{\Omega} u =0 \}.
	\end{equation*} The Sobolev space $H^1(\Omega)$ consists of all $L^2(\Omega)$ functions whose weak derivatives are square integrable. We endow this spaces with the canonical norm
	\begin{align*}
		\norm{u}_{H^1(\Omega)}^2 := \norm{u}_{L^2(\Omega)}^2 + [u,u]_{H^1(\Omega)},\quad
		[u,v]_{H^1(\Omega)}:= \int_{\Omega} \nabla u(x) \cdot \nabla v(x) \di x.
	\end{align*}
	For the inner product on $H^1(\Omega)$ we write 
	\begin{align*}
		(u,v)_{H^1(\Omega)} := (u,v)_{L^2(\Omega)} + [u,v]_{H^1(\Omega)},
	\end{align*}
	We also use the closed subspace of $H^1(\Omega)$ with mean zero
	\begin{align*}
		H_{\perp}^1(\Omega) &:= \big\{u \in H^1(\Omega) \mid  \int_{\Omega} u(x) \di x =0 \big\}
	\end{align*}
	whenever $\Omega$ is bounded. For $s \in (0,1)$ the Sobolev-Slobodeckij space $H^{s}(\Omega)$ is defined as the set of all functions in $L^2(\Omega)$ endowed with norm
	\begin{align*}
		\norm{v}_{H^{s}( \Omega)}^2 := \norm{v}_{L^2(\Omega)}^2 + [v,v]_{H^{s}(\Omega)},\quad
		[u,v]_{H^{s}(\Omega)}:= \int_{\Omega} \int_{\Omega} \frac{(u(x)-u(y))(v(x)-v(y))}{|x-y|^{d+2s}} \di x \di y.
	\end{align*}
	For the inner product on $H^{s}(\Omega)$ we write 
	\begin{align*}
		(u,v)_{H^{s}(\Omega)} := (u,v)_{L^2(\Omega)} + [u,v]_{H^{s}(\Omega)}.
	\end{align*}
	We denote by $H^{1/2}(\partial\Omega)$ the Sobolev Slobodeckij space endowed with the norm
	\begin{align*}
		\norm{v}_{H^{1/2}(\partial \Omega)}^2 &:= \norm{v}_{L^2(\partial \Omega)}^2 + [v,v]_{H^{1/2}(\partial\Omega)},\\
		[u,v]_{H^{1/2}(\partial\Omega)}&:= \int_{\partial\Omega} \int_{\partial\Omega} \frac{(u(x)-u(y))(v(x)-v(y))}{|x-y|^{d}} \sigma(\dishort x) \sigma(\dishort y).
	\end{align*}
	For the inner product on $H^{1/2}(\partial\Omega)$ we write 
	\begin{align*}
		(u,v)_{H^{1/2}(\partial\Omega)} := (u,v)_{L^2(\partial\Omega)} + [u,v]_{H^{1/2}(\partial\Omega)}.
	\end{align*}
	The following nonlocal function spaces play a key role in this work. For $s\in (0,1)$ we define the spaces
	\begin{align*}
		V^s(\Omega\,| \,\BR^d) &:=\{ u:\BR^d\to \BR \text{  measurable  } \mid [u,u]_{V^s(\Omega \, |\, \BR^d)}<\infty\},\\
		V_{ 0}^s(\Omega) &:= \{ u \in V^s(\Omega\,|\,\BR^d) \mid u=0 \ \text{\aev on  }\Omega^c \},\\
		V_{\perp}^{s}(\Omega\,|\, \BR^d)&:= \big\{ v\in V^s(\Omega\,|\, \BR^d)\,|\, \int\limits_{\Omega}v(x) \di x = 0 \big\}.
	\end{align*}
	We only consider $V_{\perp}^{s}(\Omega\,|\, \BR^d)$ if $\Omega$ is bounded. Here 
	\begin{equation*}
		[u,v]_{V^s(\Omega \, |\, \BR^d)}:= \frac{\kappa_{d,s}\,}{2}\iint\limits_{\BR^d\times \BR^d\setminus \Omega^c\times \Omega^c} \frac{(u(x)-u(y))(v(x)-v(y)) }{\abs{x-y}^{d+2s}}\di x \di y,
	\end{equation*}
	where
	\begin{equation*}
		\kappa_{d,s}:= \Big( \int_{\BR^d} \frac{1-\cos(x_1)}{\abs{x}^{d+2s}}\di x  \Big)^{-1}
	\end{equation*}
	is the normalization constant of the fractional Laplacian $(-\Delta)^s$. This bilinear form is strongly connected to the fractional Laplacian by the nonlocal Green-Gauß formula, see \cref{prop:greengauss}. We endow these spaces with the norm
	\begin{equation*}
		\norm{u}^2_{V^{s}(\Omega\,|\, \BR^d)}:= \norm{u}_{L^2(\Omega)}^2 + [u,u]_{V^s(\Omega \, |\, \BR^d)}
	\end{equation*}
	and with the inner product 
	\begin{align*}
		(u,v)_{V^s(\Omega\,|\,\BR^d)} := (u,v)_{L^2(\Omega)} + [u,v]_{V^s(\Omega \, |\, \BR^d)}.
	\end{align*}
	The following proposition is the constant $\kappa_{d,s}$ is calculated and estimated robust in $s$. This result is taken from the work \cite{bucur_valdinoci} by Bucur and Valdinoci.
	\begin{proposition}[{\cite[(2.15)]{bucur_valdinoci}}]\label{prop:asymptotics_constant}
		For $d\in \BN$, $d\ge 2$
		\begin{align}
			\kappa_{d,s}:= \frac{2^{2s}\, s\, \Gamma(\frac{d+2s}{2})}{\pi^{d/2}\, \Gamma(1-s)}.
		\end{align}
		In particular, there exists a constant $C=C(d)\ge 1$ such that 
		\begin{equation*}
			C^{-1} \le \frac{\kappa_{d,s}}{s(1-s)} \le C
		\end{equation*}
		for all $s\in(0,1)$.
	\end{proposition}
	\begin{proposition}\label{prop:Vs_separable_hilbertspaces}
		The function spaces $V^s(\Omega\,|\, \BR^d)$, $V_0^s(\Omega)$ as well as $V_\perp^s(\Omega\,|\, \BR^d)$, endowed with the inner product $(\cdot, \cdot)_{V^s(\Omega\,|\, \BR^d)}$, are separable Hilbert spaces.
	\end{proposition}
	\begin{proof}
		Since $V_0^s(\Omega)\subset V^s(\Omega\,|\,\BR^d)$ as well as $V_\perp^s(\Omega\,|\, \BR^d) \subset V^s(\Omega\,|\,\BR^d)$ are closed sub spaces the claim follows from \cite[Theorem 3.23]{FoghemGounoue2020}.
	\end{proof}
	We end this section with a basic property of the space $\ST^s(\Omega^c)$. 
	\begin{proposition}\label{lem:complete}
		The space $\ST^s(\Omega^c)$, endowed with the inner product
		\begin{equation*}
			(f,g)_{\ST^s(\Omega^c)}= (f,g)_{L^2(\Omega^c, \tau_s)} + [f,g]_{\ST^s(\Omega^c|\Omega^c)},
		\end{equation*}
	is a separable Hilbert space.
	\end{proposition}
	\begin{proof}
		Surely, $\norm{\cdot}_{\ST^s(\Omega^c)}$ is a norm on $\ST^s(\Omega^c)$ and $(\cdot, \cdot)_{\ST^s(\Omega^c)}$ is an inner product on $\ST^s(\Omega^c)$ satisfying $(g,g)_{\ST^s(\Omega^c)}= \norm{g}_{\ST^s(\Omega^c)}^2$. It remains to show that $\ST^s(\Omega^c)$ is complete and separable. Let $\{g_n \}$, $g_n\in \ST^s(\Omega^c)$ be a Cauchy sequence. Then $g_n \tau_s$ is a Cauchy sequence in $L^2(\Omega^c)$ and, thus, there exists a limit $h\in L^2(\Omega^c)$. We define $g:= h/\tau_s\in L^2(\Omega^c, \tau_s)$. By $L^2$-convergence, we find a subsequence $\{ n_k \}_k$ such that $g_{n_k}\to g$ \aev on $\Omega^c$. By Fatou's lemma, $[g_{n_k}-g, g_{n_k}-g]_{\ST^s(\Omega^c)}\le \liminf_{l}[g_{n_k}-g_{n_l}, g_{n_k}-g_{n_l}]_{\ST^s(\Omega^c)}\to 0$ as $k\to \infty$. Thereby, $g_{n_k}\to g\in \ST^s(\Omega^c)$. The same is true for the original sequence which can be seen by repeating the argument with an arbitrary subsequence. This proves the completeness. Separability follows, since $L^2(\Omega^c)\times L^2(\Omega^c\times \Omega^c)$ is a separable Hilbert space and the map $\Phi: \ST^s(\Omega^c)\to L^2(\Omega^c)\times L^2(\Omega^c\times \Omega^c)$, $\Phi(g)= \big( g(x)\tau_s(x)^{1/2}, (g(x)-g(y))(k_s(x,y))^{1/2} \big)$ is an isometric isomorphism. 
	\end{proof}	
	\section{Trace and Extension}\label{sec:trace_ext}
	In this section we prove the existence of a continuous trace operator $\trn: V^s(\Omega\,|\, \BR^d)\to \ST^s(\Omega^c)$ and continuous extension operator $\extn: \ST^s(\Omega^c)\to V^s(\Omega\,|\, \BR^d)$, see \cref{th:trace_and_extension}. We pay particular attention to $s$ dependence to gain robust estimates as $s\to 1-$. We will use the main result from the work \cite{Bogdan_trace}, summarized in the introduction. In particular, we will prove in \cref{sec:equiv} that the norm of our trace space $\ST^s(\Omega^c)$ is equivalent to the norm of $\CX^s(\Omega^c)$. 
	\begin{theorem}[{\cite[Theorem 2.3]{Bogdan_trace}}]\label{th:douglas}
		Let $\Omega\subset \BR^d$ be an open set such that $\Omega^c$ satisfies volume density condition, see \cite[(VDC)]{Bogdan_trace}, and $\abs{\partial \Omega}=0$.
		\begin{enumerate}
			\item[(i)]{ If $g\in \CX^s(\Omega^c)$, then $\SP_{s,\Omega}\,g\in V^s(\Omega\,|\, \BR^d)$ and $   [\SP_{s,\Omega}\,g,\SP_{s,\Omega}\,g]_{V^s(\Omega \, | \, \BR^d)}= [g,g]_{\CX^s(\Omega^c)}$.}
			\item[(ii)]{ If $u\in V^s(\Omega\,|\, \BR^d)$, then $g= u|_{\Omega^c}\in \CX^s(\Omega^c)$ and  $[u,u]_{V^s(\Omega \, | \, \BR^d)} \ge [g,g]_{\CX^s(\Omega^c)}$.}
		\end{enumerate}
	\end{theorem}
	Recall that $\SP_{s,\Omega}$ is the Poisson extension operator. The equality in \cref{th:douglas} (i) can be understood as a Douglas type identity. Notice that \cref{th:douglas} does not include the continuity of the trace operator ($V^s(\Omega\,|\, \BR^d)\to \CX^s(\Omega^c)$) and extension operator ($\CX^s(\Omega^c)\to V^s(\Omega\, |\, \BR^d)$) as a map between normed spaces. 
	\subsection{Equivalence results}\label{sec:equiv}
	In this section, we prove that the norm of $\ST^{s}(\Omega^c)$ is equivalent to the norm of $\CX^s(\Omega^c)$. In particular, we want to ensure that the bounds are robust in the limit $s\to 1-$. Throughout this section we assume $\Omega$ to be a bounded $C^{1,1}$ domain. Bogdan et al. proved estimates on the interaction kernel $k_s^{\star}$ in \cite[Theorem 2.6]{Bogdan_trace}. For $s\in(0,1)$, these estimates yield after a short calculation constants $c_s,C_s> 0$ such that $s(1-s)\,c_sk_s^{\star}(w,z)\le k_s(w,z)\le s(1-s)\, C_s k_s^{\star}(w,z)$ for any $w,z \in \Omega^c$ with $d_w\le \diam{\Omega}$ or $d_z\le \diam{\Omega}$. These estimates are not robust in the limit $s\to 1-$. 
	
	We will use the following estimates on the Poisson kernel $P_{s,\Omega}$ which have been proven by several authors, see \cite[Theorem 1.3]{Panki_Poisson_estimates} by Kang and Kim, \cite[Theorem 3.3, Theorem 3.4]{chensonggreen} by Chen and Song and \cite[Theorem 2.10]{Chen1999} by Chen. The following theorem requires $\Omega$ to be a bounded $C^{1,1}$ domain. The remaining statements in this work require the domain to have a $C^{1,1}$ boundary due to this result. 
	\begin{theorem}[{\cite[Theorem 2.10]{Chen1999}}]\label{th:chen99_poissonestimate}
		There exists a constant $C=C(d,\Omega)>1$ such that for any $z\in \Omega$ and $x\in \overline{\Omega}^{\,c}$ 
		\begin{equation*}
			C^{-1}\kappa_{d,s} \,\frac{d_z^{s}}{d_x^s\, (1+d_x)^s} \frac{1}{\abs{x-z}^d} \le P_{s,\Omega}(z,x) \le C \,\kappa_{d,s}\,  \frac{d_z^{s}}{d_x^s\, (1+d_x)^s} \frac{1}{\abs{x-z}^d} .
		\end{equation*}
	\end{theorem}
	The equivalence of the respective $L^2$ terms $\abs{\cdot}_{\CX^s(\Omega^c)}$ and $\norm{\cdot}_{L^2(\Omega^c, \tau_s)}$ follow directly from \cref{th:chen99_poissonestimate}.
	\begin{kor}\label{cor:L2partequivalence}
		There exists a constant $C=C(d,\Omega)>1$ such that 
		\begin{equation*}
			\sqrt{s}\,\frac{1}{C}\,\norm{f}_{L^2(\Omega^c, \tau_s)}\le |f|_{\CX^s(\Omega^c)}\le \sqrt{s}C\,  \norm{f}_{L^2(\Omega^c, \tau_s)}
		\end{equation*}
		for any $f\in \CX^s(\Omega^c)\cup \ST^s(\Omega^c)$.
	\end{kor}
	\begin{proof}
		The claim follows from \cref{th:chen99_poissonestimate} and \cref{prop:asymptotics_constant} with $C(d,\Omega):= \sup_{s\in(0,1)} \big(  \frac{c_1\,\kappa_{d,s}}{s(1-s)}\vee \frac{\kappa_{d,s}}{c_1\, s(1-s)}\big)^{1/2}$. Here $c_1=c_1(d,\Omega)>0$ is the constant from \cref{th:chen99_poissonestimate}.
	\end{proof}
	The next two technical lemmata will be used in the proof of \cref{prop:upperbound} and \cref{prop:lowerbound}.
	\begin{lemma}\label{lem:distance_upper}
		Let $D\subset \BR^d$ be an open set and $s>0$. For any $x\in \overline{D}^c$
		\begin{align*}
			\int\limits_{ D }\frac{1}{\abs{x-z}^{d+s}} \di z \le \frac{\omega_{d-1}}{s}\,\frac{1}{\dist{x}{D}^s}
		\end{align*}  
		holds. If $D$ is bounded, then there exists a constant $C=C(d,D)>0$ such that for all $x\in \overline{D}^c$
		\begin{equation*}
			\int\limits_{ D }\frac{1}{\abs{x-z}^{d+s}} \di z \le \frac{C}{s} \frac{1}{\dist{x}{D}^s\, (1+\dist{x}{D})^d}
		\end{equation*}
	\end{lemma}
	\begin{proof}
		Fix $x\in \overline{D}^c$. We use $D\subset B_{\dist{x}{D}}(x)^c$ and apply polar coordinates.
		\begin{equation*}
			\int\limits_{ D }\frac{1}{\abs{x-z}^{d+s}} \di z\le \int\limits_{ B_{\dist{x}{D}}(x)^c}\frac{1}{\abs{x-z}^{d+s}} \di z= \omega_{d-1} \int\limits_{\dist{x}{D}}^\infty t^{-1-s} \di t = \omega_{d-1} \frac{\dist{x}{D}^{-s}}{s}. 
		\end{equation*}
		In case that $\dist{x}{D}<1$, the second claim for bounded $D$ is a direct consequence of the first statement. If $D$ is bounded and $\dist{x}{D}\ge 1$, then 
		\begin{equation*}
			\int\limits_{ D }\frac{1}{\abs{x-z}^{d+s}} \di z \le \abs{D} \frac{1}{\dist{x}{D}^{d+s}}\le \abs{D}2^{d} \frac{1}{\dist{x}{D}^{s}(1+\dist{x}{D})^d}.
		\end{equation*}
	\end{proof}
	\begin{lemma}\label{lem:distance_lower}
		Let $D\subset \BR^d$ be an open set satisfying uniform interior cone condition with a compact boundary and $s\in(0,1)$. Then there exists a constant $C=C(d,D)>0$ such that 
		\begin{equation*}
			\frac{1}{\dist{x}{D}^s(1+\dist{x}{D})^d}  \le C\,  \int\limits_{ D } \frac{1}{\abs{x-z}^{d+s}} \di z
		\end{equation*}
		for all $x\in \overline{D}^c$.
	\end{lemma}
	\begin{proof}
		Fix any $x\in \overline{D}^c$ and let $x_0\in \partial D$ be a minimizer of the distance of $x$ to $D$. Since $D$ satisfies uniform interior cone condition, we find an interior cone $\CC$ with apex at $x_0$ those opening angle and volume do not depend on $x_0$. We call $h(\CC)$ the height of the cone and define $h:= h(\CC)/2$. There exists a constant $c_1=c_1(d,\Omega)$, which does not depend on $x_0$, such that $\SH\big( \CC \cap B_t(x_0) \big)\ge c_1 t^{d-1}$ for all $0<t<h(\CC)/2=h$, see \eg \cite[Lemma A.4 (A.19)]{RosOton_Vladinoci_appendix}. Recall that $\SH$ is the $(d-1)$-dimensional Hausdorff measure. Therefore, the coarea formula applied to $z \mapsto \abs{x-x_0}+ \abs{x_0-z}$, see \cref{th:coarea}, yields
		\begin{align*}
			&\int\limits_D \frac{1}{\abs{x-z}^{d+s}} \di z \ge \int\limits_\CC \frac{1}{(\abs{x-x_0}+\abs{x_0-z})^{d+s}} \di z=\int\limits_\CC \frac{1}{(\abs{x-x_0}+\abs{x_0-z})^{d+s}} \abs{\nabla_z \abs{x_0-z}}\di z\\
			&\qquad\ge  \int\limits_{\abs{x-x_0}}^{\abs{x-x_0}+ h}  t^{-d-s} \SH\big(\{ z\in \CC\,|\, \abs{z-x_0}= t- \abs{x-x_0} \}\big) \di t
			\ge c_1\, \int\limits_{0}^{h} \frac{t^{d-1}}{(t+\abs{x-x_0})^{d+s}} \di t.
		\end{align*}
		If $\dist{x}{D}=\abs{x-x_0}<h$, then we estimate
		\begin{align*}
			\int\limits_D \frac{1}{\abs{x-z}^{d+s}} \di z &\ge c_1\, \int\limits_{0}^{\abs{x-x_0}} \frac{t^{d-1}}{(2\,\abs{x-x_0})^{d+s}} \di t= \frac{c_1}{d\, 2^{d+s}}\, \abs{x-x_0}^{-s} \ge \frac{c_1}{d\, 2^{d+1}}\, \frac{1}{\abs{x-x_0}^{s}(1+\abs{x-x_0})^{d}}. 
		\end{align*}
		If $\dist{x}{D}=\abs{x-x_0}\ge h$, we define $R:= \sup\{ \abs{x_0-z} | z \in \CC \}$, which only depends on the height $h(\CC)$ and the opening angle of the $\CC$. This implies 
		\begin{equation*}
			\abs{x-z}\le \abs{x-x_0}+\abs{x_0-z}\le \abs{x-x_0}+R\le (1+R/h)\abs{x-x_0}.
		\end{equation*}
		Thereby,
		\begin{align*}
			\int\limits_D \frac{1}{\abs{x-z}^{d+s}} \di z &\ge \big( 1+R/h \big)^{-d-s} \frac{\abs{\CC}}{\abs{x-x_0}^{d+s}} \ge \frac{\abs{\CC}}{\big( 1+R/h \big)^{d+1}}\,\frac{1}{\abs{x-x_0}^s(1+\abs{x-x_0})^d}. 
		\end{align*}
		Set $C:= \max\{ d2^{d+1}/c_1, \abs{\CC}^{-1}(1+R/h)^{d+1} \}$. Note that this constant does not depend on $x_0$.
	\end{proof}
	The following two propositions compare the interaction kernels $k_s^{\star}$ and $k_s$.
	\begin{proposition}\label{prop:upperbound}
		There exists a constant $C=C(d,\Omega)>0$ such that 
		\begin{equation*}
			k_s^{\star}(x,y)\le  s\,C\,  k_s(x,y)
		\end{equation*}
		for any $x,y\in \overline{\Omega}^{\,c}$.
	\end{proposition}
	We follow some of the ideas Chen and Song used to derive \cite[Theorem 1.5]{chensonggreen} from \cite[Corolary 1.3]{chensonggreen}.
	\begin{proof}
		Let $\rho>0$ be the uniform interior and exterior ball radius of the bounded $C^{1,1}$-domain $\Omega$. Let $x,y\in \overline{\Omega}^c$. By \cref{th:chen99_poissonestimate} and \cref{prop:asymptotics_constant}, there exists a constant $c_1=c_1(d,\Omega)$ such that 
		\begin{equation*}
			k_s^{\star}(x,y)\le c_1\, s^2(1-s)^2\,\frac{1}{d_x^s (1+d_x)^s} \int\limits_{\Omega} \frac{d_z^s}{\abs{y-z}^{d+2s}} \frac{\di z}{\abs{x-z}^d} .
		\end{equation*}
		Therefore, it remains to show that there exists a constant $c_2=c_2(d,\Omega)>0$ such that \begin{equation}\label{eq:first_estimate_upper_bound}
			A:=s\int\limits_{\Omega} \frac{d_z^s}{\abs{y-z}^{d+2s}} \frac{\di z}{\abs{x-z}^d} \le c_2 \frac{1}{ d_y^s (1+ d_y)^s (\abs{x-y}+ d_x\,d_y+d_x+d_y)^d }.
		\end{equation} Notice that for any $z\in \Omega$ that $d_x \le |x-z|$ as well as $d_y \le |y-z|$ and
		\begin{equation*}
			(\abs{x-y}+ d_x\,d_y+d_x+d_y)^d\le 2^d\,  \big( (\abs{x-z}+d_x\,(1+d_y))^d +(\abs{y-z}+d_y\,(1+d_x))^d   \big)
		\end{equation*}
		This yields
		\begin{align*}
			(\abs{x-y}+ d_x\,d_y+d_x+d_y)^d &\le 2^d (|x-z|^d(2+d_y)^d + |y-z|^d(2+d_x)^d)\\
			&\le c_3(d,\Omega) \Big(|x-z|^d(1+d_y)^d + |y-z|^d(1+d_x)^d\Big).
		\end{align*}
		Thus, we estimate $A$ by
		\begin{align*}
			A\le sc_3 \frac{1}{(\abs{x-y}+ d_x\,d_y+d_x+d_y)^d}\, \Big( \underbrace{\int\limits_{\Omega} \frac{d_z^s\, (1+d_y)^d}{\abs{y-z}^{d+2s}}\di z}_{\text{(I)}}+ \underbrace{ \int\limits_{\Omega} \frac{d_z^s}{\abs{y-z}^{2s}} \frac{(1+d_x)^d}{\abs{x-z}^d}\di z}_{\text{(II)}}\Big) .
		\end{align*}
		It remains to show that there exists a constant $c_4=c_4(d,\Omega)$ such that $s(I), \, s(II)\le c_4\,d_y^{-s}(1+d_y)^{-s}$. First, a technical estimate which we use going forward. For any $z\in \Omega$, $d_z\le (1+\diam{\Omega})\abs{z-y}/(1+d_y)$. Therefore, 
		\begin{equation}\label{eq:d_z_bound_against_y}
			\frac{d_z}{\abs{z-y}}\le (1+\diam{\Omega})\frac{1}{1+d_y}= c_5(\Omega)\frac{1}{1+d_y}.
		\end{equation}
		We begin by proving the estimate $s\text{(I)}\le c_4 d_y^{-s}(1+d_y)^{-s}$. By \eqref{eq:d_z_bound_against_y} and \cref{lem:distance_upper}, 
		\begin{equation}\label{eq:upperbound_estimate_for_(I)}
			s\text{(I)}\le s c_5 \frac{1}{(1+d_y)^s}\int\limits_{\Omega} \frac{(1+d_y)^d}{\abs{y-z}^{d+s}}\di z\le c_5\, c_6 \, \frac{1}{d_y^s(1+d_y)^s}.
		\end{equation}
		Here $c_6=c_6(d,\Omega)$ is the constant from \cref{lem:distance_upper}. Thereby, the desired bound on $s$(I) is proven. We continue by estimating $s$(II). We distinguish between several cases, depending on the distances of $x,y$ to the boundary of $\Omega$. We define $\sigma = \tfrac{1}{2} \min\{ \rho, \diam{\Omega} \}$.
		
		\textbf{Case ($2\,d_x\ge d_y$): } We divide the integration domain $\Omega$ in (II) into two regions. For $z\in \Omega$ satisfying $\abs{x-z}\ge \abs{y-z}$, if $d_x\ge 1$, then
		\begin{equation*}
			\frac{1+d_x}{\abs{x-z}}\le \frac{1+d_x}{d_x}\le 2\le 2(\diam{\Omega}+1) \frac{1+d_y}{\abs{y-z}}
		\end{equation*} 
		and, if $d_x< 1$, then 
		\begin{equation*}
			\frac{1+d_x}{\abs{x-z}}\le \frac{2}{\abs{y-z}}\le 2 \frac{1+d_y}{\abs{y-z}}.
		\end{equation*}
		Therefore, $\frac{1+d_x}{\abs{x-z}} \le c_7(\Omega) \frac{1+d_y}{\abs{y-z}}$ for $z\in \Omega$ satisfying $\abs{x-z}\ge \abs{y-z}$. Thus by \eqref{eq:d_z_bound_against_y}, 
		\begin{align*}
			s\hspace{-30pt}\int\limits_{\{z\in \Omega \,|\, \abs{x-z}\ge \abs{y-z}  \}} \frac{d_z^s}{\abs{y-z}^{2s}} \frac{(1+d_x)^d}{\abs{x-z}^d}\di z &\le sc_7^d \int\limits_{\Omega} \frac{d_z^s}{\abs{y-z}^{2s}}\frac{(1+d_y)^d}{\abs{y-z}^{d}} \di z\\
			\le sc_5\,c_7^d \frac{1}{(1+d_y)^s}\int\limits_{\Omega} \frac{(1+d_y)^d}{\abs{y-z}^{d+s}} \di z &\le c_5\,c_6\,c_7^d\, \frac{1}{d_y^s\,(1+d_y)^s}. 
		\end{align*}
		Here we used again \cref{lem:distance_upper}. Now, we prove the bound for the remaining integration domain $\{z\in \Omega \,|\, \abs{x-z}< \abs{y-z}  \}$. By \eqref{eq:d_z_bound_against_y} and \cref{lem:distance_upper}, just as above,  
		\begin{align*}
			s\hspace{-30pt}\int\limits_{\{z\in \Omega \,|\, \abs{x-z}< \abs{y-z}  \}} \frac{d_z^s}{\abs{y-z}^{2s}} \frac{(1+d_x)^d}{\abs{x-z}^d}\di z &\le c_5 \frac{s}{(1+d_y)^s} \int\limits_{\{z\in \Omega \,|\,\abs{x-z}< \abs{y-z}  \}} \frac{(1+d_x)^d}{\abs{x-z}^{d+s}}\di z\\
			&\le c_5 c_6 \frac{1}{d_x^s\,(1+d_y)^s}\le 2\,c_5c_6 \frac{1}{d_y^s\,(1+d_y)^s},
		\end{align*}
		which yields the desired bound in the first case $2\,d_x\ge d_y$, \ie
		\begin{align*}
			s\text{(II)} \le c_4\, \frac{1}{d_y^s\,(1+d_y)^s}.
		\end{align*}
		\textbf{Case ($\sigma\le 2\,d_x<d_y$):} By \eqref{eq:d_z_bound_against_y},
		\begin{align*}
			s\text{(II)}= s\int\limits_{\Omega} \frac{d_z^s}{\abs{y-z}^{2s}} \frac{(1+d_x)^d}{\abs{x-z}^d}\di z&\le s\,c_5 \frac{1}{(1+d_y)^s}\int\limits_{\Omega} \frac{1}{\abs{y-z}^{s}} \frac{(1+d_x)^d}{\abs{x-z}^d}\di z\\
			\le s c_5 \frac{1}{d_y^s\,(1+d_y)^s}  \Big(\frac{1+d_x}{d_x}\Big)^d \,\abs{\Omega}&\le c_5 \Big(\frac{2+\sigma}{\sigma}\Big)^d \frac{1}{d_y^s\,(1+d_y)^s}
		\end{align*}	
		\textbf{Case ($2\,d_x<\sigma\le  d_y$):} 
		\begin{align*}
			s\text{(II)} &= s \int\limits_{\Omega} \frac{d_z^s}{\abs{y-z}^{2s}} \frac{(1+d_x)^d}{\abs{x-z}^d}\di z 
			\le s \frac{(1+\sigma)^{d+1}}{\sigma} \frac{1}{d_y^s\,(1+d_y)^s}\,	\int\limits_{\Omega}  \frac{1}{\abs{x-z}^{d-s}}\di z\\
			&\le s c_8(d,\sigma) \frac{1}{d_y^s\,(1+d_y)^s}\,	\int\limits_{B_{\sigma+\diam{\Omega}}(0)  } \abs{z}^{s-d} \di z= c_8 \omega_{d-1} (\sigma+\diam{\Omega})^s  \frac{1}{d_y^s\,(1+d_y)^s}\\
			&\le c_8 \omega_{d-1} (\sigma+\diam{\Omega}+1)  \frac{1}{d_y^s\,(1+d_y)^s}.
		\end{align*}
		\textbf{Case ($2\,d_x<d_y<\sigma$):} We split the integration domain $\Omega$ into \begin{align*}
			A_{x,y}:= \{z\in \Omega \,|\,\abs{x-z}\le 2\, \abs{y-z}  \}\,\text{ and }\,
			\Omega\setminus A_{x,y}.
		\end{align*} 
		On $\Omega\setminus A_{x,y}$, by \eqref{eq:d_z_bound_against_y} and \cref{lem:distance_upper},
		\begin{align*}
			s\int\limits_{\Omega\setminus A_{x,y}} \frac{d_z^s}{\abs{y-z}^{2s}} \frac{(1+d_x)^d}{\abs{x-z}^d} \di z &\le s c_5  \frac{(1+\sigma/2)^d}{2^d(1+d_y)^s} \int\limits_{\Omega\setminus A_{x,y}} \frac{1}{\abs{y-z}^{d+s}}  \di z\le c_5 c_6 \tfrac{(1+\sigma/2)^d}{2^d}\,\frac{1}{d_y^s \, (1+d_y)^s}.
		\end{align*} 
		For the remaining integration domain $A_{x,y}$ of (II) we first make some observations. Fix a point $\overline{x}\in \partial \Omega$ which minimizes the distance of $x$ to the boundary of $\Omega$, e.g. $d_x= \abs{x-\overline{x}}$. By the uniform exterior ball condition, there exists an exterior ball those closure intersects which $\partial \Omega$ only in $\overline{x}$, e.g. $B_{\rho}(x_0)$ where $x_0 = \overline{x}+\rho\, n_{\overline{x}}$ and $n_{\overline{x}}$ is the outer normal vector at $\overline{x}\in \partial \Omega$. We define $\overline{y}\in \Omega^c$ by 
		\begin{equation*}
			\overline{y}= \overline{x}+d_y\, n_{\overline{x}} = x+ (d_y-d_x)\, n_{\overline{x}}.
		\end{equation*}
		Notice that $\overline{x}, x, \overline{y}$ and $x_0$ are colinear. Additionally, $ x, \overline{y}\in B_\rho(x_0)$ and $d_{\overline{y}}= d_y$. For any $z\in \partial B_{\rho}(x_0)$ 
		\begin{equation} \label{eq:outerballcondition_too_estimate_II}
			\abs{x-z}\le \abs{x-\overline{y}} + \abs{\overline{y}-z}= d_y-d_x+\abs{\overline{y}-z}\le d_y + \abs{\overline{y}-z}\le 2\,\abs{\overline{y}-z}.
		\end{equation}
		By uniform exterior ball condition, the same estimate also holds for any $z\in \Omega$. Additionally, for any $z\in A_{x,y}$ 
		\begin{equation}\label{eq:z_in_A_x,y}
			\abs{\overline{y}-z} = \abs{x+(d_y-d_x)\, n_{\overline{x}  } -z } \le \abs{x-z} +d_y-d_x \le 2\, \abs{y-z}+d_y\le 3\, \abs{y-z}.
		\end{equation}
		Finally, $\abs{\overline{y}-z} \ge d_y\ge d_y-d_x$ for any $z\in \Omega$ and, by \eqref{eq:outerballcondition_too_estimate_II},
		\begin{equation}\label{eq:last_estimate_for_II}
			3\,\abs{\overline{y}-z}\ge \abs{x-z} + d_y-d_x.
		\end{equation}
		By \eqref{eq:z_in_A_x,y} and \eqref{eq:last_estimate_for_II},
		\begin{align}\label{eq:upper_bound}
			s\int\limits_{A_{x,y}} \frac{d_z^s}{\abs{y-z}^{2\,s}} \frac{(1+d_x)^d}{\abs{x-z}^d} \di z &\le s\frac{(1+\sigma/2)^d }{3^{2\,s}} \int\limits_{A_{x,y}} \frac{d_z^s}{\abs{\overline{y}-z}^{2\,s}} \frac{1}{\abs{x-z}^d} \di z\nonumber\\
			&\le s c_9(d,\Omega)	\int\limits_{\Omega} \frac{1}{\big(  \abs{x-z} + d_y-d_x\big)^{2\,s}} \frac{1}{\abs{x-z}^{d-s}} \di z.
		\end{align}
		We want to use polar coordinates to estimate the right-hand side of \eqref{eq:upper_bound}. But first, 
		\begin{align*}
			\int\limits_{d_x}^{d_x+\diam{\Omega}} \frac{t^{s-1}}{(t+d_y-d_x)^{2\,s}} \di t &= \frac{1}{s}\left( \frac{(d_x+\diam{\Omega})^s}{(d_y+\diam{\Omega})^{2\,s}} - \frac{d_x^s}{d_y^{2\,s}}   \right) +2\, \int\limits_{d_x}^{d_x+\diam{\Omega}} \frac{t^s}{(t+d_y-d_x)^{2\,s+1}} \di t\\
			&\le \frac{(\sigma/2+\diam{\Omega})^s}{s\,\diam{\Omega}^{2\,s}} +2\, \int\limits_{d_x}^{d_x+\diam{\Omega}} \frac{1}{(t+d_y-d_x)^{s+1}} \di t\\
			&\le \frac{(\sigma/2+\diam{\Omega})^s}{s\,\diam{\Omega}^{2\,s}} +\frac{2}{s} \frac{1}{d_y^s}\le \frac{4}{s}\frac{1}{d_y^s} \le \frac{4}{s}\frac{(1+\diam{\Omega})^s}{d_y^s (1+ d_y)^s}.
		\end{align*}
		The last inequality is due to the choice of $2\sigma\le \diam{\Omega}$ and by $d_y\le \sigma$. Since $x\in \Omega^c$, $\Omega \subset B_{\diam{\Omega}+d_x}(x)\setminus B_{d_x}(x)$. This calculation together with \eqref{eq:upper_bound} and polar coordinates yields
		\begin{align*}
			s\int\limits_{A_{x,y}} \frac{d_z^s}{\abs{y-z}^{2\,s}} \frac{(1+d_x)^d}{\abs{x-z}^d} \di z \le s \, c_9	\int\limits_{\Omega} \frac{1}{\big(  \abs{x-z} + d_y-d_x\big)^{2\,s}} \frac{1}{\abs{x-z}^{d-s}} \di z\\
			\le s \, c_9\, \omega_{d-1} \int\limits_{d_x}^{d_x+\diam{\Omega}} \frac{t^{s-1}}{(t+d_y-d_x)^{2\,s}} \di t\le c_{10}(d,\Omega) \frac{1}{d_y^s (1+d_y)^s}.
		\end{align*}
		Thus, the last case is finished. Hence, in all cases 
		\begin{align*}
			s\text{(II)} \le c_{11}(d,\Omega) \frac{1}{d_y^s (1+ d_y)^s}.
		\end{align*}
		Combining the estimates on (I) and (II) yield for $A$ from \eqref{eq:first_estimate_upper_bound}
		\begin{equation*}
			A\le sc_3 \frac{\text{(I)}+ \text{(II)}}{(\abs{x-y}+d_xd_y + d_x + d_y)^d}\le c_2(d,\Omega) \frac{1}{ d_y^s (1+ d_y)^s (\abs{x-y}+ d_x\,d_y+d_x+d_y)^d }.
		\end{equation*} 
		which proves \eqref{eq:first_estimate_upper_bound} and finishes the proof.  
	\end{proof}
	\begin{proposition}\label{prop:lowerbound}
		There exists a constant $C=C(d,\Omega)>0$ such that 
		\begin{equation*}
			s^2\,C\,  k_s(x,y)\le k_s^{\star}(x,y)
		\end{equation*}
		for any $x,y\in \overline{\Omega}^{\,c}$.
	\end{proposition}
	\begin{proof}
		Let $\rho>0$ be the uniform inner and outer ball radius of the $C^{1,1}$-domain $\Omega$. Fix $x,y\in \overline{\Omega}^{\,c}$. Without loss of generality $\rho\le 1$. Similar to the proof of \cref{prop:upperbound}, the proof reduces to proving the following. We need to show that there exists a constant $c_1=c_1(d,\Omega)>0$ such that
		\begin{equation}\label{eq:lowerbound_ts}
			\int\limits_{\Omega} \frac{d_z^s}{\abs{y-z}^{d+2s}} \frac{\di z}{\abs{x-z}^d} \ge c_1 \frac{1}{d_y^s\, (1+d_y)^s} \frac{1}{(\abs{x-y}+ d_x\,d_y+d_x+d_y)^d}.
		\end{equation}	
		We distinguish two cases to handle the integral on the left-hand side of \eqref{eq:lowerbound_ts}.
		
		\textbf{Case ($d_y\ge \rho/4$): } We fix a ball $B_{\rho/2}\subset \subset\Omega$ with $\dist{B_{\rho/2}}{\partial \Omega}= \rho/2$. A small calculation yields
		\begin{align*}
			\int\limits_{\Omega} \frac{d_z^s}{\abs{y-z}^{d+2s}} \frac{\di z}{\abs{x-z}^d}&\ge \big(\rho/2\big)^s\int\limits_{B_{\rho/2}} \frac{1}{\abs{y-z}^{d+2s}} \frac{1}{\abs{x-z}^d}\di z\ge \frac{\rho\, \abs{B_{\rho/2}}}{ 2} \frac{2^d}{7^{d+2s} 3^d} \frac{1}{d_y^{2s}\,((1+d_x)d_y)^d}\\
			&\ge  \frac{\rho\, \abs{B_{\rho/2}}}{2} \frac{2^d}{7^{d+2} 3^d} \frac{1}{d_y^s\, (1+d_y)^s} \frac{1}{(\abs{x-y}+ d_x\,d_y+d_x+d_y)^d}.
		\end{align*}
	
		\textbf{Case ($d_y<\rho/4$): }We fix the point $\tilde{y}\in \partial \Omega$ that minimizes the distance of $y$ to the boundary of $\Omega$, \ie such that $\abs{y-\tilde{y}}= d_y$. By uniform interior ball condition, we fix an interior ball $B_\rho(\overline{y})\subset \Omega$ such that $\overline{B_{\rho}(\overline{y})}\cap\partial \Omega=\{\tilde{y}\}$. Now we pick an open Cone $\CC\subset B_{\rho}(\overline{y})$ with apex at $\tilde{y}$, \ie $\overline{\CC}\cap \partial B_{\rho}(\overline{y})= \{ \tilde{y} \}$, and such that $\dist{z}{\partial B_{\rho}(\overline{y})}\ge \frac{\abs{z-\tilde{y}}}{16}$ for any $z\in \CC$, e.g. 
		\begin{equation*}
			\CC=\{  \tilde{y}+t\,(\overline{y}-\tilde{y})+\xi_t\,|\, t\in (0,\rho/2),\, \xi_t\perp (\overline{y}-\tilde{y}), \,\abs{\xi_t}<t/8 \},
		\end{equation*} 
		see \cref{fig:cone}. Notice that the height and the angle at the apex of this cone $\CC$ does not depend on $y$ by the interior ball condition.
		\begin{figure}[!ht]
			\begin{center}
				\begin{tikzpicture}
					\draw (-4,-1) .. controls (0,-1).. (2,-3) node[below left]{$\Omega$};
					\filldraw[black] (0.18,-1) circle (1pt) node[anchor=west]{$y$};
					\filldraw[black] (0,-1.35) circle (1pt) node[anchor=west]{$\tilde{y}$};
					\draw[red] (-0.62,-2.62) circle (1.4 cm) node[below left]{$B_\rho(\overline{y})$};
					\filldraw[black] (-0.62,-2.62) circle (1pt) node[anchor=west]{$\overline{y}$};
					\filldraw[blue] (-0.7,-2.25) -- (0,-1.35) -- (-0.1,-2.5) -- cycle node[left]{$C$};
				\end{tikzpicture}
			\end{center}
			\caption{The Cone $\CC$}
			\label{fig:cone}
		\end{figure} 
	
		Then obviously also $\dist{z}{\partial \Omega}\ge \frac{\abs{z-\tilde{y}}}{16}$ for any $z\in \CC$ holds. We estimate the left-hand side of \eqref{eq:lowerbound_ts} by
		\begin{align*}
			&\int\limits_{\Omega} \frac{d_z^s}{\abs{y-z}^{d+2s}} \frac{\di z}{\abs{x-z}^d}\ge \Big(\frac{1}{16}\Big)^s	\int\limits_{\CC} \frac{\abs{z-\tilde{y}}^s  }{\abs{y-z}^{d+2s}} \frac{\di z}{\abs{x-z}^d}\\
			& \ge \frac{1}{16}\int\limits_{\CC} \frac{\abs{z-\tilde{y}}^s  }{\big( \abs{y-\tilde{y}} + \abs{z-\tilde{y}} \big)^{d+2s}} \frac{1}{\big( \abs{x-y} + \abs{y-\tilde{y}} + \abs{z-\tilde{y}} )^d}\di z=: \text{(I)}.
		\end{align*}
		Now we apply the coarea formula to the integral (I) with the function $z \mapsto  \abs{z-\tilde{y}}$, see \cref{th:coarea}. Notice that the modulus of the gradient of this function is $\abs{\nabla_z \abs{z-\tilde{y}}}=1$. 
		\begin{equation*}
			\text{(I)}\ge \frac{1}{16}\int\limits_{0}^{\rho/3}\int\limits_{\{ z\in \CC\,|\, \abs{z-\tilde{y}}=t \}} \frac{t^s  }{\big( d_y + t \big)^{d+2s}} \frac{1}{\big( \abs{x-y} + d_y + t \big)^d}\di \SH(z)\di t.
		\end{equation*}
		The $(d-1)$-dimensional Hausdorff measure of a ball $B_r$ intersecting a hyperplane scales like $r^{d-1}$. Thus, there exists a constant $c_2(\rho)>0$ such that
		\begin{equation*}
			\SH \big(  \{ z\in \CC\,|\, \abs{z-\tilde{y}}=t \} \big)\ge c_2(\rho) t^{d-1}
		\end{equation*} for any $0<t<\rho/3$, see \eg \cite[Lemma A.4 (A.19)]{RosOton_Vladinoci_appendix}. Thus, we estimate (I) further by
		\begin{align*}
			\text{(I)}&\ge \frac{c_2}{16} \int\limits_{0}^{\rho/3} \frac{t^{d+s-1}  }{\big( d_y + t \big)^{d+2s}} \frac{1}{\big( \abs{x-y} + d_y + t \big)^d}\di t \ge \frac{c_2}{16} \int\limits_{d_y/4}^{d_y} \frac{t^{d+s-1}  }{\big( d_y + t \big)^{d+2s}} \frac{1}{\big( \abs{x-y} + d_y + t \big)^d}\di t \\
			&\ge \frac{c_2}{16} \frac{3\,d_y}{4}\, \frac{\big( d_y/4 \big)^{d+s-1}  }{\big( 2\,d_y \big)^{d+2s}} \frac{1}{\big( \abs{x-y} + 2\,d_y  \big)^d}\ge 3\,c_2 2^{-4d-9}\,\frac{1}{d_y^s\,(1+d_y)^s} \,\frac{1}{\big( \abs{x-y} +d_x\,d_y+ d_y +d_x  \big)^d}.
		\end{align*}
		Therefore, \eqref{eq:lowerbound_ts} is proven in all cases.
	\end{proof}
	\textbf{Proof of \cref{th:equivalence}:} Combine \cref{cor:L2partequivalence}, \cref{prop:upperbound} and \cref{prop:lowerbound}. \qed
	\subsection{Trace and extension operators}\label{sec:trace_and_extension}
	In this section we prove the existence of a trace and extension operator for the Sobolev-Slobodeckij-type space $V^s(\Omega \, | \, \BR^d)$ with respect to the complement $\Omega^c$. In contrast to the classical trace operator $\trc : H^1(\Omega)\to H^{1/2}(\partial\Omega)$ the construction of the nonlocal trace is simply the restriction of a function $u\in V^s(\Omega\, |\, \BR^d)$ to the complement of a domain $\Omega$, \ie $u|_{\Omega^c}$. This is due to $\Omega^c$ being $d$-dimensional. Thereby, the proof of \cref{th:trace_and_extension} $(1)$ is rather straight forward. We only need to show the continuity of the trace embedding, \ie $\norm{u|_{\Omega^c}}_{\ST^s(\Omega^c)}\le C/s \norm{u}_{V^s(\Omega\,|\, \BR^d)}$.\medskip
	
	We pay particular attention to the independence of the constant $C=C(d,\Omega)$ on $s$. The advantage of this robust estimate will be highlighted in the forthcoming \cref{sec:convergence} where we consider the limit case $s\to 1-$. In particular, \cref{th:trace_and_extension}(1) yields the classical trace inequality, \ie $\norm{\trc  u}_{H^{1/2}(\partial \Omega)}\le C\, \norm{u}_{H^1(\Omega)}$, in the limit $s\to 1-$. In addition, this robustness allows us to consider a large class of Neumann data in \cref{sec:neumann}.\medskip 
	
	In view of \cref{sec:equiv}, the results in \cite{Bogdan_trace}, obtained with stochastic methods, play a key role in our considerations. In particular, they showed that $u|_{\Omega^c}\in \CX^s(\Omega^c)$ for any $u\in V^s(\Omega\, |\, \BR^d)$ such that the trace operator $\trn : V^s(\Omega\, |\, \BR^d)\to \CX^s(\Omega^c)$ is well defined. Additionally, the estimate $[u|_{\Omega^c}, u|_{\Omega^c}]_{\CX^s(\Omega^c)}\le [u,u]_{V^s(\Omega\,|\, \BR^d)}$ from \cref{th:douglas} (ii), see \cite[Theorem 2.3]{Bogdan_trace}, together with the comparison of $k_s^{\star}$ and $k_s$ from \cref{th:equivalence}, see also \cref{prop:lowerbound} and \cref{prop:upperbound}, are crucial to our proof of the continuity of the trace operator $\trn: V^s(\Omega\,|\, \BR^d)\to \ST^s(\Omega^c)$. We remark that the continuity of $\trn: V^s(\Omega\,|\, \BR^d)\to \CX^s(\Omega^c)$ as a linear map between normed spaces has not been proven in \cite{Bogdan_trace}. In particular, the estimate $\abs{u|_{\Omega^c}}_{\CX^s(\Omega^c)}\le C \norm{u}_{V^s(\Omega\, |\, \BR^d)}$ has not been proven.\medskip 
	
	The extension operator $\extn : \ST^s(\Omega^c)\to V^s(\Omega\,|\, \BR^d)$ will be the Poisson extension $\SP_{s,\Omega}$, defined in \eqref{eq:poisson_extension}. Bogdan et al. have proven that $\SP_{s,\Omega}:\CX^s(\Omega^c)\to V^s(\Omega\,|\, \BR^d)$ is a well defined map. Additionally, they proved a Douglas identity, \ie $[\SP_{s,\Omega}g, \SP_{s,\Omega}g]_{V^s(\Omega\, |\, \BR^d)}= [g,g]_{\CX^s(\Omega^c)}$, see \cite[Theorem 2.3]{Bogdan_trace} and \cref{th:douglas} (i). By this Douglas identity and the equivalence results from the previous section, see \cref{th:equivalence}, the proof of \cref{th:trace_and_extension} (2) boils down to show a robust estimate on the weighted $L^2$-norm of the Poisson extension, \ie $\norm{\SP_{s,\Omega}g}_{L^2(\Omega)}\le C(d,\Omega) \norm{g}_{L^2(\Omega^c, \tau_s)}$. \medskip
	
	The following proposition is a key ingredient in the proof of the robust trace continuity. It is also interesting on its own. In sight of the convergence results in \cref{sec:convergence}, the inequality \eqref{eq:L2_trace_robust} is a robust approximation of the classical trace inequality $\norm{\trc  u}_{L^2(\partial \Omega)}\le C\, \norm{u}_{H^1(\Omega)}$. We split the proof in two cases. In case of a small parameter $s$ we use the fractional Hardy inequality, which has been proven in \cite{chen_hardy_song} by Chen and Song, in \cite{hardy_inequality} by Dyda and in \cite{brasco_cinti_hardy} by Brasco and Cinti. In case of big parameter $s<1$ we apply similar arguments as in the proof of the classical trace theorem $H^s(\Omega)\to L^2(\partial \Omega)$. 
	\begin{proposition}\label{prop:robust_trace_hardy_type}
		Let $\Omega$ be a bounded Lipschitz domain and $s_\star\in(0,1/2)$. There exists a constant $C=C(d,\Omega, s_\star)>0$ such that for any $s\in (s_\star,1)$ and $u \in H^s(\Omega)$ 
		\begin{equation}\label{eq:L2_trace_robust}
			(1-s)\int\limits_{ \Omega} \frac{u(x)^2}{d_x^s}\di x \le C \Big( \norm{u}_{L^2(\Omega)}^2+ (1-s)\iint\limits_{\Omega\times \Omega} \frac{\abs{u(x)-u(y)}^2}{\abs{x-y}^{d+2s}}\di y \di x  \Big).
		\end{equation}
	\end{proposition}
	\begin{proof}
		Fix $s^\star\in(1/2,1)$. We will distinguish the cases $s\in(s_\star, s^\star)$ and $s\in[s^\star,1)$.
		
		\textbf{Case 1: } Let $s\in(s_\star, s^\star)$. By fractional Hardy inequality, see \cite[Theorem 1.1, (17)]{hardy_inequality}, \cite[Theorem 2.3]{chen_hardy_song}, there exists a constant $c_1=c_1(d,\Omega,s_\star/2, s^\star/2 )>0$ such that for any $f\in C_c(\Omega)$ 
		\begin{equation*}
			\int\limits_{\Omega} \frac{\abs{f(z)}^2}{d_x^{2(s/2)}}\le c_1 \Big( \iint\limits_{\Omega\times \Omega} \frac{\abs{f(x)-f(y)}^2}{\abs{x-y}^{d+2(s/2)}}\di x \di y + \norm{f}_{L^2(\Omega)}^2\Big).
		\end{equation*}
		Since $s/2<s^\star/2<1/2$, $C_c^\infty(\Omega)$ is dense in $H^{s/2}(\Omega)$, see \cite[Theorem 1.4.2.4]{grisvard}, \cite[Theorem 3.4.3]{triebel_1}, the inequality holds for all functions in $H^{s/2}(\Omega)\supset H^s(\Omega)$. Thus, for any $u\in H^s(\Omega)$
		\begin{align*}
			(1-s)\int\limits_{\Omega} \frac{\abs{u(z)}^2}{d_x^{s}}&\le (1-s) \, c_1\Big( \iint\limits_{\Omega\times \Omega} \frac{\abs{u(x)-u(y)}^2}{\abs{x-y}^{d+s}}\di x \di y + \norm{u}_{L^2(\Omega)}^2\Big)\\
			&\le c_1\Big( \diam{\Omega}^{s} (1-s) \iint\limits_{\Omega\times \Omega} \frac{\abs{u(x)-u(y)}^2}{\abs{x-y}^{d+2s}}\di x \di y + \norm{u}_{L^2(\Omega)}^2\Big).
		\end{align*}
	
		\textbf{Case 2: } Let $s\in[s^\star,1)$. In case 1 we used Hardy inequality and the density of compactly supported smooth functions in $H^{r/2}(\Omega)$, $r/2<1/2$ to prove our claim. This argument is not robust as $s$ approaches $1$. The benefit of $s\ge s^\star>1/2$ is the existence of a continuous trace onto $\partial \Omega$. We apply arguments similar to \cite[Proposition 3.8]{hitchhiker}. Since $\partial \Omega$ is Lipschitz and compact, by \cite[Main theorem S. 146]{Extension_Jonsson_Wallin} or \cite[Theorem 5.4]{hitchhiker}, there exists a continuous extension operator $\tilde{E}_s: H^s(\Omega)\to H^s(\BR^d)$. By following the constants in the proof of \cite[Theorem 5.4, Lemma 5.1, 5.2, 5.3]{hitchhiker} it is clear that we can choose a constant $c_{\ge s^\star}\ge 1$ depending only on $d,\Omega$ and $s^\star$ such that \begin{equation}\label{eq:Hs_extension}
			\norm{\tilde{E}_{s'}f}_{L^2(\BR^d)} + (1-s')[\tilde{E}_{s'} f]_{H^{s'}(\BR^d)}\le c_{\ge s^\star} \Big( \norm{f}_{L^2(\Omega)} + (1-s')[f]_{H^{s'}(\Omega)} \Big)
		\end{equation} for all $f\in H^{s'}(\Omega)$ and $s'\in[s^\star,1)$.\medskip 
		
		Now, we localize the problem. Because $\Omega$ is a bounded Lipschitz domain, we find finitely many cubes $Q_{r}(z_i)= z_i + (-r,r)^d$, $z_i\in \partial \Omega$, $0<r<1$ such that $\partial \Omega \subset \bigcup_{i=1}^N Q_{r/2}(z_i)$, $N\in \BN$. Additionally, we fix bijective, bi-Lipschitz continuous maps $\phi_i: Q_r(z_i)\to Q_1=(-1,1)^d$ such that $\phi_i\big(\Omega\cap Q_r(z_i)\big)= \{ (x',x_d)\in Q_1\,|\, x_d>0 \}=:Q_1^+$. Since $\partial \Omega\subset \bigcup_{i=1}^N Q_{r/2}(z_i)$ is compactly embedded in an open set, there exists $r_0>0$ such that $R_0:= \{ x\in \Omega\,|\, \dist{x}{\partial \Omega}>r_0 \}$ satisfies $R_0 \cup \bigcup_{i=1}^N \big(Q_{r/2}(z_i)\cap \Omega\big)= \Omega$. Also, we fix a partition of unity $\eta_i\in C_c^\infty(Q_{r/2}(z_i))$,$0\le \eta_i\le 1$, $i=1,\dots, N$, $\eta_0 \in C_c^\infty(R_0)$ such that $\sum_{i=0}^N \eta_i=1$ in $\Omega\cup \Omega_{\tilde{r}}$ for some $0<\tilde{r}<r/2$. By our geometric considerations above, 
		\begin{align*}
			&(1-s)\int\limits_\Omega \frac{u(x)^2}{d_x^s} \di x =(1-s)\int\limits_{R_0} \eta_0(x)\frac{u(x)^2}{d_x^s} \di x\\
			&\qquad\qquad+  \sum_{i=1}^{N}(1-s)\int\limits_{\phi_i\big(\Omega\cap Q_{r/2}(z_i)\big)} \eta_i(\phi_i^{-1}(y))\frac{u(\phi_i^{-1}(y))^2}{\dist{\phi_i^{-1}(y)}{\partial \Omega}^s} \abs{\det D\phi_i^{-1}(y)}\di y.
		\end{align*}
		We split the remainder of the proof of case 2 into the case of $R_0$ and the boundary cases.
		
		\textbf{Case $R_0$: } Let $s\in[s^\star,1)$ and $u\in H^s(\Omega)$. 
		\begin{equation*}
			(1-s)\int\limits_{R_0}\eta_0(x) \frac{u(x)^2}{d_x^s}\di x \le r_0^{-s}\norm{u}_{L^2(R_0)}^2.
		\end{equation*}
	
		\textbf{The boundary cases: }Now, we prove the inequality integrating over $Q_r(z_i)\cap \Omega$. Since $\phi_i$ is bijective and bi-Lipschitz, it is differentiable \aev and there exists a constant $\lambda_i>1$ such that  $\lambda_i^{-1}\le \abs{\det D\phi_i(x)}\le \lambda_i$ for almost every $x\in Q_r(z_i)$. We extend $\phi_i$ to a map on $\BR^d$ via
		\begin{equation*}
			\phi_i(x):= \frac{x-z_i}{r}, x\notin Q_r(z_i).
		\end{equation*}
		Notice that $\phi_i(x)\notin Q_1$ for $x\notin Q_r(z_i)$.\medskip 
		
		We begin by proving the statement for the half space and localize thereafter. Let $f\in \CS(\BR^d)$ be a Schwartz function. We use the convention $x=(x',x_d)\in \BR^d$. By $\SF f$ (resp. $\SF_{x'}$) we denote the Fourier-transformation of $f$ (resp. in the first $d-1$ variables). Similar to the arguments in the proof of \cite[Proposition 3.8]{hitchhiker}
		\begin{equation*}
			\SF_{x'}f(\xi',t)= \int\limits_{\BR} \SF f(\xi', \xi_d) e^{it\xi_d} \di \xi_d.
		\end{equation*}
		Therefore,  
		\begin{align*}
			\abs{\SF_{x'}f(x',t)}^2 \le \int\limits_{\BR} (1+\abs{\xi'}^2+ \xi_d^2)^s \abs{\SF f(\xi',\xi_d)}^2 \di \xi_d \,\int\limits_{\BR} (1+\abs{\xi'}^2+ \xi_d^2)^{-s}\di \xi_d.
		\end{align*}
		Additionally, 
		\begin{equation*}
			\int\limits_{\BR} (1+\abs{\xi'}^2+ \xi_d^2)^{-s}\di \xi_d\le 2^{1+s}\, \int\limits_1^\infty \xi_d^{-2s} \di \xi_d= \frac{2^{1+s}}{2s-1}\le \frac{2^2}{2s^\star -1}.
		\end{equation*}
		The previous two estimates and Plancherel's theorem yield
		\begin{align}
			(1-s)\int\limits_{\BR^{d-1}} \int\limits_{(0,1)} \frac{f(x',x_d)^2}{x_d^s}\di x_d\di x' &= (1-s)\int\limits_{\BR^{d-1}} \int\limits_{(0,1)} \frac{\abs{\SF_{x'}f(\xi',x_d)}^2}{x_d^s} \di x_d\di \xi'\nonumber\\
			\le \tfrac{2^{1+s}}{2s-1}\int\limits_{(0,1)} \frac{1-s}{x_d^s} \di x_d\, \int\limits_{\BR^d} (1+\abs{\xi}^2)^s \abs{\SF f(\xi)}^2 \di \xi 
			&\le \tfrac{2^{1+s}}{2s-1} \Big( 2\kappa_{d,s} \iint\limits_{\BR^d\times \BR^d} \frac{(f(x)-f(y))^2}{\abs{x-y}^{d+2s}} \di x \di y + \norm{f}_{L^2(\BR^d)}^2 \Big).\label{eq:classical_trace_estimate}
		\end{align}
		The last inequality in the previous calculation follows from Sobolev embeddings \cite[Proposition 3.4]{hitchhiker}. Since the Schwartz functions are dense in $H^s(\BR^d)$, the estimate \eqref{eq:classical_trace_estimate} holds for all $f\in H^s(\BR^d)$.
		
		Now, for any $x\in Q_{r/2}(z_i)\cap \Omega$ we find a minimizer of the distance of $x$ to the boundary $\partial \Omega$ in $Q_r(z_i) \cap \partial\Omega$. If this is not possible, we pick even smaller sub-cubes $Q_{\tilde{r}}(z_i)\subset  Q_{r/2}(z_i)\subset Q_r(z_i)$  to cover the boundary. Therefore, 
		\begin{equation*}
			\dist{\phi_i^{-1}(y)}{\partial \Omega}= \inf\{ \abs{\phi_i^{-1}(z)-\phi_i^{-1}(y)}\,|\, z=(z',0)\in Q_1 \}\le \norm{\phi_i}_{C^{0,1}}^{-1} \, \abs{y_d}, \, \, y\in Q_1.
		\end{equation*}
		Now we prove that $\sqrt{\eta_i}$ is Lipschitz. By Taylor's formula, if $\eta_i(x)=0$ for any $x\in \BR^d$, then $\nabla \sqrt{\eta_i}(x)=0$. Since $\eta_i\in C_c^\infty$, $D^2 \eta_i(x)h \cdot h \le \norm{\eta_i}_{C^2(\BR^d)} \abs{h}^2$. This yields a Glaeser-type inequality 
		\begin{equation*}
			\abs{\nabla \eta_i(x)}\le \sqrt{2\norm{\eta_i}_{C^2}\, \eta_i(x)},
		\end{equation*}
		see \eg \cite[Lemma 1]{glaeser_inequality}. Therefore, $\abs{\nabla \sqrt{\eta_i}}\le \sqrt{2\norm{\eta_i}_{C^2(\BR^d)}}$ and, thus, $\sqrt{\eta_i}\in C_b^{0,1}(\BR^d)$. 
		Since $\phi_i$ is bi-Lipschitz, $\sqrt{\eta_i\circ \phi_i^{-1}}\in C_c^{0,1}(Q_1)$. Let $f\in \CS(\BR^d)$ and write $h:= \sqrt{\eta_i\circ \phi_i^{-1}} \, f\circ \phi_i^{-1}\in H^s(\BR^d)$. \eqref{eq:classical_trace_estimate} yields
		\begin{align}
			&(1-s)\int\limits_{Q_1^+} \eta_i(\phi_i^{-1}(y))\frac{f(\phi_i^{-1}(y))^2}{\dist{\phi_i^{-1}(y)}{\partial \Omega}^s} \abs{\det D\phi_i^{-1}(y)}\di y\le \lambda_i\norm{\phi_i}_{C^{0,1}}\, (1-s)\int\limits_{\BR^{d-1}}\int\limits_{ (0,1)} \frac{h(y',y_d)^2}{y_d^s} \di y_d \di y'\nonumber\\
			&\qquad\qquad\le \lambda_i\norm{\phi_i}_{C^{0,1}}\, \tfrac{2^{1+s}}{2s^\star-1} \Big( 2\kappa_{d,s} \iint\limits_{\BR^d\times \BR^d} \frac{(h(x)-h(y))^2}{\abs{x-y}^{d+2s}} \di x \di y + \norm{h}_{L^2(\BR^d)}^2 \Big).\label{eq:robust_trace_hardy_help1}
		\end{align}
		Since $\supp \eta_i \subset Q_{r/2}(z_i)$, $\phi_i$ is bi-Lipschitz and by transformation theorem, $\norm{h}_{L^2(\BR^d)}^2\le \lambda_i\, \norm{f}_{L^2(\BR^d)}$. We split the seminorm term on the right-hand side of \eqref{eq:robust_trace_hardy_help1} into $Q_1\times Q_1$, $\supp(\eta_i\circ \phi_i^{-1})\times Q_1^c$ and $(\supp(\eta_i\circ \phi_i^{-1}))^c \times Q_1^c$. Since $\eta_i\circ \phi_i^{-1}$ is zero on $Q_1^c$, the term \begin{equation*}
			\iint\limits_{(\supp(\eta_i\circ \phi_i^{-1}))^c\times Q_1^c} \frac{(h(x)-h(y))^2}{\abs{x-y}^{d+2s}}\di x \di y =0.
		\end{equation*} Now, on $Q_1\times Q_1$
		\begin{align}
			\iint\limits_{Q_1\times Q_1} \frac{(h(x)-h(y))^2}{\abs{x-y}^{d+2s}} \di x \di y&= \iint\limits_{Q_{r}(z_i)\times Q_{r}(z_i)} \frac{(h\circ\phi_i(x)-h\circ\phi_i(y))^2}{\abs{\phi_i(x)-\phi_i(y)}^{d+2s}}\abs{\det D\phi_i(y)}\abs{\det D\phi_i(x)} \di y \di x\nonumber\\
			&\le \lambda_i^2 \norm{\phi_i^{-1}}_{C^{0,1}}^{-d-2s} \iint\limits_{Q_{r}(z_i)\times Q_{r}(z_i)} \frac{(\eta_i(x)^{1/2}f(x)-\eta_i(y)^{1/2}f(y))^2}{\abs{x-y}^{d+2s}}\di x \di y\nonumber\\
			&\le \lambda_i^2 \norm{\phi_i^{-1}}_{C^{0,1}}^{-d-2s}\Bigg(\,\, \iint\limits_{\BR^d\times \BR^d} \frac{(f(x)-f(y))^2}{\abs{x-y}^{d+2s}}\di x \di y\nonumber\\
			&\quad+ \norm{\sqrt{\eta_i}}_{C^{0,1}}\int\limits_{\BR^d}f(y)^2\int\limits_{ B_2(x)} \abs{x-y}^{-d+2(1-s)}\di x \di y\Bigg)\nonumber\\
			&\le  \lambda_i^2 \norm{\phi_i^{-1}}_{C^{0,1}}^{-d-2s}\Bigg(\,\, \iint\limits_{\BR^d\times \BR^d} \frac{(f(x)-f(y))^2}{\abs{x-y}^{d+2s}}\di x \di y\nonumber\\
			&\quad+ \omega_{d-1} \frac{1}{1-s} 2^{2(1-s)} \norm{\sqrt{\eta_i}}_{C^{0,1}}\norm{f}_{L^2(\BR^d)}^2\Bigg).\label{eq:robust_hardy_help2}
		\end{align}
		Finally, we consider the integral over $(\supp(\eta_i\circ \phi_i^{-1}))\times Q_1^c$. Since $\phi_i:Q_r(z_i)\to Q_1$ is bijective and bi-Lipschitz, $\phi_i(\supp \eta_i)\ssubset Q_1$ is compactly embedded. Therefore, $\delta:= \dist{\phi_i(\supp \eta_i)}{\partial Q_1}\in(0,1)$. Thus, $\abs{\phi_i(x)-y}\ge \abs{y}-1+\delta \ge \delta\abs{y}$ for $x\in \supp(\eta_i)\subset Q_{r/2}(z_i)$ and $y\in Q_1^c$.
		\begin{align}
			\iint\limits_{(\supp(\eta_i\circ \phi_i^{-1}))\times Q_1^c} \frac{(h(x)-h(y))^2}{\abs{x-y}^{d+2s}} \di x \di y&\le \lambda_i \int\limits_{\supp(\eta_i)}\int\limits_{ Q_1^c}\eta_i(x) \frac{f(x)^2}{\abs{\phi_i(x)-y}^{d+2s}} \di y \di x\nonumber\\
			&\le \delta^{-d-2s}\lambda_i \int\limits_{\BR^d} f(x)^2 \di x \int\limits_{Q_1^c} \abs{y}^{-d-2s}\di y\nonumber\\
			&\le\delta^{-d-2s}\lambda_i \frac{\omega_{d-1} }{2s}\int\limits_{\BR^d} f(x)^2 \di x.\label{eq:robust_hardy_help3}			
		\end{align}
		Now we combine \eqref{eq:robust_trace_hardy_help1},\eqref{eq:robust_hardy_help2} and \eqref{eq:robust_hardy_help3}. Thus, there exists a constant $c_1=c_1(d,\Omega,s^\star)>0$ independent of $s$ such that for any $f\in H^s(\BR^d)$
		\begin{equation*}
			(1-s)\int\limits_{Q_r(z_i)\cap \Omega} \eta_i(x) \frac{f(x)^2}{d_x^2} \di x\le c_1 \Big(  (1-s)\iint\limits_{\BR^d\times \BR^d} \frac{(f(x)-f(y))^2}{\abs{x-y}^{d+2s}} \di x \di y + \norm{f}_{L^2(\BR^d)}^2 \Big). 
		\end{equation*}
		This concludes the boundary cases. Consider $u\in H^s(\Omega)$. Combining the $R_0$ and boundary cases yield a constant $c_2=c_2(d,\Omega,s^\star)>0$ independent of $s$ such that 
		\begin{align*}
			(1-s)\int\limits_{\Omega} \frac{(\tilde{E}_s u(x))^2}{d_x^s} \di x &\le c_2 \Big(  (1-s)\iint\limits_{\BR^d\times \BR^d} \frac{(\tilde{E}_s u(x)-\tilde{E}_s u(y))^2}{\abs{x-y}^{d+2s}} \di x \di y + \norm{\tilde{E}_s u}_{L^2(\BR^d)}^2 \Big)\\
			&\le c_{\ge s^\star}c_2 \Big(  (1-s)\iint\limits_{\Omega\times \Omega} \frac{( u(x)- u(y))^2}{\abs{x-y}^{d+2s}} \di x \di y + \norm{ u}_{L^2(\Omega)}^2 \Big).
		\end{align*}
		The last inequality follows from \eqref{eq:Hs_extension}. This proves case 2 and, thus, the proposition. 
	\end{proof}
	The previous proposition, the comparison of $k_s^{\star}$ and $k_s$ from \cref{sec:equiv} and the results in \cite{Bogdan_trace} particularly enable us to prove \cref{th:trace_and_extension}.\medskip
	
	\textbf{Proof of \cref{th:trace_and_extension}}
	\textit{(1)}: For $ v \in V^s(\Omega \, | \, \BR^d)$ we define $\trn v := v|_{\Omega^c}$. We divide the proof into the estimate for the seminorm and the $L^2$-part.
	
	\textbf{Seminorm-part: }By \cref{th:douglas}, $\trn v\in \CX^s(\Omega^c)$ and 
	\begin{equation}\label{eq:seminorm_trace_continuity_1}
		[v,v]_{V^s(\Omega \, | \, \BR^d)} \ge [\trn v,\trn v]_{\CX^s(\Omega^c)}
	\end{equation}
	There exists a constant $c_1=c_1(d,\Omega)>0$ such that $k_s^{\star}(x,y)\ge  c_1\,s^{2} k_s(x,y)$ by \cref{prop:lowerbound}. Therefore, this comparability of the interaction kernels $k_s^{\star}, k_s$ and \eqref{eq:seminorm_trace_continuity_1} yield
	\begin{equation}
		[v,v]_{V^s(\Omega \, | \, \BR^d)}  \ge  [\trn v,\trn v]_{\CX^s(\Omega^c)} \ge c_1\,s^2 [\trn v, \trn v]_{\ST^{s}(\Omega^c|\Omega^c)}.
	\end{equation}
	\textbf{$L^2$-part: }We split $\norm{\trn v}_{L^2(\Omega^c,\tau_s)}^2$ into a part close to $\partial \Omega$ and far away.
	\begin{align*}
		\norm{\trn v}_{L^2(\Omega^c, \tau_s)}^2 &= (1-s)\int\limits_{\Omega^c} \frac{v(x)^2}{d_x^s (1+d_x)^{d+s}}\di x \le (1-s)\int\limits_{\Omega_1} \frac{v(x)^2}{d_x^s}\di x + (1-s)\int\limits_{\Omega^1} \frac{v(x)^2}{d_x^{d+2s}}\di x \\
		&= (I)+ (II).
	\end{align*}
	We begin by estimating $(I)$. By \cref{lem:distance_lower}, there exists a constant $c_2=c_2(d,\Omega)$ such that
	\begin{align*}
		(I)&\le (1-s)c_2\,\int\limits_{\Omega_1} \int\limits_{\Omega}\frac{v(x)^2}{\abs{x-y}^{d+s}}\di y\di x\\
		&\le (1-s)2c_2\,\Bigg(\int\limits_{\Omega_1} \int\limits_{\Omega}\frac{(v(x)-v(y))^2}{\abs{x-y}^{d+s}}\di y\di x+ \int\limits_{\Omega}v(y)^2\int\limits_{\Omega_1}\frac{1}{\abs{x-y}^{d+s}}\di x\di y \Bigg)\\
		&=: (III)+ (IV).
	\end{align*}
	Surely, 
	\begin{align*}
		(III)&\le (1-s)\,4\,c_2\,(\diam{\Omega}+1)^{s}\int\limits_{\Omega_1} \int\limits_{\Omega}\frac{(v(x)-v(y))^2}{\abs{x-y}^{d+2s}}\di y\di x\\
		&\le 4\,c_2\,(\diam{\Omega}+1) \frac{(1-s)}{\kappa_{d,s}} [v,v]_{V^s(\Omega \, | \, \BR^d)}\le c_3(d,\Omega)\, \frac{1}{s} [v,v]_{V^s(\Omega \, | \, \BR^d)}.
	\end{align*}
	Here we used \cref{prop:asymptotics_constant}. Let $c_4=c_4(d,\Omega)>0$ be the constant from \cref{prop:robust_trace_hardy_type}. Since $v\in V^s(\Omega\,|\, \BR^d)$, $v|_{\Omega}\in H^s(\Omega)$. Now, we apply \cref{lem:distance_upper}, \cref{prop:robust_trace_hardy_type} and \cref{prop:asymptotics_constant} to estimate $(IV)$.
	\begin{align*}
		(IV)&\le \frac{\omega_{d-1}}{s}2c_2\,(1-s)\int\limits_{\Omega}\frac{v(y)^2}{d_y^s}\di y\\
		&\le \frac{\omega_{d-1}}{s}2c_2\,c_4\, \Big( \norm{u}_{L^2(\Omega)}^2+ (1-s)\iint\limits_{\Omega\times \Omega} \frac{\abs{u(x)-u(y)}^2}{\abs{x-y}^{d+2s}}\di y \di x  \Big)\\
		&\le \frac{\omega_{d-1}}{s^2}\,4\,c_2\,c_4\, \norm{v}_{V^s(\Omega\,|\, \BR^d)}^2.
	\end{align*}
	Thus, the estimate on $(I)$ is proven. Now, we estimate $(II)$. 
	\begin{align*}
		(II)&\le 2(1-s)\Bigg(\int\limits_{\Omega^1}\fint\limits_{\Omega} \frac{(v(x)-v(y))^2}{d_x^{d+2s}}\di y\di x+ \fint\limits_{\Omega} v(y)^2\di y \int\limits_{\Omega^1}\frac{1}{d_x^{d+2s}}\di x\Bigg)=: (V)+(VI).
	\end{align*}
	Firstly, by \cref{prop:asymptotics_constant}
	\begin{equation*}
		(V)\le 2\frac{(1-s)}{\abs{\Omega}}(\diam{\Omega}+1)^{d+2s} \int\limits_{\Omega^1}\int\limits_{\Omega} \frac{(v(x)-v(y))^2}{\abs{x-y}^{d+2s}}\di y\di x\le \frac{4}{\abs{\Omega}}(\diam{\Omega}+1)^{d+2s} \frac{1}{s} [v,v]_{V^s(\Omega\,|\, \BR^d)}.
	\end{equation*}
	We fix $x_0\in \Omega$. Then $B_1(x_0)\subset (\Omega^1)^c$. For any $x\in \Omega^1$ the distance of $x$ to the boundary is bounded from below by $d_x \ge \tfrac{1}{1+d_{x_0}}\abs{x-x_0}$. Therefore, 
	\begin{align*}
		(VI)&\le  2 \frac{(1-s)}{\abs{\Omega}}(1+d_{x_0})^{d+2s}\,\norm{v}_{L^2(\Omega)}^2 \int\limits_{\Omega^1} \abs{x-x_0}^{-d-2s}\di x\\
		&\le 2 \frac{(1-s)}{\abs{\Omega}}(1+d_{x_0})^{d+2}\,\omega_{d-1}\norm{v}_{L^2(\Omega)}^2 \int\limits_{1}^\infty t^{-1-2s}\di t= \frac{1}{s} \frac{(1-s)}{\abs{\Omega}}(1+d_{x_0})^{d+2}\,\omega_{d-1}\norm{v}_{L^2(\Omega)}^2 . 
	\end{align*}
	This proves \textit{(1)}.
	
	\textit{(2)}: Let $ g\in \ST^{s}(\Omega^c)$. By \cref{th:equivalence}, $g \in \CX^s(\Omega^c)$. The Poisson extension \eqref{eq:poisson_extension} satisfies $\SP_{s,\Omega}(g) \in V^s(\Omega \, | \, \BR^d)$ by \cref{th:douglas}. Therefore, we define $\extn g = \SP_{s,\Omega}(g)$. 
	
		\textbf{Seminorm-part}: By \cref{prop:upperbound} and \cref{th:douglas} there exists a constant $c_4=c_4(d,\Omega)>0$ such that for all $s\in(0,1)$ and $ g\in \ST^{s}(\Omega^c)$ 
	\begin{align*}
		[\extn g,\extn g]_{V^s(\Omega \, | \, \BR^d)}  = 	[g,g]_{\CX^s(\Omega^c)} \le s c_4 [g,g]_{\ST^{s}(\Omega^c|\Omega^c)} \le c_4 [g,g]_{\ST^{s}(\Omega^c|\Omega^c)}.
	\end{align*} 
	\textbf{$L^2$-part}: Since the Poisson kernel $P_{s,\Omega}$ integrates over $\Omega^c$ up to $1$, see \cite[Corollary A.2]{Bogdan_trace}, 
	\begin{align*}
		\norm{\extn g }_{L^2(\Omega)}^2 = \int\limits_{ \Omega}\bigg( \int\limits_{ \Omega^c} g(y) P_{s,\Omega}(x,y) \di y \bigg)^2 \di x \le \int\limits_{ \Omega^c} g(y)^2 \int\limits_{ \Omega} P_{s,\Omega}(x,y) \di x \di y.
	\end{align*}
	We prove that there exists a constant $c_5=c_5(d,\Omega)$ such that for every $y \in \Omega^c$
	\begin{align}\label{eq:to_show_1}
		\int\limits_{ \Omega} P_{s,\Omega}(x,y) \di x  \le c_5 \tau_{s}(y).
	\end{align}
	By \cref{th:chen99_poissonestimate} and \cref{prop:asymptotics_constant} there exists a constant $c_6=c_6(d,\Omega)$ such that 
	\begin{align*}
		P_{s,\Omega}(x,y) \le c_6 \,s\,(1-s)\,  \frac{d_x^{s}}{d_y^s\, (1+d_y)^s} \frac{1}{\abs{x-y}^d} .
	\end{align*}
	Therefore, it is enough to prove the existence of a constant $c_7=c_7(d,\Omega)$ such that
	\begin{align}\label{eq:to_show}
		\int_{\Omega} \frac{d_x^s}{\abs{x-y}^d} \di x \le \frac{c_7}{s} \frac{1}{(1+d_y)^d}.
	\end{align}
	We divide the proof of \eqref{eq:to_show} into two cases. First, if $y \in \Omega^1$, then 
	\begin{align*}
		\int_{\Omega} \frac{d_x^s}{\abs{x-y}^d} \di x \le 2^d \int_{\Omega} \frac{d_x^s}{(1+d_y)^d} \di x \le 2^d \diam{\Omega}^{s} \abs{\Omega} \frac{1}{(1+d_y)^d}\le \frac{2^d (1\vee \diam{\Omega}) \abs{\Omega} }{(1+d_y)^d}.
	\end{align*} 
	This proves \eqref{eq:to_show} in the first case. Next, suppose $y \in \Omega_1$. By polar coordinates, we receive
	\begin{align*}
		\int_{\Omega} \frac{d_x^s}{\abs{x-y}^d} \di x &\le \int_{\Omega} \frac{1}{\abs{x-y}^{d-s}} \di x \le \omega_{d-1}\int_{0}^{\diam{\Omega}+1} t^{-1+s}  \di t\\ 
		&= \omega_{d-1} \frac{1}{s} (\diam{\Omega}+1)^s  \le  \frac{2^d\omega_{d-1}(\diam{\Omega}+1)}{s} \frac{1}{(1+d_y)^d}.
	\end{align*}
	Thus, \eqref{eq:to_show} is proven and it implies \eqref{eq:to_show_1}. Finally, we conclude by \eqref{eq:to_show_1}
	\begin{align*}
		\norm{\extn g}_{L^2(\Omega)}^2 \le c_5 \int_{\Omega^c} g(y)^2 \tau_{s}(y)\di y \le c_5  \norm{g}^2_{L^2(\Omega^c, \tau_s)}
	\end{align*}\qed
	\begin{cor}\label{prop:dense subset}
		Let $\Omega$ be a bounded $C^{1,1}$ domain. $\trn C_c^{\infty}(\BR^d)=C_c^\infty(\Omega^c) \subset \ST^s(\Omega^c)$ is dense in $\ST^s(\Omega^c)$.
	\end{cor}
	\begin{proof}
		Take any $f \in \ST^s(\Omega^c)$. \cref{th:trace_and_extension} yields $\extn f \in V^s(\Omega\,|\, \BR) $. By \cite[Theorem 3.70]{FoghemGounoue2020}, the space $C_c^{\infty}(\BR^d)$ is dense in $V^s(\Omega\,|\, \BR^d)$, thus there exists a sequence of functions $v_n \in C_c^{\infty}(\BR^d)$ which converges to $\extn f$ in $V^s(\Omega\, |\, \BR^d)$. By the continuity of the nonlocal trace operator $\trn$, see \cref{th:trace_and_extension}, it follows that $u_n := \trn v_n \in C_c^{\infty}(\Omega^c)$ converges to $f$ in $\ST^s(\Omega^c)$.
	\end{proof}
	\subsection{Abstract trace space}\label{sec:abstract_trace_space}
	It is a classical assertion that the image of the trace operator $\trc :H^1(\Omega)\to L^2(\partial \Omega)$ is isometrically isomorph to the quotient space $\trc (H^1(\Omega))\simeq H^1(\Omega)/ \ker(\trc )= H^1(\Omega)/H_0^1(\Omega)$. Thereby, we can identify $H^1(\Omega)/H_0^1(\Omega)$ as the abstract trace space of $H^1(\Omega)$ with respect to the topological boundary $\partial \Omega$. In this setting, the surjective, bounded linear trace operator is simply the map
	\begin{align*}
		H^1(\Omega)\to H^1(\Omega)/H_0^1(\Omega), \,
		u\mapsto [u]= \{ u+v\,|\, v\in H_0^1(\Omega) \}.
	\end{align*} 
	$H^1(\Omega)/H_0^1(\Omega)$ can be equipped canonically with the quotient topology. We know for sufficiently regular domains $\Omega$ that $H^1(\Omega)/H_0^1(\Omega)\simeq H^{1/2}(\partial \Omega)$. 
	In the nonlocal setting the same considerations are reasonable. The following definition and the statement therein is taken from \cite{Kassmann_Foghem2022}. 
\begin{Def}[{\cite[Definition 2.29, Theorem 2.30]{Kassmann_Foghem2022} }]
	The abstract trace operator corresponding to $V^s(\Omega\,|\, \BR^d)$ is defined by 
	\begin{align*}
	\overline{\trn}: V^s(\Omega\,|\, \BR^d)&\to V^s(\Omega\,|\, \BR^d)/ V_0^s(\Omega\,|\, \BR^d)\\
	u&\mapsto [u]= \{ u+v\,|\, v\in V_0^s(\Omega\,|\, \BR^d) \}.
	\end{align*}
	The  quotient space $V^s(\Omega\,|\, \BR^d)/V_0^s(\Omega\,|\,\BR^d)$ with its natural topology given by the norm 
	\begin{equation*}
	\norm{[u]}:= \inf\{ \norm{u+v}_{V^s(\Omega\,|\,\BR^d)}\,|\, v \in V_0^s(\Omega\,|\,\BR^d) \}
	\end{equation*}
	is called the abstract trace space. It is isometrically isomorphic to  
	\begin{equation*}
	T^s(\Omega^c) := \{g : \Omega^c \to \BR \text{ measurable } \mid \text{ there exists } u \in V^s(\Omega\,|\,\BR^d) \text{ with } u|_{\Omega^c}= g \}.
	\end{equation*}
	endowed with the norm
	\begin{equation*}
	\norm{g}_{T^s(\Omega^c)}:= \inf\{ \norm{u}_{V^s(\Omega\,|\,\BR^d)}\,|\, u\in V^s(\Omega\,|\, \BR^d), u=g \text{ on } \Omega^c \}.
	\end{equation*}
\end{Def}
The next theorem shows that the space $\ST^{s}(\Omega^c)$ is norm equivalent to the abstract trace space $T^s(\Omega^c)$ robust in the limit $s\to 1-$. It complements \cite[Proposition 2.31]{Kassmann_Foghem2022}. 
	\begin{theorem}\label{th:equivalence_abstract_tracespace}
		$\ST^s(\Omega^c)$ and $T^s(\Omega^c)$ coincide and for $s_\star \in (0,1)$ there exists a constant $C = C(d,\Omega,s_\star)\ge 1$ such that for all $s\in (s_\star,1)$
		\begin{equation*}
			\tfrac{1}{C}\norm{g}_{\ST^{s}(\Omega^c)} \le \norm{g}_{T^s(\Omega^c)}\le C\, \norm{g}_{\ST^{s}(\Omega^c)}
		\end{equation*}
	for all $g\in T^s(\Omega^c)$.
	\end{theorem}
	\begin{proof}Let $g \in \ST^{s}(\Omega^c)$ and $s_\star \in (0,1)$. By \cref{th:trace_and_extension} it follows that $\extn g \in V^s(\Omega\,|\,\BR^d)$ and $(\extn g)|_{\Omega^c} = u$. Therefore, $g \in T^s(\Omega^c)$. Additionally, there exists a constant $c_1=c_1(d,\Omega)>0$ such that for all $s \in (0,1)$
		\begin{align*}
			\norm{g}_{T^s(\Omega^c)} \le \norm{\extn g}_{V^s(\Omega\,|\,\BR^d)} \le c_1 \norm{g}_{\ST^s(\Omega^c)}.
		\end{align*}
		If $g \in T^s(\Omega^c)$ then $ g = u|_{\Omega^c} = \trn(u)$ for some $ u \in V^s(\Omega\,|\,\BR^d)$. Thus, by \cref{th:trace_and_extension} there exists a constant $c_2=c_2(d,\Omega,s_\star)>0$ such that for all $s \in (s_\star,1)$ 
		\begin{align*}
			\norm{g }_{\ST^{s}(\Omega^c)} \le c_2 \norm{u}_{ V^s(\Omega\,|\,\BR^d)}.
		\end{align*}
		Since this holds true for every extension $u$ of $g$, we receive
		\begin{align*}
			\norm{g }_{\ST^{s}(\Omega^c)} \le c_2 \norm{g}_{T^s(\Omega^c)}
		\end{align*}
		and $g \in \ST^{s}(\Omega^c)$.
	\end{proof}
\section{Convergence of trace spaces}\label{sec:convergence}
In this section we prove \cref{th:convergence_pointwise}, i.e the convergence $[g,g]_{\ST^s(\Omega^c)} \to [g,g]_{H^{1/2}(\partial\Omega)}$ and $\norm{g}_{L^2(\Omega^c,\tau_{s})} \to \norm{g}_{L^2(\partial\Omega)}$ for functions $g \in H^1(\Omega^c)$ in the limit $s \to 1-$. The crucial step is to approximate the surface measure on $\partial\Omega$ by the family of measures $\frac{1-s}{d_x^s}\1_{\Omega_\rho}(x)\di x$, see \cref{lem:weak_convergence_measures} below. In \cref{sec:pointwise} we prove the convergence for functions $u\in C_b^{0,1}(\Omega^c)$, see \cref{prop:convergence_Lipschitz}, and extend the result in \cref{th:convergence_pointwise_abstract}. In \cref{sec:hilbert_space_convergence} we prove the convergence in the sense of converging Hilbert spaces introduced by Kuwae and Shioya in \cite{kuwae_mosco}. This notion of convergence is crucial for the convergence of Neumann problems is \cref{sec:neumann}. Lastly, we want to mention that the convergence of the function spaces $V^s(\Omega\,|\,\BR^d)$ to $H^1(\Omega)$ is well-known. We refer the reader to \cite[Corollary 2]{bourgain_compactness}, \cite{ponce} and \cite[Theorem 3.4, (3.5)]{kassmann_mosco}.
\subsection{Pointwise convergence}\label{sec:pointwise}
The following lemma provides an approximation of the surface measure of a compact $C^1$-submanifold. We were not able to find this result in the literature and, thus, provide the proof for the convenience of the reader. After localizing the majority of the proof is to show the convergence \eqref{eq:approx_main_result}. Thereafter, the result follows by standard arguments using an approximate identity $t\mapsto (1-s)t^{-s}$ as $s\to 1-$.
\begin{lemma}\label{lem:weak_convergence_measures}
	Let $\Omega$ be a bounded $C^1$-domain. For $r >0$ we define a family of measures $\mu_s(\dishort x):=\eta_s(x)\di x$ on $(\BR^d, \CB(\BR^d))$ via
	\begin{align*}
		\eta_s(x):= \frac{1-s}{ d_x^{s}} \1_{\Omega_{r}}(x).
	\end{align*}
	Let $\sigma$ be the surface measure on the $C^1$-submanifold $\partial \Omega$ and set $\sigma(D)=\sigma(\partial \Omega \cap D)$ for sets $D\in \CB(\BR^d)$. $\{\mu_s\}_s$ converges weakly to $\sigma$. 
\end{lemma}
\begin{proof}
	For any $f\in C_b$ and $\varepsilon>0$, 
	\begin{align*}
		\int\limits_{\Omega^{\varepsilon}} \abs{f(x)} \mu_s(\dishort x)\le \norm{f}_{L^\infty} \, \int\limits_{\Omega^{\varepsilon}\cap \Omega_{r}} \frac{1-s}{d_x^s}\di x \le \norm{f}_{L^\infty} \, \abs{\Omega^{\varepsilon}\cap \Omega_{r}} \frac{1-s}{\varepsilon^s}\to 0 \text{ as } s\to 1-. 
	\end{align*}
	This allows us to reduce the problem to the part of $\Omega_r$ close to $\partial \Omega$. Thus, the problem localizes. Without loss of generality there exists a cube $Q=(-\rho,\rho)^d$ and a $C^1$-function $\phi:\BR^{d-1}\to \BR$ such that 
	\begin{equation*}
		\Omega\cap Q = \{ (x',x_d)\,|\, x_d<\phi(x') \}\cap Q.
	\end{equation*}
	Since the boundary $\partial \Omega$ is compact, we can cover it with finitely many cubes. We choose $\varepsilon>0$ such that $\Omega_{\varepsilon}$ is covered by theses cubes.	Fix $f\in C_b(\BR^d)$. 
	\begin{align*}
		\int\limits_{Q\cap \Omega^c} f(x)\mu_s( d x) = \int\limits_{(-\rho,\rho)^{d-1}}\int\limits_{\phi(x')}^{\rho} f(x',x_d) \frac{1-s}{d_{(x',x_d)}^s} \di x_d \di x'.
	\end{align*}
	For each $(x',x_d)$ we want to express $d_x$ in terms of $\abs{x_d-\phi(x')}$. For any $(x',x_d)\in (-\rho,\rho)^{d-1}\times (\phi(x'), \rho)$ we pick $y'=y'(x',x_d)$ such that $(y', \phi(y'))$ minimizes the distance of $(x',x_d)$ to the boundary. If $(x',\phi(x'))$ minimizes the distance of $(x', x_d)$, then we always pick $y'(x',x_d)= x'$. If needed, we may choose smaller cubes to guarantee that the minimizer is in the graph of $\phi$. We fix $x'\in(-\rho,\rho)^{d-1}$. Now, we consider two cases.
	
	\textbf{Case 1: } If there exists $\tilde{x_d}\in(\phi(x'), \rho)$ such that $(x',\phi(x'))$ minimizes the distance of $(x',\tilde{x_d})$ to the boundary $\partial \Omega$, then $(x',\phi(x'))$ also minimizes the distance of $(x', x_d)$ to the boundary for all $x_d\in(\phi(x'), \tilde{x_d})$. Thereby, $x'$ is a local maximum of $\phi$ and thus $\nabla \phi(x')=0$. Therefore,
	\begin{equation*}
		d_{(x',x_d)} =  \frac{\abs{x_d-\phi(x')}}{\sqrt{1+\abs{\nabla \phi(x')}^2}}.
	\end{equation*}
	The values of $d_{(x',x_d)}$ for $x_d> \tilde{x}_d$ play no role for the convergence in \eqref{eq:convergence_proof_final_step}.
	
	\textbf{Case 2: } Here, we assume that $y'(x',x_d)\ne x'$ for every $x_d\in(\phi(x'),\rho)$. Let $\gamma=\gamma(x',x_d)$ be the angle enclosed by $\overrightarrow{(x',\phi(x'),(x',x_d)}$ and $\overrightarrow{(x',\phi(x'), (y',\phi(y'))}$, see \cref{fig:omega}.
	The angle $\gamma$ satisfies the relation 
	\begin{equation*}
		\gamma= \arctan\Big( \frac{\abs{y'-x'}}{\abs{\phi(y') - \phi(x')   }}\Big).
	\end{equation*}
	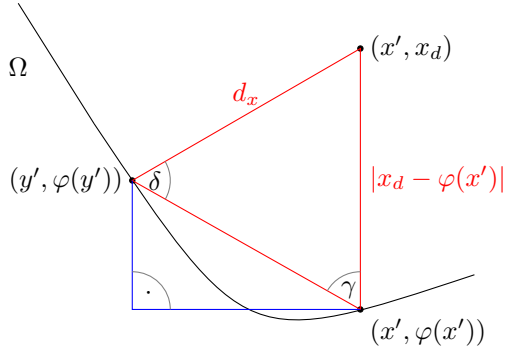
\begin{figure}[h!]
		\begin{center}
			\begin{tikzpicture}
			\draw (-2,-2) .. controls (0,-5).. (3,-4) node[at start, below, xshift=-1cm, yshift=1cm]{$\Omega$}; 
			\draw (-2,-2) -- (-3,-0.4);
			\filldraw[black] (1.5,-1) circle (1pt) node[anchor=west]{$(x',x_d)$};
			\filldraw[black] (1.5,-4.46) circle (1pt) ;
			\filldraw[black] (-1.5,-2.75) circle (1pt) ;
			\draw[red] (1.5,-4.46) node[below right]{\textcolor{black}{$(x',\varphi(x'))$}} coordinate (b) -- (1.5,-1) node[midway, right]{$\abs{x_d-\varphi(x')}$}  coordinate (c) -- (-1.5,-2.75) node[midway, above]{$d_x$} coordinate (a) node[left]{\textcolor{black}{$(y',\varphi(y'))$}} -- (1.5,-4.46) 
			pic["\textcolor{black}{$\gamma$}", draw=gray, -, angle eccentricity=0.6, angle radius=0.5cm]{angle=c--b--a}
			pic["\textcolor{black}{$\delta$}", draw=gray, -, angle eccentricity=0.6, angle radius=0.5cm]{angle=b--a--c};
			
			\draw[blue] (-1.5,-2.75) coordinate (d) -- (-1.5,-4.46) coordinate (e)-- (1.5,-4.46) coordinate (f)
			pic["\textcolor{black}{$\cdot$}",draw=gray, -, angle eccentricity=0.6, angle radius=0.5cm]{angle=f--e--d};
			\end{tikzpicture}
		\end{center}
		\caption{Geometry close to $\partial \Omega$}
		\label{fig:omega}
	\end{figure}
	We denote the angle at $(y', \phi(y'))$ enclosed by $\overrightarrow{(y',\phi(y'),(x',\phi(x')}$ and $\overrightarrow{(y',\phi(y'), (x',x_d)}$ by $\delta(\tilde{x}, x_d)$. Notice that $\abs{\overrightarrow{(y',\phi(y'), (x',x_d)}}=d_{(x',x_d)}$. Thus, we can express $d_x$ via
	\begin{equation}\label{eq:d_x_vs_vertical_distance}
		\frac{\abs{x_d-\phi(x')}}{d_x} = \frac{\sin(\delta(x', x_d))}{\sin(\gamma(x',x_d))}= \sin(\delta(x', x_d))\,\sqrt{1+\Big( \frac{\abs{\phi(y') - \phi(x')   }  } {\abs{y'-x'}} \Big)^2}.
	\end{equation}
	Now, we calculate the limit of the RHS of \eqref{eq:d_x_vs_vertical_distance}. Notice that 
	\begin{align*}
		\abs{(x',x_d)-(y',\phi(y'))}=d_{(x',x_d)}\le \abs{(x',x_d)- (x',\phi(x'))}= \abs{x_d-\phi(x')}\to 0 \text{ as } x_d \to \phi(x').
	\end{align*}
	Thereby, $\abs{y'(x',x_d) - x'}\to 0$ as $x_d\to \phi(x')$.
	
	\textbf{Claim: } Now, we prove 
	\begin{equation*}
		\frac{\abs{\phi(y'(x',x_d)) - \phi(x')   }  } {\abs{y'(x',x_d)-x'}} \to \abs{\nabla \phi(x')} \text{ as } x_d\to \phi(x').
	\end{equation*}		
	
	We begin by proving that $\phi(z')\le \phi(y')$ for all $z'\in (-\rho,\rho)^{d-1}$ satisfying $\abs{z'-x'} = \abs{y'-x'}$. We assume the contrary, \ie there exists $z'\in (-\rho, \rho)^{d-1}$ with $\abs{z'-x'}=\abs{y'-x'}$ such that $\phi(z')>\phi(y')$. Since we assumed that $(x',\phi(x'))$ does not minimize the distance of $(x',x_d)$ to $\partial \Omega$, $\phi(y')>\phi(x')$. By the continuity of $\phi$ and intermediate value theorem, there exists $w'=x'+t(z'-x')$, $t\in(0,1)$ such that $\phi(z')>\phi(w')=\phi(y')>\phi(x')$. Thereby, 
	\begin{align*}
		\abs{(w',\phi(w')) -(x',x_d)}^2 &= \abs{w'-x'}^2 + \abs{\phi(y')-x_d}^2 = t^2 \abs{y'-x'}^2 + \abs{\phi(y')-x_d}^2\\
		&< \abs{y'-x'}^2 + \abs{\phi(y')-x_d}^2 = d_{(x',x_d)}^2.
	\end{align*} 
	This is a contradiction and, thus, $\phi(z')\le \phi(y')$ for all $\abs{z'-x'}=\abs{y'-x'}$. We set $r=r(x',x_d)= \abs{y'-x'}$. Therefore, 
	\begin{align*}
		\frac{\phi(y')-\phi(x')}{\abs{y'-x'}}= \max\limits_{\substack{z'\in (-\rho,\rho)^{d-1}\\ \abs{z'-x'}=\abs{y'-x'}}} \frac{\phi(z')-\phi(x')}{\abs{z'-x'}} = \max\limits_{v\in S^{d-1}} \frac{\phi(x'+rv)-\phi(x')}{r}.
	\end{align*}
	Now, we finishing the proof of the claim. 
	\begin{align}
		\abs{\nabla \phi(x')}&= \max\limits_{v\in S^{d-1}} v \cdot \nabla \phi(x')= \max\limits_{v\in S^{d-1}}\lim\limits_{x_d\to \phi(x')+} \frac{\phi(x'+rv)-\phi(x')}{r}\nonumber\\
		&\le \lim\limits_{x_d\to \phi(x')+} \max\limits_{v\in S^{d-1}} \frac{\phi(x'+rv)-\phi(x')}{r} =  \lim\limits_{x_d\to \phi(x')}\frac{\phi(y')-\phi(x')}{\abs{y'-x'}}\nonumber\\
		&= \lim\limits_{x_d\to \phi(x')}\frac{\nabla \phi(x')\cdot (y'-x') +R(y'-x')}{\abs{y'-x'}}\le \abs{\nabla \phi(x')} + \lim\limits_{x_d\to \phi(x')+}\frac{\abs{R(y'-x')}}{\abs{y'-x'}}\nonumber\\
		&= \abs{\nabla \phi(x')}.\label{eq:convergence_mod_nabla_phi_x}
	\end{align}
	Here we used Taylor's formula. This proves the claim.
	
	Next, we prove that $\frac{y'-x'}{\abs{y'-x'}}$ converges to $\frac{\nabla \phi(x')}{\abs{\nabla \phi(x')}}$ as $x_d\to \phi(x')$. By \eqref{eq:convergence_mod_nabla_phi_x}, 
	\begin{equation*}
		1= \lim\limits_{x_d\to \phi(x')}\frac{\nabla \phi(x')}{\abs{\nabla \phi(x')}}\cdot  \frac{y'-x'}{\abs{y'-x'}}.
	\end{equation*}
	Take an arbitrary sequence $\{t_n\}_n$ such that $t_n \to \phi(x')$ as $n\to \infty$ and $t_n \in(\phi(x'), \rho)$. The sequence $\{ \frac{y'(x',t_n)- x'}{\abs{y'(x',t_n)-x'}} \}$ is bounded and thus there exists a converging subsequence, $\{ n_k \}_k$. We denote the limit by 
	\begin{equation*}
		v:= \lim\limits_{k\to \infty}  \frac{y'(x',t_{n_k})-x'}{\abs{y'(x',t_{n_k})-x'}}
	\end{equation*}
	Thereby, 
	\begin{equation*}
		1= \frac{\nabla \phi(x')}{\abs{\nabla \phi(x')}}\cdot v
	\end{equation*}
	and, thus, $v= \frac{\nabla \phi(x')}{\abs{\nabla \phi(x')}}$. Since the sequence was arbitrary, $\lim\limits_{x_d\to \phi(x')}\frac{y'(x',x_d)- x'}{\abs{y'(x',x_d)-x'}}= \frac{\nabla \phi(x')}{\abs{\nabla \phi(x')}}$.
	
	\textbf{Claim: } Now, we prove that $\delta$ converges to a right angle. By the definition of $\delta$, 
	\begin{equation}\label{eq:delta_equation}
		\cos(\delta(x', x_d))= \frac{(y'-x', \phi(y')-\phi(x'))}{\abs{(y'-x', \phi(y')-\phi(x'))}}\cdot \frac{(y'-x', \phi(y')-x_d)}{\abs{(y'-x', \phi(y')-x_d)}}.
	\end{equation}
	We consider both vectors in \eqref{eq:delta_equation} separately. Firstly, 
	\begin{align*}
		\frac{y'-x'}{\abs{(y'-x', \phi(y')-\phi(x'))}}= \frac{y'-x'}{\abs{y'-x'}} \Big( 1 + \Big( \frac{\phi(y')-\phi(x')}{\abs{y'-x'}} \Big)^2 \Big)^{-1/2}\to \frac{\nabla\phi(x')}{\abs{\nabla \phi(x')}} \Big( 1+ \abs{\nabla\phi(x')}^2 \Big)^{-1/2},
	\end{align*}
	as $x_d\to \phi(x')$. Secondly, 
	\begin{align*}
		\frac{\phi(y')-\phi(x')}{\abs{(y'-x', \phi(y')-\phi(x'))}}= \frac{\phi(y')-\phi(x')}{\abs{y'-x'}  }\Big( 1+ \Big( \frac{\abs{\phi(y')-\phi(x')}}{\abs{y'-x'} }  \Big)^2 \Big)^{-1/2}\to  \frac{\abs{\nabla \phi(x')}}{\big( 1+\abs{\nabla \phi(x')}^2 \big)^{1/2}},
	\end{align*}
	as $x_d\to \phi(x')$. Therefore, 
	\begin{equation*}
		\frac{(y'-x', \phi(y')-\phi(x'))}{\abs{(y'-x', \phi(y')-\phi(x'))}}\to \frac{ 1 }{\big( 1+\abs{\nabla \phi(x')}^2 \big)^{1/2}}\, \big( \frac{\nabla \phi(x')}{\abs{\nabla \phi(x')}} ,\abs{\nabla \phi(x')}\big).
	\end{equation*}
	The function $\phi$ is differentiable. Thereby, the outer normal vector $n_{(y',\phi(y'))}= \frac{(\nabla \phi(y'), -1)  }{\sqrt{1+\abs{\nabla \phi(y')}^2 }}$ exists. Since $(y', \phi(y'))$ minimizes the distance of $(x',x_d)$ to the boundary, we know
	\begin{equation*}
		\frac{(y'-x', \phi(y')-x_d)}{\abs{(y'-x', \phi(y')-x_d)}} = n_{(y',\phi(y'))} = \frac{(\nabla \phi(y'), -1)  }{\sqrt{1+\abs{\nabla \phi(y')}^2 }}.
	\end{equation*}
	Since $\nabla \phi$ is continuous and $y'\to x'$ as $x_d\to \phi(x')$, 
	\begin{equation}\label{eq:alternative_1}
		\frac{(y'-x', \phi(y')-x_d)}{\abs{(y'-x', \phi(y')-x_d)}}\to \frac{(\nabla \phi(x'), -1)}{\sqrt{1+\abs{\nabla \phi(x')}^2 }}.
	\end{equation}
	Together with \eqref{eq:delta_equation}, we conclude that $\delta$ converges to a right angle. 
	\begin{equation*}
		\lim\limits_{x_d\to \phi(x')} \cos(\delta(x',x_d)) = 0.
	\end{equation*}
	We combine these results with \eqref{eq:d_x_vs_vertical_distance} 
	\begin{equation}\label{eq:approx_main_result}
		\lim\limits_{x_d\to \phi(x')} \frac{\abs{x_d-\phi(x')}}{d_x}= \sqrt{1+\abs{\nabla \phi(x')}^2}.
	\end{equation}
	The sequence of functions $\{\tfrac{1-s}{(x_d-\phi(x'))^s}\}_{s}$ is an approximate identity in $x_d=\phi(x')$ as $s\to 1-$. Together with $\Big(\frac{x_d-\phi(x')}{d_{(x',x_d)}}\Big)^s\to \frac{x_d-\phi(x')}{d_{(x',x_d)}}$ in $L^1((-\rho,\rho)^{d-1})$ and Young's inequality, we conclude
	\begin{align}\label{eq:convergence_proof_final_step}
		\int\limits_{Q\cap \Omega^c} f(x)\mu_s( d x) &= \int\limits_{(-\rho,\rho)^{d-1}}\int\limits_{\phi(x')}^{\rho} f(x',x_d) \frac{1-s}{(x_d-\phi(x'))^s}\Big(\frac{x_d-\phi(x')}{d_{(x',x_d)}}\Big)^s \di x_d \di x'\nonumber\\
		&\to \int\limits_{(-\rho,\rho)^{d-1}} f(x', \phi(x')) \, \sqrt{1+\abs{\nabla \phi(x')}^2}\di x' = \int\limits_{Q\cap \partial \Omega} f(x) \sigma(\dishort x) \text{ as }s\to 1-.
	\end{align}
	Now, we combine this local result with a partition of unity to finish the proof.  
\end{proof}
\begin{proposition}\label{prop:convergence_Lipschitz}
	Let $\Omega\subset\BR^d$ be a bounded $C^1$-domain. If $f\in C_b^{0,1}(\Omega^c)$, then 
	\begin{align*}
		&\norm{f}_{L^2(\Omega^c, \tau_s)} \to \norm{ f|_{\partial\Omega}}_{L^2(\partial \Omega)},	\\
		&[f,f]_{\ST^s(\Omega^c|\Omega^c)}\to [f|_{\partial\Omega}, f|_{\partial\Omega}]_{H^{1/2}(\partial \Omega)} 
	\end{align*}
	as $ s \to 1-$.
\end{proposition}
\begin{proof}
We fix $1>r>0$ and define the measures $\mu_s(\dishort x)$ as in \cref{lem:weak_convergence_measures}. Take any arbitrary $f\in C_b^{0,1}(\Omega^c)$.

\textbf{Convergence of the $L^2$-part:} We split $ \norm{f}_{L^2(\Omega^c, \tau_s)}^2$ into $ \norm{f}_{L^2(\Omega^{r}, \tau_s)}^2$ and $ \norm{f}_{L^2(\Omega_{r}, \tau_s)}^2$. Fix $x_0\in \Omega$. Then $\abs{x-x_0}\le d_x ( 2+d_{x_0}/r+\diam{\Omega}/r)= c_1 d_x$ for $x \in \Omega^r$. The first term converges to zero because
\begin{align*}
	\norm{f}_{L^2(\Omega^{r}, \tau_s)}^2 &\le (1-s) \norm{f}_{L^\infty}^2 \int\limits_{\Omega^{r}} d_x^{-d-2s} \di x\le (1-s)c_1^{d+2s} \norm{f}_{L^\infty}^2 \int\limits_{\Omega^{r}} \abs{x-x_0}^{-d-2s} \di x\\
	&\le (1-s)c_1^{d+2s} \omega_{d-1}\norm{f}_{L^\infty}^2 \int\limits_{r}^{\infty} t^{-1-2s} \di t= (1-s)c_1^{d+2s} \omega_{d-1}\norm{f}_{L^\infty}^2\frac{r^{-2s}}{2s}\to 0 \text{ as } s\to 1-.
\end{align*}
\cref{lem:weak_convergence_measures} yields for the second term
\begin{equation*}
	\norm{f}_{L^2(\Omega_{r}, \tau_s)}^2 \ge  \int\limits_{\BR^d} f(x)^2(1+d_x)^{-d-1} \mu_s(\dishort x)\to \int\limits_{\partial \Omega} f(x)^2 \sigma(\dishort x) = \norm{f}_{L^2(\partial \Omega)}^2 \text{ as } s\to 1-.
\end{equation*}
Similarly, 
\begin{equation}
\norm{f}_{L^2(\Omega_{r}, \tau_s)}^2 \le  \int\limits_{\BR^d} f(x)^2 \mu_s(\dishort x)\to \int\limits_{\partial \Omega} f(x)^2 \sigma(\dishort x) = \norm{f}_{L^2(\partial \Omega)}^2 \text{ as } s\to 1-.
\end{equation}
\textbf{Convergence of the seminorm: }We split the integration domain into $\Omega^c\times \Omega^c = (\Omega_{r}\times \Omega_{r} )\cup (\Omega_{r} \times \Omega^{r}) \cup (\Omega^{r} \times \Omega_{r})\cup (\Omega^{r} \times \Omega^{r})$. Firstly,  
\begin{align*}
	[f,f]_{\ST^s(\Omega^{r}\,|\, \Omega^{r})  }^2 &\le  (1-s)^2 (2\norm{f}_{L^\infty})^2 \int\limits_{\Omega^{r}} \int\limits_{\Omega^{r}} d_x^{-d-2s}d_y^{-d-2s}\di y \di x\\
	&\le (1-s)^2\, c_1^{2d+4s} \omega_{d-1}^2 \, 4 \norm{f}_{L^\infty}^2 \frac{r^{-4s}}{(2s)^2}\to 0 \text{ as }s\to 1-.
\end{align*}
By symmetry the cases $(\Omega^{r} \times \Omega_{r})$ and $(\Omega_{r} \times \Omega^{r})$ are equivalent. 
\begin{align*}
	[f,f]_{\ST^s(\Omega_{r}\,|\, \Omega^{r})  }^2 &\le (1-s)^2 (2\norm{f}_{L^\infty})^2 \int\limits_{\Omega_{r}} \int\limits_{\Omega^{r}} \frac{1}{d_x^s d_y^{d+2s} }\di y \di x\\
	&\le (1-s) (2\norm{f}_{L^\infty})^2 c_1^{d+2s} \omega_{d-1} \frac{r^{-2s}}{2s} \mu_s(\BR^d).
\end{align*}
This converges to $0$ as $s\to 1-$ because $\mu_s(\BR^d)$ converges to $\sigma(\partial \Omega)<\infty$. Lastly, we consider the case $(\Omega_{r} \times \Omega_{r})$. By \cref{lem:weak_convergence_measures}, $\mu_s\otimes \mu_s$ converges weakly to $\sigma\otimes \sigma$. We define
\begin{equation*}
	h(x,y):= \frac{(f(x)-f(y))^2}{(1+d_x)^s(1+d_y)^s (\abs{x-y}+d_x+d_y+d_xd_y)^d}.
\end{equation*}
This function is neither continuous nor bounded on $\Omega^c\times \Omega^c$ and therefore \cref{lem:weak_convergence_measures} is not directly applicable. To circumvent this problem, we fix a radial, bump function $\eta \in C_c^\infty(\BR^d)$ such that $0\le \eta \le 1$, $\eta=0$ on $B_1(0)^c$, the profile of $\eta$ is monotonically decreasing and $\eta =1$ on $B_{1/2}(0)$. Now, we define $\eta_\varepsilon(x):= \eta(x/\varepsilon)$ and 
\begin{equation*}
	h_\varepsilon(x,y):= h(x,y)(1-\eta_\varepsilon(x-y)), \quad g_\varepsilon(x,y):=h(x,y)\eta_\varepsilon(x-y).
\end{equation*}
The function $h_\varepsilon$ is bounded and continuous on $\Omega^c\times \Omega^c$. Therefore, \cref{lem:weak_convergence_measures} is applicable to $h_\varepsilon$. 
\begin{align*}
	\int\limits_{\BR^d\times \BR^d} h_\varepsilon(x,y) (\mu_s \otimes \mu_s) \big(\dishort (x,y)\big)
	&\to \int\limits_{\partial \Omega\times \partial \Omega} \frac{(f(x)-f(y))^2}{\abs{x-y}^d}(1-\eta_\varepsilon(x-y)) (\sigma\otimes \sigma)(\dishort (x,y)) \text{ as } s\to 1-\\
	&\to [f,f]_{H^{1/2}(\partial \Omega)}^2 \text{ as } \varepsilon\to 0.
\end{align*}
The first limit follows from \cref{lem:weak_convergence_measures} and the second limit is a consequence of monotone convergence with $\eta_\varepsilon \to 0$ \aev as $\varepsilon\to 0$. Now we will prove that 
\begin{equation*}
	\int\limits_{\BR^d\times \BR^d} g_\varepsilon(x,y) (\mu_s \otimes \mu_s) \big(\dishort (x,y)\big)\to 0 \text{ as } \varepsilon\to 0
\end{equation*}
uniformly in $s\in(0,1)$. Just as in the proof of \cref{lem:weak_convergence_measures} the problem localizes since $\Omega$ is a bounded Lipschitz domain. We cover $\partial \Omega$ with finitely many cubes with side length $2\rho$. Without loss of generality we may assume that $r>0$ is small enough such that these cubes cover $\Omega_r$. Let $Q= (-\rho,\rho)^d$ be one of these cubes and $\phi:\BR^{d-1}\to \BR$ the $C^1$-function such that 
\begin{equation*}
	\Omega\cap Q= \{ (x',x_d)\,|\, x_d<\phi(x') \}\cap Q.
\end{equation*}
Since $\phi$ is Lipschitz continuous, a short calculation yields a constant $c_2=(1+\norm{\phi}_{C^{0,1}})>1$ such that $ c_2^{-1}\abs{x_d-\phi(x')}\le d_{(x',x_d)}\le \abs{x_d-\phi(x')}$ for any $(x',x_d)\in \Omega^c\cap Q$. We bound 
\begin{equation*}
	g_\varepsilon(x,y)\le \norm{f}_{C^{0,1}}^2\, \frac{1}{\abs{x'-y'}^{d-2}}\1_{B_{\varepsilon}(0)}(x'-y') 
\end{equation*}
for any $(x',x_d),(y',y_d)\in \Omega^c\cap Q$. Therefore, 
\begin{align*}
	&\int\limits_{(\Omega^c\times \Omega^c)\cap (Q\times Q)} g_\varepsilon(x,y) (\mu_s \otimes \mu_s) \big(\dishort (x,y)\big)\\ 
	&= \iint\limits_{(-\rho,\rho)^{2d-2}} \int\limits_{\phi(x')}^\rho \int\limits_{\phi(y')}^{\rho}  \frac{g_\varepsilon((x',x_d),(y',y_d)) (1-s)^2}{d_{(x',x_d)}^s\,d_{(y',y_d)}^s} \di y_d\di x_d \di(x',y')\\
	&\qquad\le \norm{f}_{C^{0,1}(\Omega_{\rho})}^2\,c_2^{2} \iint\limits_{(-\rho,\rho)^{2d-2}} \int\limits_{\phi(x')}^\rho \int\limits_{\phi(y')}^{\rho}   \frac{1}{\abs{x'-y'}^{d-2}}  \frac{(1-s)^2\,\1_{B_{\varepsilon}(0)}(x'-y')}{\abs{x_d-\phi(x')}^s\,\abs{y_d- \phi(y')}^s} \di y_d\di x_d \di(x',y')\\
	&\qquad\le \norm{f}_{C^{0,1}(\Omega_{\rho})}^2\,c_2^{2} \rho^{2-2s} \iint\limits_{(-\rho,\rho)^{2d-2}} \frac{1}{\abs{x'-y'}^{d-2}}\1_{B_{\varepsilon}(0)}(x'-y')  \di(x',y')\\
	&\qquad \le \norm{f}_{C^{0,1}(\Omega_{\rho})}^2\,c_2^{2} (2\rho)^{d-1}\rho^{2-2s} \int\limits_{B_\varepsilon(0)} \frac{1}{\abs{y'}^{d-2}} \di(y')= \norm{f}_{C^{0,1}(\Omega_{\rho})}^2\,c_2^{2} (2\rho)^{d-1}\rho^{2-2s} \omega_{d-2}\,\varepsilon\to 0 \text{ as }\varepsilon\to 0.
\end{align*}
The result follows from 
\begin{equation*}
	[f,f]_{\ST^s(\Omega_{\rho}\,|\, \Omega_{\rho})  }^2 =  \int\limits_{\BR^d\times \BR^d} (h_\varepsilon(x,y)+g_\varepsilon(x,y)) (\mu_s \otimes \mu_s) \big(\dishort (x,y)\big).
\end{equation*}
\end{proof}

Now we prove the convergence for functions in spaces $X$, which are uniformly embedded in $\ST^s(\Omega^c)$, have a continuous trace operator on $H^{1/2}(\partial\Omega)$ and where $C_b^{0,1}(\Omega^c)\cap X$ is dense in $X$. An example of such a space $X$ is $H^{1}(\Omega^c)$, see \cref{th:convergence_pointwise}.  

\begin{theorem}\label{th:convergence_pointwise_abstract}
	Let $\Omega$ be a bounded $C^1$-domain and $(X, \norm{\cdot }_X)$ be a space of functions $f:\Omega^c \to \BR$ with the following properties:
	\begin{enumerate}
		\item{ $(X, \norm{\cdot }_X) \hookrightarrow \big(\ST^s(\Omega^c), \norm{\cdot }_{\ST^s(\Omega^c)}\big)$ uniformly, i.e. there exists $s_0\in (0,1)$ and a constant $C >0$ such that $\norm{f}_{\ST^s(\Omega^c)} \le C \norm{f}_X$ holds for all $f\in X, \ s_0<s<1$.}
		\item{ $C_b^{0,1}(\Omega^c) \cap X$ is dense in $(X, \norm{\cdot }_X)$.}
		\item{  There exists a continuous trace operator $\tr:X \to H^{1/2}(\partial\Omega)$ such that if $f\in C(\Omega^c)$ we have $\tr f = f|_{\partial\Omega}$}.			
	\end{enumerate} 
	Then 
	\begin{align*}
		\norm{f}_{\ST^{s}(\Omega^c)} \to \norm{\tr f}_{H^{1/2}(\partial \Omega)}	,\, s \to 1^-\\
		\intertext{as well as}
		\norm{f}_{L^2(\Omega^c, \tau_s)}\to \norm{\tr f}_{L^2(\partial \Omega)}
	\end{align*}
	for all $f \in X$.
\end{theorem}
\begin{proof}
	Let $f \in X$ and $\varepsilon >0$. By assumption (1) and (2) there exists a function $g\in C_b^{0,1}(\Omega^c) \cap X$ such that $\norm{f-g}_{\ST^s (\Omega^c)} \le  C \norm{f-g}_{X} \le C\,\varepsilon$ for all $s_0 < s < 1$. Furthermore, it follows $ \norm{\tr g - \tr f}_{H^{1/2}(\partial\Omega)} \le c_1 \,\norm{g - f}_X \le c_1  \,\varepsilon$ by assumption (3). We apply \cref{prop:convergence_Lipschitz} and choose $s_0 < s_1<1$ large enough such that $\big\vert \norm{u}_{\ST^{s}(\Omega^c)} - \norm{ \tr g}_{H^{1/2}(\partial\Omega)} \big\vert \le \varepsilon$ as well as $\big\vert \norm{g}_{L^2(\Omega^c, \tau_s)} - \norm{ \tr g}_{L^2(\partial\Omega)} \big\vert \le \varepsilon$ for all $  s_1 \le s <1$.
	By triangle inequality, we conclude
	\begin{align*}
		\big\vert \norm{f}_{\ST^{s}(\Omega^c)}- \norm{\tr f}_{H^{1/2}(\partial\Omega)} \big\vert & \le \big\vert \norm{f}_{\ST^{s}(\Omega^c)}-  \norm{g}_{\ST^{s}(\Omega^c)} \big\vert + \big\vert \norm{g}_{\ST^{s}(\Omega^c)}-  \norm{\tr g}_{H^{1/2}(\partial\Omega)} \big\vert \\
		&\quad+ \big\vert\norm{\tr g}_{H^{1/2}(\partial\Omega)}-  \norm{\tr f}_{H^{1/2}(\partial\Omega)} \big\vert\\
		&\le (C+1+c_1)\, \varepsilon.
	\end{align*} 
	In a similar fashion, 
	\begin{align*}
		\big\vert \norm{f}_{L^2(\Omega^c, \tau_s)}- \norm{\tr f}_{L^{2}(\partial\Omega)} \big\vert & \le (C+1+c_1)\varepsilon.
	\end{align*}
\end{proof}
\begin{proposition}\label{prop:definition_E}
	Let $\Omega$ be a bounded Lipschitz domain. There exists an continuous extension operator $E: H^{1/2}(\partial\Omega) \to H^{1}(\Omega^c)$ and a continuous trace operator $\tr:H^{1}(\Omega^c) \to   H^{1/2}(\partial\Omega)$ such that for all $f \in C(\Omega^c)$ we have $\tr f = f|_{\partial\Omega}$ and $\tr \circ E = \id$.
\end{proposition}
We omit the proof since it is standard.
\begin{proposition}\label{example:h1_pointwise_convergence}
	Let $\Omega$ be a $C^{1,1}$-domain. Then $(H^1(\Omega^c), \norm{\cdot}_{H^1(\Omega^c)})=(X,\norm{\cdot}_X)$ is admissible in \cref{th:convergence_pointwise_abstract}.
\end{proposition}

\begin{proof}
	The existence of a trace operator follows from \cref{prop:definition_E} and the density condition is clear. We now prove that the space is uniformly embedded in $\ST^s(\Omega^c)$.
	
	Take any $f\in H^1(\Omega^c)$. There exits a continuous extension operator $\tilde{E} : H^1(\Omega^c) \to H^1(\BR^d)$ since $\Omega^c$ is Sobolev extension domain, see \eg \cite{hitchhiker}. Thus $\tilde{E}f \in H^1(\BR^d)\subset H^s(\BR^d)\subset V^s(\Omega\,|\,\BR^d)$. By \cref{th:trace_and_extension}, Sobolev embeddings, see \cite[Proposition 3.4]{hitchhiker} and the continuity of $\tilde{E}$ there exist constants $c_1, c_2, c_3>0$ independent of $s$ such that 
	\begin{align*}
		\norm{f}_{\ST^s(\Omega^c)}^2 &\le c_1^2 \big([\tilde{E}f, \tilde{E}f]_{V^s(\Omega\,|\,\BR^d)} + \norm{\tilde{E}f}_{L^2(\Omega)}^2\big) \le  c_1^2  \big(\tfrac{\kappa_{d,s}}{2}[\tilde{E}f, \tilde{E}f]_{H^s(\BR^d)} + \norm{\tilde{E}f}_{L^2(\BR^d)}^2\big)\\ 
		&\le c_2 \norm{\tilde{E}f}_{H^1(\BR^d)}^2 \le c_3 \norm{f}_{H^1(\Omega^c)}^2 .
	\end{align*}
\end{proof}
\textbf{Proof of \cref{th:convergence_pointwise}:}  \cref{example:h1_pointwise_convergence} and \cref{th:convergence_pointwise_abstract} yield the result. \qed
\begin{rem}
	By \cref{th:convergence_pointwise}, for any $g\in H^{1/2}(\partial \Omega)$ there exists $f\in \ST^s(\Omega^c)$ such that 
	\begin{equation*}
		\norm{f}_{\ST^{s}(\Omega^c)} \to \norm{g}_{H^{1/2}(\partial \Omega)}.
	\end{equation*}
	This is easily obtained via $Eg\in H^1(\Omega^c)\subset\ST^{s}(\Omega^c)$ from \cref{prop:definition_E} by \cref{example:h1_pointwise_convergence}. 
\end{rem}
\subsection{Convergence of Hilbert spaces}\label{sec:hilbert_space_convergence}
In this subsection we prove the convergence of trace spaces in the sense of converging Hilbert spaces, introduced by Kuwae and Shioya in \cite{kuwae_mosco}, see \cref{def:convergence_hilbert}. As a consequence every bounded sequence in $\ST^s(\Omega^c)$ (respectively $L^2(\Omega^c,\tau_{s})$) admits a weakly convergent subsequence to some element in $H^{1/2}(\partial\Omega)$ (respectively $L^2(\partial\Omega)$), see \cref{lem:compact_embedding}. This is crucial in \cref{sec:neumann} for the nonlocal to local convergence of Neumann problems. The definition and basic properties of this notion of convergence are summarized in \cref{sec:appendix_convergence}.
\begin{theorem}\label{th:convergence_hilbert}
	Let $\{s_n\}$ be a sequence converging to $1$ from below. 
	\begin{enumerate}
		\item  The sequence of separable Hilbert spaces $\{ L^2(\Omega^c,\tau_{s_n})\}$ converges to $L^2(\partial\Omega)$ in the sense of \cref{def:convergence_hilbert}.
		\item The sequence of separable Hilbert spaces $\{ \ST^{s_n}(\Omega^c)\}$ converges to $H^{1/2}(\partial\Omega)$ in the sense of \cref{def:convergence_hilbert}.
	\end{enumerate}
\end{theorem}

\begin{proof}
	We prove \textit{(1)} and \textit{(2)} together. In alignment with the notation in this section we set $H:=L^2(\partial \Omega)$ (respectively $H^{1/2}(\partial\Omega))$ and $H_n:=L^2(\Omega^c\,, \tau_{s_n})$ (respectively $\ST^{s_n}(\Omega^c)$). These spaces are separable Hilbert spaces. For $\ST^{s_n}(\Omega^c)$ this follows by \cref{lem:complete} and for $L^2(\Omega^c\, , \tau_{s_n})$ notice that 
	\begin{equation*}
		L^2(\Omega^c)\ni g \mapsto \tau_{s_n}^{-1} g = \frac{d_x^{s_n}(1+d_x)^{d+s_n}}{1-s_n} \, g \in L^2(\Omega^c\, , \tau_{s_n})
	\end{equation*}
	is an isometric isomorphism. In both cases let $C:= H^{1/2}(\partial \Omega)$. Since $C^{0,1}(\partial \Omega)$ is dense in $L^2(\partial \Omega)$, which follows easily by localizing, $C=H^{1/2}(\partial \Omega)$ is dense in $L^2(\partial \Omega)$. For any $n\in \BN$ we define the linear operator
	\begin{align*}
		\Phi_n: C&\to H_n\\
		C\ni g &\mapsto Eg,
	\end{align*}
	where $E$ is the extension operator from \cref{prop:definition_E}. This is a well defined map since $\Phi_n(g)\in H^1(\Omega^c)$ and $\Phi_n g \in  H_n$ for any $g\in C$ by \cref{prop:definition_E}. \cref{prop:definition_E} and \cref{th:convergence_pointwise} yield
	\begin{equation*}
		\lim\limits_{n\to \infty} \norm{\Phi_n g}_{H_n} = \lim\limits_{n\to \infty}  \norm{E g}_{H_n} =  \norm{g}_{H}.
	\end{equation*}
	Thus, $\{H_n\}$ converges to $H$ in sense of \cref{def:convergence_hilbert}.	
\end{proof}
\section{Convergence of Neumann Problems}\label{sec:neumann}
In this section we prove that solutions of nonlocal Neumann problems for integro-differential operators converge to solutions of Neumann problems for second order elliptic operators. We study two different questions. In \cref{th:convergence_neumann_mean_zero} we start with a sequence of solutions to nonlocal problems for operators $\CL_{s}$ which are comparable to $(-\Delta)^{s}$ with inhomogeneities and Neumann data given by functionals on $V^s_\perp(\Omega\,|\,\BR^d)$ and $\ST^s(\Omega^c)$. We prove, if the functionals are uniformly bounded in $s$, then a subsequence of solutions converges in $L^2(\Omega)$ and weakly in $ V^s_\perp(\Omega\,|\,\BR^d)$ to a solution of a local Neumann problem for a second order elliptic differential operator. We study more specific inhomogeneities and Neumann data in \cref{th:convergence_neumann_mean_zero_example}. In our second approach we study the reversed problem. We begin with a weak solution $u$ to a local Neumann problem for a symmetric, elliptic, second order differential operator in divergence form. Then we construct a sequence of nonlocal operators $\CL_{s}$ comparable to $(-\Delta)^s$, inhomogeneities and Neumann data such that the solutions to the nonlocal problems converge to $u$. This is done in \cref{th:convergence_neumann_mean_zero_contrary} and \cref{th:convergence_neumann_mean_zero_contrary_example}. Our results essentially use the compactness result \cref{th:local_compactness}, see \cite{bourgain_compactness}, \cite[Theorem 2.1]{ponce} and \cite[Theorem 5.76]{FoghemGounoue2020}. \medskip

We begin by introducing suitable conditions used throughout this section. Now $\Omega \subset \BR^d$ will always be a bounded $C^{1,1}$ domain. This condition is necessary, since we use the robust trace continuity from \cref{th:trace_and_extension}. Throughout this section let $J_s:\BR^d\times \BR^d\setminus\text{diag}\to (0,+\infty)$, $s\in(0,1)$, be symmetric and positive. The function $J_{s}$ is the kernel of an integral operator. In our approaches the kernel $J_s$ needs to be comparable to the kernel of the fractional Laplacian. We assume the following condition. There exists a constant $\Lambda\ge 1$ independent of $s$ such that for any $x,y\in \BR^d$ with $x\ne y$
\begin{equation}\label{eq:condition_nu_s}\tag{A}
	\tfrac{\kappa_{d,s}}{\Lambda}\abs{x-y}^{-d-2s}\le J_s(x,y)\le \Lambda \kappa_{d,s} \abs{x-y}^{-d-2s}.
\end{equation}
This condition is necessary to apply \cite[Theorem 5.76]{FoghemGounoue2020}, see \cite[p. 189, (G-E)]{FoghemGounoue2020}. The function $\kappa_{d,s}\abs{ \cdot }^{-d-2s}$ is the integral kernel of the fractional Laplacian $(-\Delta)^s$. For sufficiently regular functions $u:\BR^d\to \BR$ we define the nonlocal operator $\CL_s$ associated to $J_s$ by
\begin{equation}\label{eq:operators_Ls}\tag{$\CL_s$}
	\CL_s u(x):= \text{p.v.}\int\limits_{\BR^d} (u(x)-u(y))J_s(x,y)\di y.
\end{equation}
Additionally, the following bilinear form is connected to $\CL_s$ via a nonlocal Green-Gauß formula, see \cref{prop:greengauss}.
\begin{equation}\label{eq:E_s}\tag{$\CE_s$}
	\CE^s(u,v):= \tfrac{1}{2}\iint\limits_{(\Omega^c\times \Omega^c)^c} (u(x)-u(y))(v(x)-v(y))J_s(x,y)\di x \di y.
\end{equation}
This type of bilinear form also appeared in \cite{Energy_Valdinoci} by Servadei and Valdinoci and \cite{Felsinger2013} by Felsinger, Kassmann, Voigt. In the case $J_s (x,y) = J_s(x-y)$ it is a density of a L\'{e}vy measure and the operators $\CL_s$ are generators of associated L\'{e}vy processes. The benefit of the condition \eqref{eq:condition_nu_s} is that we can study nonlocal Neumann problems for $\CL_s$ in the Hilbert space $V^s(\Omega\,|\, \BR^d)$, because the forms $\CE^s$ and $[\cdot, \cdot]_{V^s(\Omega\,|\,\BR^d)}$ are comparable. Therefore, the results on the trace space $\ST^s(\Omega^c)$ are applicable.

Let $A(\cdot):\BR^d\to \BR^d\times\BR^d$ be a matrix valued function. The second order differential operator $u\mapsto -\dive \big( A(\cdot ) \nabla u(\cdot) \big)$ is called symmetric if $A(x)$ is a symmetric matrix for all $x\in \BR^d$ and elliptic if there exists a constant $\lambda\ge 1$ such that for all $\xi\in \BR^d$ and $x\in \BR^d$ 
\begin{equation}\label{eq:elliptic_matrix}
	\lambda^{-1}\abs{\xi}^2 \le (A(x)\xi )\cdot \xi \le \lambda \abs{\xi}^2.
\end{equation}
We define the energy associated to $-\text{div}(A(\cdot)\nabla)$ by
\begin{equation}\label{eq:E_A}\tag{$\CE^A$}
	\CE^A(u,v):= \int\limits_{\Omega} \big(A(x)\nabla u(x)\big)\cdot \nabla v(x) \di x.
\end{equation}
For kernels $J_{s}$ satisfying \eqref{eq:condition_nu_s} the bilinear forms $\CE^s$ converge to a bilinear form $\CE^A$ of a symmetric, elliptic second order differential operator in the limit $s\to 1-$, see \cref{th:local_compactness}.\medskip

Before we define nonlocal Neumann problems, we discuss Neumann problems for $-\dive\big(A(\cdot)\nabla\big)$. In sight of the Green-Gauß formula a weak solution is defined as follows.
\begin{Def}[Solution to local Neumann problems]\label{def:neumann_problem_mean_zero_local}
	Let $A(\cdot)$ satisfy \eqref{eq:elliptic_matrix}, $F\in H^1_\perp(\Omega)'$ and $G\in H^{1/2}(\partial \Omega)'$. We say $u\in H_\perp^1(\Omega)$ is a weak solution to the Neumann problem
	\begin{align*}
		-\dive\big(A(\cdot)\nabla u\big) &= F \text{ in }\Omega, \\
		\partial_{n_A}u:=n(\cdot)\cdot(A(\cdot)\nabla u) &= G \text{ on } \partial \Omega, \nonumber
	\end{align*}
	if 
	\begin{equation}\label{eq:local_neumann_mean_zero}
		\CE^A(u  ,v)= F(v)+ G(\trc v) \text{ for every } v \in H^1_{\perp}(\Omega). \tag{$N^{\text{loc}}_\perp$}
	\end{equation}
	Here $n(\cdot)$ is the outer normal vector on $\partial\Omega$ and $\trc :H^1(\Omega)\to H^{1/2}(\partial \Omega)$ is the classical trace operator. 
\end{Def}
The classical Neumann problem has often been studied, e.g we refer the reader to the book \cite{Neumann_Mikhailov} and the recent article \cite{Neumann_Droniou} by Droniou and V\'{a}zquez. Instead of $H^1_\perp(\Omega)$ we can use $H^1(\Omega)$ additional the compatibility assumption $F(1)+G(1)=0$. Then solutions are only unique up to an additive constant. In alignment with \cite{Valdinoci_nonlocal_derivative, Kassmann_Foghem2022}, we define the nonlocal normal derivative. 
\begin{Def}[Nonlocal normal derivative]
	For $s\in(0,1)$ and a domain $\Omega \subset \BR^d$ we define the nonlocal normal derivative corresponding to $\CL_s$ as
	\begin{equation*}
		\CN_su(y):= \text{p.v.}\int\limits_{\Omega} (u(x)-u(y)) J_s(x,y) \di x,\, y\in \Omega^c,
	\end{equation*}
	for any sufficiently regular, measurable function $u:\BR^d \to \BR$.
\end{Def}
As mentioned in the introduction, a similar operator has been introduced in \cite{Du_nonlocal_normal}. In analogy to the local case, the following nonlocal Green-Gauß formula holds. 
\begin{proposition}[Nonlocal Green-Gauß formula, {\cite[Theorem 4.9]{FoghemGounoue2020}}] \label{prop:greengauss}
	Assume $\Omega\subset \BR^d$ is open and bounded with Lipschitz boundary. Let $J_s$ satisfy \eqref{eq:condition_nu_s}. For every $u\in V^s(\Omega\,|\, \BR^d)$ with $\CL_s u \in L^2(\Omega)$ and any $v\in V^s(\Omega\,|\, \BR^d)$
	\begin{equation*}
		\int\limits_{\Omega} \big(\CL_s u(x)\big) v(x) \di x = \CE^{s}(u,v)- \int\limits_{ \Omega^c} \big( \CN_s u(y)\big) \, v(y) \di y.
	\end{equation*}
\end{proposition}
In sight of this formula we define solutions to nonlocal Neumann problems. This solution concept was also used in \cite{Kassmann_Foghem2022, Valdinoci_nonlocal_derivative}.
\begin{Def}[Solution to nonlocal Neumann problems]\label{def:neumann_problem_mean_zero}
	Let $s\in (0,1)$ and $J_{s}$ satisfy \eqref{eq:condition_nu_s}. Additionally, let $F_s\in V_{\perp}^s(\Omega\,|\, \BR^d)'$ and $G_s\in \ST^{s}(\Omega^c)'$. We say $u\in V_{\perp}^s(\Omega\,|\, \BR^d)$ is a weak solution to the Neumann problem
	\begin{align*}
		\CL_s u &= F_s \text{ in } \Omega,\\
		\CN_s u &= G_s \text{ on } \Omega^c,\nonumber
	\end{align*}
	if 
	\begin{align}\label{eq:neumann_mean_zero}
		\CE^{s}(u,v)= F_s(v)+ G_s(\trn v)\tag{$N_\perp$} \text{ for any $v\in V_{\perp}^s(\Omega\,|\, \BR^d)$.}
	\end{align}
\end{Def}
Analogously to the local case, we can solve the Neumann problem in the space $V^s(\Omega\, | \, \BR^d)$ instead of $V^s_\perp(\Omega\, | \, \BR^d)$ if we assume the compatibility assumption $F(1)+G(1) =0$. Then the solutions will only be unique up to an additive constant, see \cite[Theorem 4.9]{Kassmann_Foghem2022}. We also want to mention that $u \in V^s_\perp(\Omega\,|\,\BR^d)$ is a solution to \eqref{eq:neumann_mean_zero}  if and only if it minimizes the functional $v \mapsto \tfrac{1}{2}\CE^s(v,v) -F(v)-G(v)$ in $V^s_\perp(\Omega\,|\,\BR^d)$. This is proven in \cite[Proposition 4.7]{Kassmann_Foghem2022}. 

The following lemma connects linear functionals on $\ST^s(\Omega^c)$ with $\CN_{s}(V^s(\Omega\,|\,\BR^d))$. It is a minor modification of \cite[Theorem 4.10]{FoghemGounoue2020}.
\begin{lemma}[{\cite[Theorem 4.10]{FoghemGounoue2020}}]
	Let $J_s$ satisfy \eqref{eq:condition_nu_s}. For any $l\in \ST^s(\Omega^c)'$ there exists $w\in V^s(\Omega\,|\, \BR^d)$ such that for any $v\in C_c^\infty(\overline{\Omega}^c)$
	\begin{equation*}
		l(v)= \int\limits_{ \Omega^c} \CN_s w(y)\, v(y) \di y. 
	\end{equation*}
	In particular, if $g:\Omega^c\to \BR$ is a measurable function such that $l_g:= (g,\cdot )_{L^2(\Omega^c)}$ is a continuous functional on $\ST^s(\Omega^c)$, then there exists $w\in V^s(\Omega\,|\,\BR^d)$ such that $g=\CN_s w$ \aev on $\Omega^c$.
\end{lemma}
\begin{proof}
	$(\cdot, \cdot)_{L^2(\Omega)}+ \CE_s(\cdot, \cdot)$ is an equivalent inner product on $V^s(\Omega\,|\, \BR^d)$ by condition \eqref{eq:condition_nu_s}. The result follows from \cite[Theorem 4.10]{FoghemGounoue2020}.
\end{proof} 
For the convergence of solutions we need the existence of solutions to \eqref{eq:neumann_mean_zero} with a robust bound in $V^s(\Omega\, | \, \BR^d)$ for $s\to 1-$. Therefore, we recall the robust Poincaré inequality for $V^s(\Omega\,|\, \BR^d)$ proven by Foghem in \cite{FoghemGounoue2020}. This is an essential tool for proving the existence of solutions via the Lax-Milgram lemma. Furthermore, the robust Poincaré inequality and the robust trace continuity, see \cref{th:trace_and_extension}, enables us to pick the constant $C(d,\Omega, s_\star)$ in the inequality in \cref{th:solution_neumann_nonlocal} such that it only depends on a lower bound on $s$. This robustness is crucial for the convergence of solutions.
\begin{lemma}[Robust Poincaré inequality, {\cite[Corollary 5.43]{FoghemGounoue2020}}]\label{lem:poincare}
	Let $s_*\in(0,1)$ and $s\in(s_*,1)$. There exists a constant $C=C(d,\Omega,s_*)$ such that 
	\begin{align*}
		\norm{u-\fint_{\Omega}u}_{L^2(\Omega)}^2 \le C\,[u,u]_{V^s(\Omega\,|\, \BR^d)} \text{ for every $u\in V^s(\Omega\,|\, \BR^d)$}.
	\end{align*}
\end{lemma}
The next theorem proves the existence of solutions to the Neumann problem for operators $\CL_s$ which are comparable to $(-\Delta)^s$. The novelty of \cref{th:solution_neumann_nonlocal} is a uniform bound in $V^s$ depending only on a lower bound on $s$. The existence of solutions is known in the literature, we refer the reader to \cite[Theorem 3.6]{Valdinoci_nonlocal_derivative} and \cite[Theorem 4.9]{Kassmann_Foghem2022}. We apply standard techniques and pay particular attention to the independence of the constant $C=C(d,\Omega, s_\star)$ on $s$.
\begin{theorem}[Existence of solutions with a robust bound]\label{th:solution_neumann_nonlocal}
	Let $s\in(0,1)$, $G\in \ST^s(\Omega^c)'$, $F\in V_\perp^s(\Omega\,|\, \BR^d)'$ and $J_s$ satisfy \eqref{eq:condition_nu_s}. There exists a weak solution $u\in V_\perp^s(\Omega\,|\,\BR^d)$ to the problem \eqref{eq:neumann_mean_zero}. Additionally, for $s_\star \in (0,s)$, there exists a constant $C=C(d,\Omega,s_\star)>0$, such that the solution $u$ satisfies satisfying the bound
	\begin{equation*}
		\norm{u}_{V^s(\Omega\,|\, \BR^d)}\le C\big( \norm{F}_{V^s \to \BR}+ \norm{G}_{\ST^s\to \BR} \big).
	\end{equation*}
\end{theorem}
\begin{proof}
	$V_\perp^s(\Omega\,|\, \BR^d)$ is a closed subspace of $V^s(\Omega\,|\, \BR^d)$ and, thus, a separable Hilbert space. Notice that $V_\perp^s(\Omega\,|\,\BR^d) \ni \phi \mapsto G(\trn \phi)$ is a continuous, linear functional in $V_\perp^{s}(\Omega\,|\, \BR^d)'$ since $\trn: V^{s}(\Omega\,|\, \BR^d)\to \ST^s(\Omega^c)$ is linear and continuous by \cref{th:trace_and_extension}. Additionally, the bilinear form $\CE^{s}:V_\perp^s(\Omega\,|\, \BR^d)\times V_\perp^s(\Omega\,|\, \BR^d)\to \BR$ is continuous. To use the Lax-Milgram lemma it remains to show coercivity. Let $s_\star \in (0,s)$. For $u\in V_\perp^s(\Omega\,|\, \BR^d)$, by Poincaré inequality \cref{lem:poincare}, exists a constant $c_1=c_1(d,\Omega, s_\star)$ such that
	\begin{equation*}
		\CE^{s}(u,u)\ge \Lambda^{-1} [u,u]_{V^s(\Omega \, |\,\BR^d)}   \ge \tfrac{1}{2\Lambda} [u,u]_{V^s(\Omega \, |\,\BR^d)}  + \tfrac{c_1}{2\Lambda}\norm{u-\fint_{\Omega}u}_{L^2(\Omega)}\ge \frac{1\wedge c_1}{2\Lambda} \norm{u}_{V^s(\Omega\,|\, \BR^d)}^2.
	\end{equation*}
	Thus, the application of the Lax-Milgram lemma yields a unique element $u\in V_\perp^s(\Omega\,|\, \BR^d)$ such that \eqref{eq:neumann_mean_zero} is satisfied. Lastly, notice that 
	\begin{align*}
		\norm{u}_{V^s(\Omega\,|\, \BR^d)}^2 &\le \frac{2\Lambda}{1\wedge c_1}\CE^{s}(u,u)= \frac{2\Lambda}{1\wedge c_1}\big( F(u)+ G(\trn u)\big)\\
		&\le \frac{2\Lambda}{1\wedge c_1}\big( \norm{F}_{V^s_\perp\to \BR}\norm{u}_{V^s(\Omega\,|\, \BR^d)}+ \norm{G}_{\ST^s\to \BR}\, \norm{\trn u}_{\ST^s(\Omega^c)}\big)\\
		&\le \frac{2\Lambda (1+c_2)}{1\wedge c_1} \big( \norm{F}_{V^s_\perp\to \BR}+ \norm{G}_{\ST^s\to \BR}\big)\norm{u}_{V^s(\Omega\,|\, \BR^d)}.
	\end{align*}
	Here we used \cref{th:trace_and_extension} to estimate $\norm{\trn u}_{\ST^s(\Omega^c)} \le c_2 \norm{u}_{V^s(\Omega\,|\, \BR^d)}$, with $c_2=c_2(\Omega, s_\star)$. Thus, we conclude the result with $C:= \tfrac{2\Lambda c_2}{1\wedge c_1}$. 
\end{proof}
\subsection{Related literature}\label{sec:related_literature_neumann} Before we state our convergence theorems, we discuss related results in the literature. Foghem and Kassmann considered in \cite[Theorem 5.4]{Kassmann_Foghem2022} weak nonlocal Neumann problems of the form  
\begin{align*}
	(-\Delta)^s u_s &= f_s \text{ in } \Omega,\\
	\CN_s u_s &= \CN_{s}\varphi \text{ on } \Omega^c, 
\end{align*}
\ie $\CE^s(u,v)= (f_s,v)_{L^2(\Omega)} + (\CN_{s}\varphi, v)_{L^2(\Omega^c)} $ for all $v\in V^s_\perp(\Omega\,|\, \BR^d)$, with $f_s \in L^2(\Omega)$ and $\varphi \in C^2_b(\BR^d)$. Here the weak solution concept matches \eqref{eq:neumann_mean_zero}. If $\{f_s\}$ converges weakly to some $f \in L^2(\Omega)$ the authors have proven that $\{u_s\}$ converges in $L^2(\Omega)$ to a solution $u \in H^1_\perp(\Omega)$ of the local Neumann problem
\begin{align*}
	-\Delta u &= (f,\cdot)_{L^2(\Omega)} \text{ in } \Omega,\\
	\partial_n u &= (\partial_{n}\varphi, \cdot)_{L^2(\Omega^c)} \text{ on } \partial\Omega.
\end{align*}
Additionally, for all $v \in H^1(\BR^d)$
\begin{align*}
	[u_s,v]_{V^s(\Omega\,|\,\BR^d)} \rightarrow [u,v|_\Omega]_{H^1(\Omega)}, \text{    as } s \to 1-.
\end{align*} 
In \cite{FoghemGounoue2020} Foghem proved this convergence result for a larger class of integro-differential operators $L_s$. This includes the case where the integral kernel is comparable to the one of the fractional Laplacian, see \eqref{eq:condition_nu_s}. The limit function solves a Neumann problem with a second order elliptic differential operator in divergence form. The key ingredients are an asymptotic compactness result, see \cite{bourgain_compactness}, \cite[Theorem 2.1]{ponce}, \cite[Theorem 5.76]{FoghemGounoue2020}, and 
\begin{align}\label{eq:guy_converngence_normal_derivative}
	\int_{\Omega^c}\CN_{s} \varphi(x)v(x) \di x \rightarrow \int_{\partial\Omega} \partial_{n} \varphi(x)\, \trc  v(x) \di \sigma(x)
\end{align}
for $v \in H^1(\BR^d)$, see \cite[Lemma 5.75]{FoghemGounoue2020}.
Recall that $\trc  : H^1(\Omega) \to H^{1/2}(\partial\Omega)$ denotes the classical trace operator. This convergence has been proven for the fractional Laplacian in \cite[Proposition 5.1]{Valdinoci_nonlocal_derivative}. \cref{cor:nonlocal_compactness} allows us to considerably relax the assumption on the Neumann data using the convergence of the trace spaces $\ST^s(\Omega^c)$ to $H^{1/2}(\partial\Omega)$, see \cref{th:convergence_neumann_mean_zero_example} and \cref{rem:comparison_with_Guy}. 

An early approach to Neumann problems for the fractional Laplacian has been done by Dipierro, Ros-Oton and Valdinoci in \cite{Valdinoci_nonlocal_derivative}. Their solution concept is different from ours and the test space in \eqref{eq:neumann_mean_zero} depends on the Neumann data, see \cite[Definition 3.6, Equation (3.1)]{Valdinoci_nonlocal_derivative}. Therefore, they could not study the inhomogeneous problem. Regularity results for the homogeneous Neumann problem can be found in \cite[Theorem 1.1, Theorem 1.3]{Ros_Oton_Neumann_regularity} by Audrito, Felipe-Navarro and Ros-Oton. In \cite[Theorem 1.1]{Abatangelo_Neumann} by Abatangelo representations of $(-\Delta)^s u$ as a regional operator for functions $u$ satisfying $\CN_{s} u =0$ were proven. Existence theory for solutions to a Neumann problem for the fractional Schrödinger equation was done in \cite[Theorem 1.1]{Chen_Neumann_Schrodinger} by Chen. The case of the fractional $p$-Laplacian, including the discussion of eigenvalue problems, can be found in \cite{Neumann_nonlinear_1, Mungai_Neumann_p_Laplace} by Mungai and Proietti Lippi. Various nonlocal Neumann problems with nonlinearities were studied in \cite{Neumann_nonlinear} by Cinti and Colasuonno, \cite{Neumann_nonlinear_2} by Alves and Torres Ledesma and \cite{Neumann_nonlinear_3} by Bahrouni and Salort. In \cite{vu_trier} Frerick, Vollmann and Vu considered Neumann problems for a large class of symmetric and nonsymmetric, nonlocal integro-differential operators and proved Poincaré inequalities as well as well posedness results. Additionally, they studied Robin problems and proved a representation formula for solutions, which is a generalization of the work \cite{Abatangelo_Neumann}. For the higher order fractional Laplacian we refer the reader to \cite{Neumann_higher_order} by Barrios et al.. Mixed Dirichlet and Neumann problems have been studied in the context of peridynamics by Du, Tian and Zhou in \cite{Du_neumann_problems}. They also proved convergence results of solutions as the operators localize. Nonlocal diffusion equations for the regional fractional Laplacian with Neumann condition were studied in \cite{Cortazar_Neumann, Cortazar_2_Neumann} by Cortazar et al., \cite{Chaves_Neumann} by Chasseigne, Chaves and Rossi. For the spectral fractional Laplacian and related Neumann problems we refer the reader to \cite{Montefusco_neumann} by Montefusco, Pellacci and Verzini, \cite{Stinga_Neumann} by Stinga and Volzone. Deterministic reflections of the diffusion corresponding to the fractional Laplacian were considered \cite{Barles_2_Neumann, Barles_Neumann} by Barles et al.. The boundary condition $\partial_n (u/ d_x^{s-1}) =g$ on $\partial\Omega$ and $u =0$ on $\Omega^c$ was considered by Grubb in \cite{Grubb_Neumann}. A detailed discussion on these different approaches can be found in \cite[Section 7]{Valdinoci_nonlocal_derivative}.

\subsection{Convergence of nonlocal Neumann problems to local Neumann problems}
In order to prove the convergence of solutions to \eqref{eq:neumann_mean_zero} to solutions of \eqref{eq:local_neumann_mean_zero}, we need to show that if $G_s \in \ST^s(\Omega^c)'$ is given as the Neumann data in \eqref{eq:neumann_mean_zero} for all $s$, then there exists $G \in H^{1/2}(\partial\Omega)'$ such that $G_s \to G$ weakly as $s \to 1-$. In contrast to the approach in \cite[Lemma 5.75]{FoghemGounoue2020}, we use the convergence of $\ST^s(\Omega^c)$ to $H^{1/2}(\partial\Omega)$ in the sense of \cref{def:convergence_hilbert} and the compactness result \cref{lem:compact_embedding} to guarantee the existence of $G$. 
\begin{theorem}\label{th:nonlocal_compactness}
	We fix $s_\star\in(0,1)$. Let $g_{s}\in \ST^{s}(\Omega^c)$ (resp. $g_s\in L^2(\Omega^c, \tau_s)$) be a family of functions for $s\in(s_\star,1)$ such that 
	\begin{equation*}
	\sup\limits_{s\in(s_\star,1)} \norm{g_{s}}_{\ST^{s}(\Omega^c)}<\infty\quad \Big( \text{resp.} 	\sup\limits_{s\in(s_\star,1)} \norm{g_{s}}_{L^2(\Omega^c, \tau_s)}<\infty \Big).
	\end{equation*}
	There exists $g\in H^{1/2}(\partial \Omega)$ (resp. $g\in L^2(\partial \Omega)$) and a sequence $\{ s_n\}$ with $s_n \to 1-$ as $n\to \infty$ such that $\{g_{s_n}\}$ converges to $g$ weakly in sense of \cref{def:weak_convergence} w.r.t. $H_n= \ST^{s_n}(\Omega^c)$ (resp. $H_n= L^2(\Omega^c, \tau_{s_n})$) and $H=H^{1/2}(\partial \Omega)$ (resp. $H=L^2(\partial \Omega)$). In particular, for any $v\in H^1(\BR^d)$ 
	\begin{align*}
	\lim\limits_{n\to \infty} (g_{s_n},\trn v)_{\ST^{s_n}(\Omega^c)}&= (g,\tr (v|_{\Omega^c}))_{H^{1/2}(\partial \Omega)}\\
	\Big( \text{resp.} 	\lim\limits_{n\to \infty} (g_{s_n},\trn v)_{L^2(\Omega^c, \tau_{s_n})}&= (g,\tr (v|_{\Omega^c}))_{L^2(\partial \Omega)} \Big).
	\end{align*} 
\end{theorem}
\begin{proof}
	We prove both statements together. Let $H_s:=\ST^s(\Omega^c)$ (resp. $H_s:= L^2(\Omega^c, \tau_s)$) and $H:= H^{1/2}(\partial \Omega)$ (resp. $H:=L^2(\partial \Omega)$). From \cref{th:convergence_hilbert} and \cref{lem:compact_embedding} we immediately get the existence of a sequence $(s_n)$ such that $(g_{s_n})$ converges weakly to some $g \in H$ in the sense of \cref{def:weak_convergence}. Now, we fix $v \in H^1(\BR^d)$. Notice that the constant sequence $v_n := v|_{\Omega^c} $ is in $H_{s_n}$ for all $n$ by \cref{example:h1_pointwise_convergence} and converges to $\tr (v|_{\Omega^c})$ in the sense of \cref{def:strong_convergence} by \cref{th:convergence_pointwise}. This is due to the construction of $\Phi_n$ in the proof of \cref{th:convergence_hilbert} and since
	\begin{align*}
		\limsup\limits_{n \to \infty} \norm{(E \circ \tr v_n)-v|_{\Omega^c}}_{H_{s_n}} = 0 
	\end{align*}
	by \cref{prop:definition_E} and \cref{th:convergence_pointwise}. 
\end{proof}
\begin{cor}\label{cor:nonlocal_compactness}
	Fix any sequence $\{s_n\}\subset (0,1)$ that converges to $1$ and $G_n\in \ST^{s_n}(\Omega^c)'$. Suppose $\{ G_{n}(v|_{\Omega^c}) \}$ is a Cauchy sequence in $\BR$ for every $v\in H^1(\BR^d)$ and $\sup\limits_n \norm{G_{n}}_{\ST^{s_n}\to \BR}<\infty$. There exists $G\in H^{1/2}(\partial \Omega)'$ such that for any $v\in H^1(\BR^d)$ 
	\begin{equation*}
		\lim\limits_{n\to \infty} G_{n}(v|_{\Omega^c})= G(\tr v|_{\Omega^c}).
	\end{equation*}
\end{cor}
\begin{proof}
	By Riesz-representation theorem, there exists a unique $g_n\in \ST^{s_n}(\Omega^c)$ such that $G_n=(g_n,\cdot)_{\ST^{s_n}(\Omega^c)}$. The result follows from \cref{th:nonlocal_compactness}. Lastly, the original sequence converges since we assumed $\{G_n(v|_{\Omega^c})\}$ is a Cauchy sequence in $\BR$.
\end{proof}
For solutions $u_s \in V^s_\perp(\Omega\,|\,\BR^d)$ to \eqref{eq:neumann_mean_zero} the next theorem is the key tool to guarantee the existence of a limit function $u \in H^1_\perp(\Omega)$, which will be a solution to the local Neumann problem \eqref{eq:local_neumann_mean_zero}. Asymptotic compactness is crucial in the proof of this theorem, see \cite{bourgain_compactness, ponce, FoghemGounoue2020}. This theorem is a slight modification of \cite[Theorem 5.76]{FoghemGounoue2020}.
\begin{theorem}[{\cite[Theorem 5.76]{FoghemGounoue2020}}]\label{th:local_compactness}
	Let $s_\star\in(0,1)$. For any $s\in(s_\star,1)$ let $j_s:\BR^d\times \BR^d\setminus\text{diag}\to (0,\infty)$ be a kernel that satisfies \eqref{eq:condition_nu_s}. We define the symmetrization $J_s(x,y):= \tfrac{1}{2}(j_s(x,y)+j_s(y,x))$. Let $u_{s}\in V^s(\Omega\,|\, \BR^d)$ such that 
	\begin{equation*}
		\sup\limits_{s\in(s_\star,1)} \norm{u_{s}}_{V^s(\Omega\mid \BR^d)}<\infty.
	\end{equation*}
	There exists $u\in H^{1}(\Omega)$ and a sequence $\{s_n\}$ with $s_n\to 1$ as $n\to \infty$ such that $\{u_{s_n}\}$ converges to $u$ in $L^2(\Omega)$ and for any $v\in H^1(\BR^d)$ 
	\begin{equation*}
		\lim\limits_{n\to \infty} \CE^{s_n}(u_{s_n},v)= \CE^A(u, v\mid_{\Omega}), 
	\end{equation*}
	where $A(\cdot):\BR^d \times \BR^d\to \BR$ is symmetric, satisfies \eqref{eq:elliptic_matrix} and is given by $A(\cdot)=(a_{i,j}(\cdot))_{i,j}$,
	\begin{align}\label{eq:conv_A}
		a_{i,j}(x):= \lim\limits_{s\to 1-} \frac{1}{2} \int\limits_{B_\delta(0)} h_i h_j j_s(x,x+h)\di h, \,\delta>0.
	\end{align}
	The last expression is independent of $\delta$.
\end{theorem}
\begin{proof}
	Firstly, the symmetrization $J_s$ obviously satisfies \eqref{eq:condition_nu_s}. In \cite[Theorem 5.76]{FoghemGounoue2020} the result has been proven for symmetric $j_s$. Recall that $\CE^s(\cdot, \cdot)$ is equipped with the kernel $J_s$. Due to the symmetry of the double integral
	\begin{equation*}
		\CE^s(u,v)= \frac{1}{2} \iint\limits_{(\Omega^c\times \Omega^c)^c} (u(x)-u(y))(v(x)-v(y))j_s(x,y)\di y \di x.
	\end{equation*}
	Therefore, the symmetrization defines the same bilinear form. By adapting the proof of \cite[Theorem 5.69]{FoghemGounoue2020} yields for $u\in H^1(\Omega)$
	\begin{equation}\label{eq:replacement_th_5.56}
		\lim\limits_{s\to 1-}\frac{1}{2} \iint\limits_{\Omega\times \Omega} (u(x)-u(y))^2J_s(x,y)\di y \di x = \lim\limits_{s \to 1-} \frac{1}{2} \iint\limits_{\Omega\times \Omega} (u(x)-u(y))^2j_s(x,y)\di y \di x = \CE^A(u,u).
	\end{equation}
	Now, \cref{th:local_compactness} follows with the same arguments as in the proof of \cite[Theorem 5.76]{FoghemGounoue2020}, the only essential difference being that we use \eqref{eq:replacement_th_5.56} as a replacement for \cite[Theorem 5.69]{FoghemGounoue2020}.
\end{proof}
The next corollary guarantees a weak limit for the inhomogeneities in \eqref{eq:neumann_mean_zero} as $s$ approaches $1$ from below. 
\begin{cor}\label{cor:local_compactness}
	Fix any sequence $\{s_n\}\subset (0,1)$ that converges to $1$ from below and $F_n \in V_\perp^{s_n}(\Omega\,|\, \BR^d)'$. Suppose $\{F_{n}(v)\}$ is a Cauchy sequence in $\BR$ for every $v\in H^1(\BR^d)\cap L_\perp^2(\Omega)$ and $\sup\limits_n \norm{F_n}_{V_\perp^{s_n}\to \BR}<\infty$. There exists $F\in H_\perp^{1}(\Omega)'$ such that 
	\begin{equation*}
		F_n(v)\to F(v|_{\Omega})
	\end{equation*}
	for any $v\in H^1(\BR^d)\cap L_\perp^2(\Omega)$. 
\end{cor}
\begin{proof}
	By Riesz-representation theorem, there exist $f_n\in V_{\perp}^{s_n}(\Omega\,|\, \BR^d)$ such that $F_n = (f_n, \cdot)_{V^{s_n}(\Omega\,|\,\BR^d)}$. Since $\norm{f_n}_{V^{s_n}(\Omega\,|\, \BR^d)}= \norm{F_n}_{V_\perp^{s_n}\to \BR}$ is bounded in $n$, \cref{th:local_compactness} yield the existence of $f\in H^1(\Omega)$ and a subsequence $\{n_k\}$ such that $f_{n_k}\to f$ in $L^2(\Omega)$ and $(f_{n_k}, v)_{V^{s_{n_k}}(\Omega\,|\,\BR^d)} \to (f,v)_{H^1(\Omega)}$ for $v\in H^1(\BR^d)\cap L^2_\perp(\Omega)$. Since $f_{n_k}\to f$ in $L^2(\Omega)$ and $f_{n_k}\in L^2_\perp(\Omega)$, we have $f\in H^1_\perp(\Omega)$. By the Cauchy condition, the convergence holds for the original sequence. We set $F:= (f,\cdot)_{H^1(\Omega)}\in H_{\perp}^1(\Omega)'$. 
\end{proof}
Now we can prove our first convergence result.
\begin{theorem}[Convergence of Neumann Problems I]\label{th:convergence_neumann_mean_zero}
	Let $\{s_n\}$ be a sequence converging to $1$ from below and $J_{s_n}$ satisfy \eqref{eq:condition_nu_s}. Fix $G_n\in \ST^{s_n}(\Omega^c)'$ for all $n\in \BN$ such that 
	\begin{itemize}
		\item[($N_1$)]{  $\{G_n(v|_{\Omega^c})\}$ is a Cauchy sequence in $\BR$ for every $v\in H^1(\BR^d)$.				}
		\item[($N_2$)]{ $\sup\limits_{n} \norm{G_n}_{\ST^{s_n}\to \BR}<\infty $ }.
	\end{itemize}
	Let $F_n\in V_\perp^{s_n}(\Omega\,|\, \BR^d)'$ for all $n\in \BN$ such that 
	\begin{itemize}
		\item[($I_1$)]{$\{F_n(v)\} $  is a Cauchy sequence in $\BR$ for every $v\in H^1(\BR^d)\cap L_\perp^2(\Omega)$.	 }
		\item[($I_2$)]{ $\sup\limits_{n} \norm{F_n}_{V_\perp^{s_n}\to \BR}<\infty $ }.
	\end{itemize}
	There exist $F\in H_\perp^1(\Omega)'$, $G\in H^{1/2}(\partial \Omega)'$ and $u\in H_\perp^1(\Omega)$ solving \eqref{eq:local_neumann_mean_zero} with $A$ given by \eqref{eq:conv_A}. Additionally, the weak solutions $u_{n}\in V_\perp^{s_n}(\Omega\,|\, \BR^d)$ to \eqref{eq:neumann_mean_zero}, i.e. $\CL_{s_n} u_{n} = F_n$ in $\Omega$ and $\CN_{s_n} u_{n} = G_n$ on $\Omega^c$ converge to $u$ in $L^2(\Omega)$
	and
	\begin{align*}
		\CE^{s_n}(u_{n}, v) \to \CE^A(u,v|_{\Omega})
	\end{align*}
	for all $v\in H^1({\BR^d})\cap L_\perp^2(\Omega)$.
\end{theorem}
\begin{rem}
	\begin{enumerate}
		\item{In sight of \cref{th:nonlocal_compactness} and \cref{th:local_compactness} it is clear that one can drop the conditions ($N_1$) and ($I_1$) and receive the statement in \cref{th:convergence_neumann_mean_zero} for a subsequence.   }
		\item{We can consider the linear functionals $ V_\perp^s(\Omega\mid \BR^d) \ni v \mapsto (h_s, v \mid_{\Omega})_{L^2(\Omega)} $ in \cref{th:convergence_neumann_mean_zero} as inhomogeneities in \eqref{eq:neumann_mean_zero}, where $\{h_s\}$ is a bounded sequence in $L^2(\Omega)$. We set $F_s(v):= (h_s, v|_{\Omega})_{L^2(\Omega)}$ for $v\in V_\perp^{s}(\Omega\,|\, \BR^d)$. $F_s$ satisfies the condition $(N_2)$ since
			\begin{align*}
				\norm{h_s}_{L^2(\Omega)} = \sup_{ \norm{v}_{L^2(\Omega) }\le 1} \big| \int_{\Omega} h_s v  \big| \ge \sup_{\substack{ v\in V_\perp^s(\Omega\,|\, \BR^d)\\ \norm{v}_{V^s(\Omega\mid \BR^d)} \le 1}} \big| \int_{\Omega} h_s v\mid_{\Omega}  \big| = \norm{F_s}_{V_\perp^s\to \BR}.
			\end{align*} Thus the condition $(N_2)$ is more general than the assumptions on the inhomogeneity in \cite[Theorem 5.4]{Kassmann_Foghem2022}.}
	\end{enumerate}
\end{rem}

\begin{proof}[Proof of \cref{th:convergence_neumann_mean_zero}]
	Under conditions $(N_1), (N_2),(I_1), (I_2)$ it follows by \cref{cor:nonlocal_compactness} and \cref{cor:local_compactness} that there exists $G \in H^{1/2}(\partial\Omega)'$ and $F\in H_\perp^1(\Omega)'$ such that we have for all $v \in H^1(\BR^d)\cap L_\perp^2(\Omega)$
	\begin{equation*}
		\lim\limits_{n\to \infty} G_{n}(v|_{\Omega^c})= G(\tr v|_{\Omega^c})
	\end{equation*}
	as well as
	\begin{equation*}
		\lim\limits_{n\to \infty} F_n(v)= F(v|_{\Omega}).
	\end{equation*}
	Now let $u_{n} \in  V_\perp^{s_n}(\Omega\,|\,\BR^d)$ be the weak solution to the Neumann problem \eqref{eq:neumann_mean_zero} from \cref{th:solution_neumann_nonlocal}, \ie
	\begin{align*}
		\CE^{s_n}(u_{n}, v) = F_n(v) + G_n(v|_{\Omega^c})
	\end{align*}
	for all $v\in V_\perp^{s_n}(\Omega\,|\,\BR^d)$. By \cref{th:solution_neumann_nonlocal}, there exists a constant $c_1=c_1(d, \Omega)$ such that
	\begin{align*}
		\sup_{n} \norm{u_{n}}_{V^{s_n}(\Omega\,|\, \BR^d)} \le c_1\,\sup_{n} \big( \norm{F_n}_{V^{s_n}_\perp\to \BR}+ \norm{G_n}_{\ST^{s_n}\to \BR} \big) < \infty.
	\end{align*}
	By \cref{th:local_compactness}, there exists $u\in H^{1}(\Omega)$ and a subsequence $\{s_{n_k}\}$ such that for any $v\in H^1(\BR^d)$
	\begin{equation*}
		\lim\limits_{k\to \infty}\mathcal{E}^{s_{n_k}}(u_{n_k},v)= \mathcal{E}^A(u,v|_{\Omega})
	\end{equation*}
	and $\{u_{n_k}\}$ converge to $u$ in $L^2(\Omega)$. Thus, $\int_{\Omega}u=0$ since $\int_{\Omega} u_{n_k} =0$ for all $n$. Finally, we conclude for all $ v\in H^1(\BR^d)\cap L_\perp^2(\Omega)$ the equality
	\begin{align*}
		\mathcal{E}^A(u,v|_{\Omega}) =   F(v|_{\Omega})+G(\tr( v|_{\Omega^c})).
	\end{align*}
	Since $\Omega$ is a Sobolev extension domain, there exists $\tilde{v}\in H^1(\BR^d) \cap L_\perp^2(\Omega)$ for any $v\in H_\perp^1(\Omega)$ such that $\tilde{v}=v$ in $\Omega$. By construction of the traces operators $\tr, \trc $, the identity $\tr (\tilde{v}|_{\Omega^c})= \trc  v$ holds. Thus, $u$ is the unique weak solution of \eqref{eq:local_neumann_mean_zero}. To argue the convergence of the original sequence $\{u_{n}\}$, we choose an arbitrary subsequence and repeat the procedure above. The result follows from the uniqueness of the solution $u$.
\end{proof}
In applications Neumann problems are typically studied with inhomogeneities and Neumann data from $L^2$ spaces. The following theorem is a convergence result in this setup. A crucial tool is the convergence of $L^2(\Omega^c,\tau_{s})$ to $L^2(\partial\Omega)$ in sense of \cref{def:convergence_hilbert}.
\begin{theorem}[Convergence of Neumann Problems II]\label{th:convergence_neumann_mean_zero_example}
	Let $\{s_n\}$ be a sequence converging to $1$ from below, $J_{s_n}$ be symmetric kernels satisfying \eqref{eq:condition_nu_s} and $g_{n} \in L^2(\Omega^c, \tau_{s_n}^{-1})$ such that 
	\begin{enumerate}
		\item[$(N_3)$]{ $\{(g_{n}, v)_{L^2(\Omega^c)}\} $  is a Cauchy sequence in $\BR$ for every $v\in H^1(\BR^d) \cap L^2_\perp(\Omega)$, }
		\item[$(N_4)$]{  $\sup_{n} \norm{g_{n}}_{L^2(\Omega^c, \tau_{s_n}^{-1})} <\infty$.}	
	\end{enumerate}
	Furthermore, let $f_{n} \in L^2(\Omega)$ converge weakly in $L^2(\Omega)$ to some $f \in L^2(\Omega)$. There exist $g\in L^2(\partial \Omega)$ and a unique function $u\in H_\perp^1(\Omega)$ solving \eqref{eq:local_neumann_mean_zero} with $A(\cdot)$ given by \eqref{eq:conv_A}, inhomogeneity $(f,\cdot )_{L^2(\Omega)}$ and Neumann data $(g, \cdot )_{L^2(\partial\Omega)}$. Additionally, let $u_{n}\in V_\perp^{s_n}(\Omega\,|\, \BR^d)$ be the unique weak solutions to \eqref{eq:neumann_mean_zero} with $\CL_{s_n}$ having the kernel $J_{s_n}$, inhomogeneity $(f_{n}, \cdot )_{L^2(\Omega)}$ and Neumann data $  (g_{n}, \cdot )_{L^2(\Omega^c)}$. Then $u_n$ converges to $u$ in $L^2(\Omega)$ and
	\begin{align*}
		\CE^{s_n}(u_{n}, v) \to \mathcal{E}^A(u,v|_{\Omega})
	\end{align*}
	for all $v\in H^1({\BR^d})\cap L_\perp^2(\Omega)$.
\end{theorem}
Again the condition $(N_3)$ can be dropped and the statement of \cref{th:convergence_neumann_mean_zero_example} can be recovered for a subsequence.
\begin{proof}
	First notice that $(f_{n}, \cdot)_{L^2(\Omega)} \in V^{s_n}_{\perp}(\Omega\mid \BR^d)'$. Now define $h_{n}:= g_{n} \tau_{s_n}^{-1}$. Thus, 
	\begin{align*}
		(g_{n}, \cdot )_{L^2(\Omega^c)} =  	(h_{n}, \cdot )_{L^2(\Omega^c, \tau_{s_n})} \in  \ST^{s_n}(\Omega^c)'.
	\end{align*}
	The sequence $\{ h_n \}$ is bounded in $L^2(\Omega^c, \tau_{s_n})$ since $\norm{h_n}_{L^2(\Omega^c, \tau_{s_n})}=\norm{g_n}_{L^2(\Omega^c, \tau_{s_n}^{-1})}$ and $(N_4)$. By \cref{th:convergence_hilbert}, the spaces $L^2(\Omega^c, \tau_{s_n})$ converge to $L^2(\partial\Omega)$ in the sense of \cref{def:convergence_hilbert}. Hence \cref{lem:compact_embedding} and \cref{th:nonlocal_compactness} yield the existence of $g \in L^2(\partial\Omega)$ such that for all $v \in H^1(\BR^d)\cap L^2_\perp(\Omega)$
	\begin{equation}\label{eq:convergence_right_site}
		\lim\limits_{n\to \infty} (g_{n}, v)_{L^2(\Omega^c)} = \lim\limits_{n\to \infty} (h_{n}, v)_{L^2(\Omega^c, \tau_{s_n})}= (g,\tr v|_{\Omega^c})_{L^2(\partial \Omega)}.
	\end{equation}
	We don't have to consider a subsequence because of $(N_3)$. By \cref{th:solution_neumann_nonlocal}, there exists a solution $u_{n} \in  V^{s_n}_{\perp}(\Omega\mid \BR^d)$ satisfying
	\begin{align*}
		\CE^{s_n} (u_{n}, v) = (f_{n}, v)_{L^2(\Omega)} + (h_{n}, v )_{L^2(\Omega^c, \tau_{s_n})} = (f_{n}, v)_{L^2(\Omega)} + (g_{n}, v)_{L^2(\Omega^c)}
	\end{align*}
	for all $v \in V^{s_n}_{\perp}(\Omega\mid \BR^d)$. Additionally, there exists $c_1=c_1(d,\Omega) >0$ such that
	\begin{align*}
		\sup_{n} \norm{u_{n}}_{V^{s_n}(\Omega\,|\, \BR^d)}&\le c_1 \sup_{n} \big( \norm{f_{n}}_{L^2(\Omega)} + \norm{g_{n}}_{L^2(\Omega^c, \tau_{s_n}^{-1})} \big)<\infty.
	\end{align*}
	This is finite by $(N_4)$ and since $\{f_{n}\}$ is uniformly bounded in $L^2(\Omega)$. By \cref{th:local_compactness}, there exists $u\in H^{1}(\Omega)$ and a subsequence $\{s_{n_k}\}$ such that for any $v\in H^1(\BR^d)$ 
	\begin{equation*}
		\lim\limits_{k\to \infty}\mathcal{E}_{s_{n_k}}(u_{{n_k}},v)= \mathcal{E}^A(u,v|_{\Omega}).
	\end{equation*}
	Furthermore, $\{u_{{n_k}}\}$ converges to $u$ in $L^2(\Omega)$. Thus,  $u \in H^1_{\perp}(\Omega)$ since $u_n \in V^{s_n}_{\perp}(\Omega\,|\,\BR^d)$ for all $n$. Together with \eqref{eq:convergence_right_site} and the weak convergence of $\{f_{n}\}$ we conclude for all $ v\in H^1(\BR^d) \cap L^2_\perp(\Omega)$ the equality
	\begin{align*}
		\mathcal{E}^A(u,v|_{\Omega}) = (f,v|_\Omega)_{L^2(\Omega)} + (g,\tr v|_{\Omega^c})_{L^2(\partial\Omega)}.
	\end{align*}
	As in the proof of \cref{th:convergence_neumann_mean_zero}, $u$ is the unique weak solution of \eqref{eq:local_neumann_mean_zero}. To argue the convergence of the original sequence $\{u_{n}\}$, we choose an arbitrary subsequence and repeat the procedure above. The result follows from the uniqueness of the solution $u$.
\end{proof}
\begin{rem}\label{rem:comparison_with_Guy}
	In \cite[Theorem 5.4]{Kassmann_Foghem2022} the convergence of Neumann problems has been proven for $g_s := \CN_{s} \varphi$ with $\varphi \in C^2_b(\BR^d)$ and $\CL_s= (-\Delta^s)$. For $\varphi \in C_b^{0,1}(\BR^d)$ a calculation yields  $\sup_{s \in (s_\star, 1)} \norm{g_{s}}_{L^2(\Omega^c, \tau_{s}^{-1})} <\infty$. Therefore, \cite[Theorem 5.4]{Kassmann_Foghem2022} is a consequence of \cref{th:convergence_neumann_mean_zero_example}. Additionally, for $\varphi \in C^2_b(\BR^d)$ \cite[Lemma 5.3]{Kassmann_Foghem2022} implies that the limit function $g$ from \cref{th:convergence_neumann_mean_zero_example} satisfies $g = \partial_{n} \varphi \, \, \, a.e.$.
\end{rem}
\subsection{Approximation of local Neumann problems by nonlocal Neumann problems}
On the contrary, we want to approximate the solution $u$ to the local Neumann problem \eqref{eq:local_neumann_mean_zero} with inhomogeneity $F $ and Neumann boundary data $G$ by a sequence of solutions to the nonlocal problem. In sight of \cref{th:convergence_neumann_mean_zero} we need to pick kernels $J_{s}$ satisfying \eqref{eq:condition_nu_s}, inhomogeneities $ F_s \in V^{s}_{\perp}(\Omega\,|\, \BR^d)'$ and Neumann data $ G_s \in \ST^{s}(\Omega^c)'$ that satisfy the conditions $(I_1)$, $(I_2)$, $(N_1)$ and $(N_2)$. In the following proposition we give a possible choice of kernels $J_s$ that yield a given symmetric matrix satisfying \eqref{eq:elliptic_matrix}. 

\begin{proposition}[{\cite{balci_kassmann_diening_lee}}]\label{prop:elliptic_matrix_choice_of_levy_measures}
	Let $A(\cdot):\BR^d\to \BR^d\times \BR^d$ be a symmetric matrix-valued function satisfying the ellipticity condition \eqref{eq:elliptic_matrix}. Define
	\begin{equation*}
		j_s(x,x+h):= \kappa_{d,s} \abs{B(x)h}^{-d-2s} \abs{\det B(x)},
	\end{equation*}
	where $B(x):= \sqrt{A(x)^{-1}}$. Then $j_s$ satisfies \eqref{eq:condition_nu_s} and $A(\cdot)=(a_{i,j}(\cdot))_{i,j}$ where $a_{i,j}(\cdot)$ is given via \eqref{eq:conv_A}.
\end{proposition}
Notice that $j_s$ is in general nonsymmetric. We will later use the standard symmetrization $J_s(x,y)= \tfrac{1}{2}(j_s(x,y)+ j_s(y,x))$.
\begin{proof}
	Let $x\in \BR^d$. $A(x)$ is symmetric and positive definite, because $A(\cdot)$ satisfies \eqref{eq:elliptic_matrix}. Therefore, there exists $\SO(x)\in \text{SO}(d)$ and positive eigenvalues $\lambda_1(x), \dots, \lambda_d(x)$ such that $A(x)= \SO(x)M(x)^{2}\SO(x)^{-1}$ with the diagonal matrix
	\begin{equation*}
		M(x):= \begin{pmatrix}
			\sqrt{\lambda_1(x)} &  &0\\
			& \ddots &  \\
			0 & & \sqrt{\lambda_d(x)} 
		\end{pmatrix}.
	\end{equation*}
	Thus, $B(x)= \sqrt{ A(x)^{-1} }=  \SO(x)^{-1}M(x)^{-1}\SO(x)$. By \eqref{eq:elliptic_matrix}, there exists $\lambda\ge 1$ such that $\lambda^{-1}\le \lambda_i(x)\le \lambda$ for all $i$. Hence, \begin{align}
		\abs{\det B(x)}&= \frac{1}{\sqrt{\lambda_1(x)\cdots \lambda_d(x)}}\in[ \lambda^{-d/2}, \lambda^{d/2} ],\nonumber\\
		\lambda^{-1/2} \abs{\xi}&\le \abs{B(x)\xi}\le \lambda^{1/2} \abs{\xi}, \,\xi \in \BR^d.\label{eq:B(x)_bounds}
	\end{align} Thereby, $j_s$ satisfies \eqref{eq:condition_nu_s} with the constant $\Lambda= \lambda^{d+s}$. We define $D(x):=(d_{i,j}(x))_{i,j}:= B(x)^{-1}$, $r= 1/\sqrt{\lambda}$ and $R:= \sqrt{\lambda}$. Notice that \eqref{eq:B(x)_bounds} yields $B_r(0)\subset B(x)B_1(0) \subset B_R(0)$. It remains to prove $A(\cdot ) = (a_{i,j}(\cdot ))_{i,j}$. 
	\begin{align*}
		\frac{1}{2}\int\limits_{B_1(0)} h_i h_j j_s(x,x+h)\di h &=\frac{\kappa_{d,s}}{2}\int\limits_{B(x)B_1(0)} \frac{(B(x)^{-1} y )_i \,(B(x)^{-1} y )_j }{\abs{y}^{d+2s}} \di y\\
		=\frac{\kappa_{d,s}}{2}\int\limits_{B_{r}(0)  } \frac{(D(x) y )_i \,(D(x) y )_j }{\abs{y}^{d+2s}} \di y& +  \frac{\kappa_{d,s}}{2}  \int\limits_{B(x)B_1(0)\setminus B_{r}(0)  }  \frac{(D(x) y )_i \,(D(x) y )_j }{\abs{y}^{d+2s}} \di y =: (I)+ (II).
	\end{align*}
	We consider $(I)$ and $(II)$ separately. For $(I)$ notice
	\begin{align*}
		(I)= \sum_{k,l=1}^d d_{k,i}(x)\,d_{l,j}(x) \frac{\kappa_{d,s}}{2}\underbrace{\int\limits_{B_{r}(0)  } \frac{ y_k  \,y_l }{\abs{y}^{d+2s}} \di y}_{=:E_{l,k}}.
	\end{align*}
	By symmetry, $E_{l,k}=0$ for $l\ne k$. For $l=k$
	\begin{equation*}
		E_{k,k}= \tfrac{1}{d}\int\limits_{ B_r(0) } \abs{y}^{-d+2(1-s)} \di y = \tfrac{\omega_{d-1}}{d}\int\limits_{0}^r t^{1-2s}\di t = \frac{\omega_{d-1}}{2d(1-s)} r^{2-2s}.
	\end{equation*}
	Thus, by \cref{prop:asymptotics_constant} we find, since $A(x)$ is symmetric and therefore $D(x)$ is symmetric, 
	\begin{equation*}
		\lim\limits_{s\to 1-} (I)= \sum_{k=1}^d d_{k,i}(x)d_{k,j}(x) = \big(D(x)\cdot D(x)\big)_{i,j}= \big(A(x) \big)_{i,j}.
	\end{equation*}
	It is left to show that $(II)$ converges to $0$ as $s\to 1-$. 
	\begin{equation*}
		\abs{(II)}\le \frac{\kappa_{d,s}}{2}\lambda \omega_{d-1} \int\limits_{r}^R t^{-1+2(1-s)}\di t = \frac{\kappa_{d,s}}{4(1-s)}\lambda \omega_{d-1} \big( R^{2-2s}-r^{2-2s} \big)\to 0, \text{ as } s\to 1-.
	\end{equation*}
\end{proof}
\begin{theorem}[Convergence of Neumann Problems III]\label{th:convergence_neumann_mean_zero_contrary}
	Let $A(\cdot):\BR^d\to \BR^d\times\BR^d$ be symmetric satisfying \eqref{eq:elliptic_matrix}, $G\in (H^{1/2}(\partial\Omega))'$ and $F\in H_\perp^{1}(\Omega)'$. Let $u \in H_{\perp}^1(\Omega)$ be the weak solution of \eqref{eq:local_neumann_mean_zero}. There exist symmetric kernels $J_s$ satisfying \eqref{eq:condition_nu_s}, $G_s \in \ST^{s}(\Omega^c)'$ and $F_s \in V^s_{\perp}(\Omega \,|\, \BR^d)'$ such that the unique weak solutions $u_s \in V^s_{\perp}(\Omega \,|\, \BR^d)$ to \eqref{eq:neumann_mean_zero} satisfy the following. The solutions $u_{s}$ converge to $u$ in $L^2(\Omega)$ and 
	\begin{align*}
		\CE^{s}(u_{s}, v) \to \CE^A(u,v|_{\Omega}), s\to 1-,
	\end{align*}
	for all $v\in H^1({\BR^d})\cap L_{\perp}^2(\Omega)$.
\end{theorem}
\begin{proof}
	Let $u \in H_{\perp}^1(\Omega)$ be the solution to \eqref{eq:local_neumann_mean_zero}. By the Riesz-representation theorem, there exists  $g \in H^{1/2}(\partial\Omega)$ such that $G= (g,\cdot)_{H^{1/2}(\partial \Omega)}$. The extension $E$ from \cref{prop:definition_E} yields $Eg\in H^1(\Omega^c)$. Again Riesz-representation theorem yields a unique $f\in H_\perp^1(\Omega)$ such that $F=(f,\cdot)_{H^1(\Omega)}$. Since $\Omega$ is a Sobolev extension domain, we fix an extension $\overline{f} \in H^1(\BR^d)\cap L^2_\perp(\Omega)$, see \eg \cite{hitchhiker}. By Sobolev embeddings, see \cite[Proposition 3.4]{hitchhiker}, the extension satisfies $\overline{f}\in V^s_\perp(\Omega\,|\, \BR^d)$. We set $G_s:= (Eg, \cdot )_{\ST^{s}(\Omega^c)} \in \ST^s(\Omega^c)'$ as well as $F_s := (\overline{f}, \cdot )_{  V^s(\Omega \,\mid\, \BR^d)  }  \in V^s_\perp(\Omega \,|\, \BR^d)'$. Let $j_s$ be the kernel from \cref{prop:elliptic_matrix_choice_of_levy_measures} and define the standard symmetrization $J_s(x,y):= \tfrac{1}{2}(j_s(x,y)+j_s(y,x))$. By \cref{th:solution_neumann_nonlocal}, there exist unique solutions $u_s \in V^s_\perp(\Omega \,|\, \BR^d)$ to \eqref{eq:neumann_mean_zero}, \ie $\CL_{s}$ is equipped with the kernel $J_s$, $\CL_{s}u_s=F_s$ in $\Omega$ and $\CN_{s}u_s=G_s$ on $\Omega^c$. By the convergence of the norms $\norm{\cdot}_{V^s(\Omega\,|\,\BR^d)}\to \norm{\cdot}_{H^1(\Omega)}$, see \cite[Corollary 2]{bourgain_compactness}, \cite{ponce} and \cite[Theorem 3.4, (3.5)]{kassmann_mosco}, we conclude for every sequence $s_n \to 1$ and all $v \in H^1(\BR^d)\cap L_\perp^2(\Omega)$ the convergence
	\begin{align*}
		F_{s_n}(v)&=(\overline{f},v)_{V^{s_n}(\Omega\,|\, \BR^d)}= \tfrac{1}{4} \Big( \norm{\overline{f}+v}^2_{V^{s_n}(\Omega\,|\, \BR^d)}-\norm{\overline{f}-v}^2_{V^{s_n}(\Omega\,|\, \BR^d)} \Big) \\
		&\to \tfrac{1}{4}\Big( \norm{f+v|_{\Omega}}_{H^1(\Omega)}^2-\norm{f-v|_{\Omega}}_{H^1(\Omega)}^2 \Big)= (f,v|_\Omega)_{H^1(\Omega)} = F(v|_{\Omega})
	\end{align*}
	and by \cref{th:convergence_pointwise}
	\begin{align*}
		G_{s_n}(v|_{\Omega^c})&= (Eg,v|_{\Omega^c})_{\ST^{s_n}(\Omega^c)}=\tfrac{1}{4} \Big( \norm{Eg + v|_{\Omega^c}}^2_{\ST^{s_n}(\Omega^c)}-\norm{Eg-v|_{\Omega^c}}^2_{\ST^{s_n}(\Omega^c)} \Big)  \\
		&\to \tfrac{1}{4}\Big( \norm{g+\tr (v|_{\Omega^c})}_{H^{1/2}(\partial \Omega)}^2 - \norm{g-\tr (v|_{\Omega^c})}_{H^{1/2}(\partial \Omega)}^2 \Big)=(g,\tr (v|_{\Omega^c}))_{H^{1/2}(\partial\Omega)}\\
		&= G(\tr (v|_{\Omega^c}))=G(\trc  (v|_{\Omega})).
	\end{align*}
	The last equality follows from the construction of the trace operators $\tr, \trc $ just as in \cref{th:convergence_neumann_mean_zero}. Thereby, for every $v\in H^1(\BR^d)\cap L^2_\perp(\Omega)$
	\begin{align*}
		\CE^{s_n}(u_{s_n},v)= F_{s_n}(v)+ G_{s_n}(v|_{\Omega^c}) \to F(v|_{\Omega^c})+ G(\trc (v|_{\Omega}))= \CE^A(u,v|_{\Omega}).
	\end{align*}
	By Sobolev embeddings, see \eg \cite[Proposition 3.4]{hitchhiker}, there exists a constant $c_1\ge 1$ such that the inequality $\norm{\overline{f}}_{V^{s_n}(\Omega\,|\,\BR^d)}\le c_1\, \norm{\overline{f}}_{H^1(\BR^d)}$ holds for all $n$. Additionally, \cref{prop:definition_E}, \cref{example:h1_pointwise_convergence} yield a constant $c_2\ge 1$ such that $\norm{Eg}_{\ST^{s_n}(\Omega^c)}\le c_2 \norm{Eg}_{H^1(\Omega^c)}\le c_2^2 \norm{g}_{H^{1/2}(\partial \Omega)}$. 
	By \cref{th:solution_neumann_nonlocal}, there exists a constant $c_3\ge 1$ such that the solutions $u_{s_n}$ satisfy the bound  
	\begin{align*}
		\norm{u_{s_n}}_{V^{s_n}(\Omega\,|\, \BR^d)}&\le c_3\big( \norm{F_{s_n}}_{V_\perp^{s_n}\to \BR}+ \norm{G_{s_n}}_{\ST^{s_n}\to \BR} \big)\\
		= c_3\big( \norm{\overline{f}}_{V^{s_n}(\Omega\,|\,\BR^d)} + \norm{Eg}_{\ST^{s_n}(\Omega^c)} \big)&\le c_3\Big( c_1 \norm{\overline{f}}_{H^1(\BR^d)} + c_2^2 \norm{g}_{H^{1/2}(\partial \Omega)}\Big)
	\end{align*}
	for all $n$. By \cref{th:local_compactness} and \cref{prop:elliptic_matrix_choice_of_levy_measures}, there exists  $u' \in H_{\perp}^1(\Omega)$ such that $ u_{s_n} $ converges to $u'$ in $L^2(\Omega)$ and for every $v\in H^1(\BR^d)\cap L_\perp^2(\Omega)$	
	\begin{align*}
		\mathcal{E}_{s_n}(u_{s_n},v) \to \mathcal{E}^A(u',v|_{\Omega}).
	\end{align*}
	Thus, $\CE^A(u',v|_{\Omega}) = \CE^A(u,v|_{\Omega})$ for every $v\in H^1(\BR^d)\cap L^2_\perp(\Omega)$. Plugging in $ v:= u-u'$ yields $u =u'$. Note that this is possible, since $u-u'$ has an extension in $H^1(\BR^d)\cap L_\perp^2(\Omega)$. 
\end{proof}		
Our last convergence theorem is in the spirit of \cref{th:convergence_neumann_mean_zero_contrary} but we consider more specific inhomogeneities and Neumann data. 
\begin{theorem}[Convergence of Neumann Problems IV]\label{th:convergence_neumann_mean_zero_contrary_example}
	Let $A(\cdot):\BR^d\to \BR^d\times\BR^d$ be a symmetric matrix-valued function satisfying \eqref{eq:elliptic_matrix}, $g\in L^2(\partial\Omega)$ and $f \in L^2(\Omega)$. Additionally, let $u \in H_{\perp}^1(\Omega)$ be the weak solution to \eqref{eq:local_neumann_mean_zero} with inhomogeneity $(f, \cdot)_{L^2(\Omega)}$ and Neumann data $(g, \cdot )_{L^2(\partial\Omega)}$. There exists a sequence $\{s_n\}$, $s_n\to 1-$, symmetric kernels $J_{s_n}$ satisfying \eqref{eq:condition_nu_s} and a sequence of functions $g_{n} \in L^2(\Omega^c, \tau_{s_n}^{-1})$ such that the following holds. The sequence of weak solution $u_{n} \in V_{\perp}^{s_n}(\Omega \mid \BR^d)$ to \eqref{eq:neumann_mean_zero} with inhomogeneity $(f,\cdot)_{L^2(\Omega)}$ and Neumann data $(g_{n}, \cdot )_{L^2(\Omega^c)}$ converges to $u$ in $L^2(\Omega)$ and
	\begin{align*}
		\CE^{s_n}(u_{n}, v) \to \CE^A(u,v|_{\Omega})
	\end{align*}
	for all $v\in H^1({\BR^d})\cap L_\perp^2(\Omega)$.
\end{theorem}
\begin{proof}
	We begin by constructing an appropriate sequence of Neumann data. Fix a sequence $\{h_n\}$ in $H^{1/2}(\partial\Omega)$ converging to $g$ in $L^2(\partial\Omega)$. Using the operator $E$ from \cref{prop:definition_E} we define
	\begin{align*}
		h^n_s := \tau_s E h_n.
	\end{align*}
	By \cref{prop:definition_E}, \cref{example:h1_pointwise_convergence} and \cref{th:trace_and_extension}, there exists a constant $c_1\ge 1$ such that for any $v \in L^2(\Omega^c, \tau_s)$
	\begin{align*}
		\abs{	(h^n_s,v)_{L^2(\Omega^c)}} &\le  \norm{Eh_n}_{L^2(\Omega^c, \tau_s)} \norm{v}_{L^2(\Omega^c, \tau_s)} \le c_1\norm{Eh_n}_{H^1(\Omega^c)} \norm{v}_{L^2(\Omega^c, \tau_s)} \\
		&\le  c_1^2\norm{h_n}_{H^{1/2}(\partial\Omega)} \norm{v}_{L^2(\Omega^c, \tau_s)}.
	\end{align*}
	Thus, $(h^n_s,\cdot )_{L^2(\Omega^c)} \in L^2(\Omega^c, \tau_s)'$. Our next goal is to find a diagonal sequence $\{ s_n \}$ such that $Eh_n=\tau_{s_n}^{-1}\,h_{s_n}^n\in L^2(\Omega^c, \tau_{s_n})$ converges to $g\in L^2(\partial \Omega)$ in the sense of \cref{def:strong_convergence}. In the notation of \cref{th:convergence_hilbert}, \cref{example:h1_pointwise_convergence} the following is true. For any $n,m\in \BN$ there exists $s_{n,m}\in (0,1)$ such that 
	\begin{equation*}
		\abs{\norm{\Phi_{s}h_m - Eh_n}_{L^2(\Omega^c, \tau_{s})} - \norm{h_m-h_n}_{L^2(\partial \Omega)} }= \abs{\norm{Eh_m - Eh_n}_{L^2(\Omega^c, \tau_{s})} - \norm{h_m-h_n}_{L^2(\partial \Omega)} }\le \frac{1}{n}
	\end{equation*} for all $s\in [s_{n,m},1)$. Therefore, for any $n\in \BN$ there exists $s_n=\max\{ s_{n,1},\dots, s_{n,n}, 1-1/n \}\in(0,1)$ such that 
	\begin{equation*}
		\abs{\norm{\Phi_{s_n}h_m - Eh_n}_{L^2(\Omega^c, \tau_{s_n})} - \norm{h_m-h_n}_{L^2(\partial \Omega)} }= \abs{\norm{Eh_m - Eh_n}_{L^2(\Omega^c, \tau_{s_n})} - \norm{h_m-h_n}_{L^2(\partial \Omega)} }\le \frac{1}{n}
	\end{equation*}
	for all $m\le n$. Notice that $s_n \to 1-$ as $n\to \infty$. In particular, for any $m\in \BN$
	\begin{align*}
		&\limsup\limits_{n\to \infty} \abs{\norm{\Phi_{s_n}h_m - Eh_n}_{L^2(\Omega^c, \tau_{s_n})} - \norm{h_m-g}_{L^2(\partial \Omega)} }\\
		&\qquad\le \limsup\limits_{n\to \infty} \abs{\norm{\Phi_{s_n}h_m - Eh_n}_{L^2(\Omega^c, \tau_{s_n})} - \norm{h_m-h_n}_{L^2(\partial \Omega)} } + \abs{\norm{h_m-h_n}_{L^2(\partial \Omega)}- \norm{h_m-g}_{L^2(\partial \Omega)}}\\
		&\qquad\le \limsup\limits_{n\to \infty}\frac{1}{n} + \norm{h_n-g}_{L^2(\partial \Omega)}=0.
	\end{align*}
	Since $h_m\to g$ in $L^2(\partial \Omega)$ and $\lim\limits_{m \to \infty} \limsup\limits_{n\to \infty} \norm{\Phi_{s_n}h_m - Eh_n}_{L^2(\Omega^c, \tau_{s_n})} = \lim\limits_{m \to \infty}\norm{h_m-g}_{L^2(\partial \Omega)} = 0$, 
	the sequence $\{Eh_n\}$ converges to $g$ in the sense of \cref{def:strong_convergence}. By \cref{prop:eigenschaften_konvergenz}, $\{ Eh_n \}$ converges weakly to $g$ in sense of \cref{def:weak_convergence} and $\sup_n \norm{Eh_n}_{L^2(\Omega^c,\tau_{s_n})}<\infty$. In particular, for any $v\in H^1(\Omega^c)$
	\begin{align}
		\lim\limits_{n \to \infty}	(h^n_{s_n},v )_{L^2(\Omega^c)} &=\lim\limits_{n \to \infty} (Eh_n,v )_{L^2(\Omega^c,\tau_{s_n})} =  (g, \tr (v|_{\Omega^c}))_{L^2(\partial\Omega)},\label{eq:convergence_RHS}\\
		\sup\limits_{n}\norm{h^n_{s_n}}_{L^2(\Omega^c, \tau_{s_n}^{-1})} &= \sup\limits_n \norm{Eh_n}_{L^2(\Omega^c, \tau_{s_n})} <\infty.\label{eq:bound_RHS}
	\end{align} 
	We define $g_n:= h_{s_n}^n\in L^2(\Omega^c,\tau_{s_n}^{-1})$. Recall that $(g_n,\cdot)_{L^2(\Omega^c)}\in \ST^{s_n}(\Omega^c)'$. Let $J_{s_n}$ be the kernels from \cref{prop:elliptic_matrix_choice_of_levy_measures}. By \cref{th:solution_neumann_nonlocal}, there exist unique solutions $u_{n} \in V_\perp^{s_n}(\Omega\mid \BR^d)$ to
	\begin{align*}
		\CL_{s_n} u_{n} &= (f, \cdot )_{L^2(\Omega)} \text{ in } \Omega,\\
		\CN_{s_n} u_{n} &=  (g_n, \cdot )_{L^2(\Omega^c)} \text{ on }\Omega^c.
	\end{align*}
	for all $n$. Additionally, there exists a constant $c_2=c_2(d,\Omega)\ge 1$ such that for all $n\ge n_0$
	\begin{align*}
		\norm{u_{n}}_{V^{s_n}(\Omega\mid \BR^d)} \le c_2 \big( \norm{f}_{L^2(\Omega)} + \norm{g_n}_{L^2(\Omega^c, \tau_{s_n}^{-1})} \big).
	\end{align*} 
	By \eqref{eq:bound_RHS}, $\norm{u_{n}}_{V^{s_n}(\Omega\mid \BR^d)}$ is bounded in $n$. We fix $v\in H^1(\BR^d) \cap L^2_\perp(\Omega)$. By the construction of the traces $\tr, \trc $, it holds $\trc  (v |_\Omega) = \tr (v|_{\Omega^c})$. Therefore, 
	\begin{align*}
		\CE^{s_n}(u_{n},v)= (f, v|_{\Omega})_{L^2(\Omega)} + (g_n, v|_{\Omega^c})_{L^2(\Omega^c)}\to  (f, v|_{\Omega})_{L^2(\Omega)} + (g,\trc  ( v|_{\Omega}))_{L^2(\Omega^c)} = \CE^A(u,v|_{\Omega}).
	\end{align*}
	It remains to prove that $u_n$ converges to $u$ in $L^2(\Omega)$. By \cref{th:local_compactness} and \cref{prop:elliptic_matrix_choice_of_levy_measures}, there exists  $u' \in H_\perp^1(\Omega)$ such that $ u_{n} $ converges to $u'$ in $L^2(\Omega)$ and for every $v\in H^1(\BR^d)\cap L_\perp^2(\Omega)$
	\begin{align*}
		\mathcal{E}_{s_n}(u_{n},v) \to \mathcal{E}^A(u',v|_{\Omega}).
	\end{align*}
	Thus, we conclude $\CE^A(u',v|_{\Omega}) = \CE^A(u,v|_{\Omega})$ for every $v\in H^1(\BR^d)\cap L_\perp^2(\Omega)$. Plugging in $ v:= u-u'$ yields $u -u'=\text{constant}$ in $\Omega$. Since $u,u'\in L^2_{\perp}(\Omega)$, $u -u'=0$. Note that this is possible since $u-u'$ has an extension in $H^1(\BR^d) \cap L^2_\perp(\Omega)$. Therefore, $u_n \to u$ in $L^2(\Omega)$. 
\end{proof}
\begin{rem}
	For local Neumann problems \eqref{eq:local_neumann_mean_zero} it is also common to work with the Hilbert space $H^1(\Omega)$ instead of $H^1_\perp(\Omega)$. Since $H^1(\Omega)$ contains nonzero constant functions, the Neumann problem has a solution only if we assume the additional compatibility assumption $F(1)+G(1)=0$. In this case the solutions are only unique up to an additive constant. Analogously, we can consider weak solutions to \eqref{eq:neumann_mean_zero} in the space $V^s(\Omega\, | \, \BR^d)$ instead of $V^s_\perp(\Omega\, | \, \BR^d)$. Just as in the local setting, we have to assume the same compatibility assumption for solutions to exist and they will only be unique up to an additive constant. In this setup similar convergence results can be proven. This has been done in \cite[Theorem 5.78]{FoghemGounoue2020} and \cite[Chapter 4]{Kassmann_Foghem2022}.
\end{rem}

\appendix
\section{}
The coarea formula is an important tool in our proofs. We recall it here for the convenience of the reader. 
\begin{theorem}[{Coarea formula, \cite[Theorem 3.2.12]{coarea}}]\label{th:coarea}
	Let $D\subset\BR^d$ be an open set, $f:D\to \BR$ Lipschitz continuous and $g\in L^1(D)$. The following equation holds.
	\begin{equation*}
		\int\limits_{D} g(x) \abs{\nabla f(x)}\di x = \int\limits_{\BR} \Big( \int\limits_{f^{-1}(t)} g(x) \di \,\SH(x)  \Big) \di t.
	\end{equation*}
\end{theorem} 

\section{}\label{sec:appendix_convergence}
Here we give the definition and basic properties of convergent Hilbert spaces introduced in \cite{kuwae_mosco} by Kuwae and Shioya.

\begin{Def}[{\cite[Section 2.2]{kuwae_mosco}}]\label{def:convergence_hilbert}
	Let $ H_n $ and $H$ be real Hilbert spaces. We say that $\{ H_n \}$ converges to $H$ if there exists a dense subspace $C\subset H$ and a sequence of linear operators $\Phi_n :C\to H_n$ with 
	\begin{equation*}
	\lim\limits_{n\to \infty} \norm{\Phi_n u}_{H_n}= \norm{u}_H \text{ for every } u \in C.
	\end{equation*}
\end{Def}
\begin{Def}[Strong convergence {\cite[Definition 2.4]{kuwae_mosco}}]\label{def:strong_convergence}
	Let $\{ H_n \}$ converge to $H$ in the sense of \cref{def:convergence_hilbert}. We say that a sequence of vectors $\{u_n\}$, $u_n \in H_n$ converges strongly to a vector $u\in H$ if there exists a sequence $\{ \tilde{u}_m\}\subset C$ such that 
	\begin{align*}
	&\lim\limits_{m\to \infty}\norm{\tilde{u}_m - u}_H =0,\\
	&\lim\limits_{m\to \infty} \limsup\limits_{n\to \infty}\norm{\Phi_n \tilde{u}_m - u_n}_{H_n}=0.
	\end{align*}
\end{Def}
\begin{Def}[Weak convergence {\cite[Definition 2.5]{kuwae_mosco}}]\label{def:weak_convergence}
	Let $\{ H_n \}$ converge to $H$ in the sense of \cref{def:convergence_hilbert}. We say that a sequence of vectors $\{u_n\}$, $u_n \in H_n$ converges weakly to a vector $u\in H$ if 
	\begin{equation*}
	\lim\limits_{n\to \infty}(u_n,v_n)_{H_n}= (u,v)_H
	\end{equation*}
	for every sequence $\{ v_n \}$, $v_n\in H_n$ strongly converging to $v\in H$.
\end{Def}

\begin{proposition}[{\cite[Lemma 2.1, Lemma 2.3]{kuwae_mosco}}]\label{prop:eigenschaften_konvergenz}
	Let $\{H_n\}$ converge to $H$, $\{u_n\}$ be a sequence with $u_n \in H_n$ for all $n \in \BN$ and $u \in H$.
	\begin{enumerate}
		\item If $\{u_n\}$ converges strongly to $u$, then it converges weakly to $u$.
		\item If $\{u_n\}$ is weakly convergent to $u$, then $\sup_{n \in \BN} \norm{u_n}_{H_n} < \infty$.
	\end{enumerate}
\end{proposition}
For the nonlocal to local convergence of Neumann problems the next Lemma is important. It states the weak compactness of the ball in the disjoint union of all $H_n$.

\begin{lem}[{\cite[Lemma 2.2]{kuwae_mosco}}]\label{lem:compact_embedding}
	Let $\{H_n\}$ be a sequence of separable Hilbert spaces converging to a separable Hilbert space $H$ in sense of \cref{def:convergence_hilbert}. Let $\{u_n\}$ with $u_n \in H_n$ be a sequence such that $\{ \norm{u_n}_{H_n} \}$ is bounded. There exists a subsequence converging weakly to $u\in H$.
\end{lem}

\medskip


\end{document}